\documentclass[11pt, twoside]{article}
\usepackage{amsfonts,amsmath,amssymb,amsthm}
\usepackage[mathscr]{euscript}
\usepackage{indentfirst}
\usepackage{enumerate}
\usepackage{color}
\usepackage{lipsum}
\usepackage{paralist}
\usepackage{graphics}
\usepackage{graphicx}
\usepackage{subfigure}
\usepackage{tikz}
\usepackage{epstopdf}
\usepackage{float}\usepackage{mathrsfs,caption}
\usepackage{color}
\usepackage{accents}

\usepackage{paralist}
\newtheorem{theo}{Theorem}[section]

\newtheorem{lem}{Lemma}[section]
\newtheorem{remark}{Remark}[section]
\newtheorem{col}{Corollary}[section]
\newtheorem{defi}{Definition}[section]
\numberwithin{equation}{section}

\newcommand{\lbl}[1]{\label{#1}}
\allowdisplaybreaks[4]

\newcommand\bes{\begin{eqnarray}} \newcommand\ees{\end{eqnarray}}
\newcommand{\bess}{\begin{eqnarray*}}
\newcommand{\eess}{\end{eqnarray*}}
\newcommand{\bbbb}{\left\{\begin{aligned}}
\newcommand{\nnnn}{\end{aligned}\right.}
\newcommand{\bea}{\begin{align*}}
\newcommand{\eea}{\end{align*}}

\newcommand\ep{\varepsilon}
\newcommand\kk{\left}
\newcommand\rr{\right}
\newcommand\dd{\displaystyle}

\newcommand\dx{{\rm d}x}
\newcommand\dy{{\rm d}y}
\newcommand\dt{{\rm d}t}
\newcommand\lm{\lambda}
\newcommand\nm{\nonumber}
\newcommand\yy{\infty}
\newcommand\qq{\eqref}

\newcommand\ol{\overline}
\newcommand\ud{\underline}
\newcommand\pp{\partial}

\newcommand\oo{\Omega}
\newcommand\boo{\overline\Omega}
\newcommand\bqq{\overline Q_T}

\newcommand\sss{\mathbb{S}}
\newcommand\vvv{\vskip 4pt}
\newcommand\www{\vspace{-2mm}}
\newcommand\zzz{\vspace{-1mm}}

\newcommand\ccc{\color{blue}}
\newcommand\rrr{\color{red}}

\def\theequation{\arabic{section}.\arabic{equation}}

\begin{document}\thispagestyle{empty}
\setlength{\abovedisplayskip}{7pt}
\setlength{\belowdisplayskip}{6pt}

\begin{center}{\Large\bf Generalized principal eigenvalue of time-periodic cooperative}\\[2mm]
{\Large\bf nonlocal dispersal systems and applications}\\[2mm]
Mingxin Wang\footnote{{\sl E-mail}: mxwang@sxu.edu.cn. Mingxin Wang was supported by National Natural Science Foundation of China Grant 12171120.}\\
{\small School of Mathematics and Statistics, Shanxi University, Taiyuan 030006, China}\\
Lei Zhang\footnote{The corresponding author, {\sl E-mail}: zhanglei890512@gmail.com. Lei Zhang was supported by the National Natural Science Foundation of China (12471168, 12171119) and the Fundamental Research Funds for the Central Universities (GK202304029, GK202306003, GK202402004). }\\
{\small School of Mathematics and Statistics, Shaanxi Normal University, Xi'an 710119, China}
\end{center}

\begin{quote}
\noindent{\bf Abstract.} It is well known that, in the study of the dynamical properties of nonlinear reaction-diffusion systems, the sign of the principal eigenvalue of the linearized system plays an important role. However, for the nonlocal dispersal systems, due to the lack of compactness, the essential spectrum appear, and the principal eigenvalue may not exist. In this paper, by constructing monotonic upper and lower control systems, we obtain the generalized principal eigenvalue of the cooperative irreducible system and demonstrate that this generalized principal eigenvalue plays the same role as the usual principal eigenvalue.


\end{quote}

\pagestyle{myheadings}
\section{Introduction}{\setlength\arraycolsep{2pt}
\markboth{\rm$~$ \hfill Cooperative nonlocal dispersal systems with time period\hfill $~$}{\rm$~$ \hfill M.X. Wang \& L. Zhang\hfill $~$}

In the study of differential equations, eigenvalues play a crucial role, particularly the principal eigenvalue, which can determine the dynamic behaviors of many monotone systems. In classical reaction-diffusion equations and cooperative systems, the theory of the principal eigenvalue can be established through the Krein-Rutman theorem. However, for nonlocal dispersal equations and systems, the lack of compactness can lead to the emergence of essential spectral points in the corresponding linear systems, and the principal eigenvalue does not always exist. In this paper, we define the generalized principal eigenvalue for cooperative time-periodic nonlocal dispersal systems and demonstrate that it serves the same purpose as the usual principal eigenvalue.

{\bf Notations}: For any given $u,v\in\mathbb{R}^m$. We say $u\ge v$ refers to $u_i\ge v_i$ for all $i\in\sss$; $u >v$ refers to $u_i \ge v_i$ for all $i\in\sss$ but $u\neq v$; and $u \gg v$ refers to $u_i>v_i$ for all $i\in\sss$. Let $\oo\subset\mathbb{R}^N$ be a bounded with smooth boundary, and $0<T<+\yy$ be a given constant and $m\ge 1$ be an integer. Denote
\bess
 &Q_T=\boo\times(0,T],\;\;\;\sss=\{1,\cdots,m\},\;\;\;\mathbb{R}_+^m
 =\big\{u\in\mathbb{R}^m:\,u_i\ge 0,\;\forall\,i\in\sss\big\},&\\
  &\mathbb{X}^m=\kk\{\phi:\,\phi_i\in C^{0,1}(\ol Q_T),\;\,\;\forall\,i\in\sss\rr\},\;\;\phi=(\phi_1,\cdots,\phi_m),&\\
   &\mathbb{X}^m_T=\kk\{\phi:\,\phi_i\in C^{0,1}(\ol Q_T),\;\,\phi_i(x,0)=\phi_i(x,T)\;\; {\rm in }\;\;\boo,\;\forall\,i\in\sss\rr\},&\\
    &\mathbb{P}^m=\mathrm{Int}((\mathbb{X}^m_T)_+)=\{\phi \in \mathbb{X}^m_T: \phi \gg 0 \text{ in } \bqq\},&\\
   &\mathbb{Y}^m_T=\{\phi:\,\phi_i\in C^{1}( [0,T]),\;\,
 \phi_i(0)=\phi_i(T),\,\, i\in\mathbb{S}\},&\\
 &\mathbb{Z}^m=\big\{u\in[L^\yy(Q_T)]^m:\,u_i(x,\cdot)\in C^1([0,T]),\,\,\forall\,x\in\boo,\,i\in\sss\big\}.&
 \eess

Consider the initial value problem of nonlocal dispersal cooperative system
\bes\left\{\begin{array}{lll}
 u_{it}=d_i\dd\int_\oo J_i(x,y)u_i(y,t)\dy-d^*_i(x)u_i+f_i(x,t,u),\; &x\in\ol\oo,\;&t>0,\\[3mm]
 	u_i(x,0)=u_{i0}(x)\ge0,\,\not\equiv 0, &x\in\boo,\\
i=1,\cdots,m,
 \end{array}\rr.\lbl{z1.3}\ees
where, for each $i\in\sss$, $d_i>0$ is a constant, $u_{i0}\in C(\boo)$, $J_i(x,y)$ satisfies
\begin{itemize}
	\item[{\bf(J)}] $J_i(x,y)\ge 0$ is a continuous function of $(x,y)\in\mathbb{R}^{2N}$, and
 \[J_i(x,x)>0,\;\;\int_{\mathbb{R}^N}J_i(x,y)\dy=1\;\;\;\mbox{for\; all}\;\;x\in\mathbb{R}^N;\]
\end{itemize}
and $d^*_i(x)$ is defined by the following manner:
\begin{itemize}
	\item[{\bf(D)}]
Either $d^*_i(x)=d_i$ (corresponding to Dirichlet boundary condition) or $d^*_i(x)=d_ij_i(x)$ (corresponding to Neumann boundary condition), where
 \[j_i(x)=\int_{\Omega} J_i(y,x){\rm d}y,\;\;\,x\in\ol\oo.\]
\end{itemize}

Set $b_{ik}(x,t)=\pp_{u_k}f_i(x,t,0)$ and
  \bess
   L(x,t)=(\ell_{ik}(x,t))_{m\times m},
   \eess
where $\ell_{ik}(x,t)=b_{ik}(x,t)$ for $i\not=k$, and $\ell_{ii}(x,t)=b_{ii}(x,t)-d_i^*(x)$. It is well known that the dynamical properties of \qq{z1.3} depends strongly on the eigenvalue of the operator ${\mathscr L}$ defined by
 \bes\lbl{1.2}\begin{cases}
  	{\mathscr L}[\phi]=({\mathscr L}_1[\phi],\cdots,{\mathscr L}_m[\phi]),\;\;\;\phi\in \mathbb{X}^m_T,\\[1mm]
 	{\mathscr L}_i[\phi]=d_i\dd\int_\oo J_i(x,y)\phi_i(y,t)\dy+\sum_{k=1}^m \ell_{ik}(x,t)\phi_k(x,t)-\phi_{it}(x,t).
 \end{cases}\ees

We always assume that $\ell_{ik}\in C(\ol Q_T)$ is $T$-periodic in time $t$ for all $i,k\in\sss$. We recall that the matrix $A$ is cooperative if $a_{ik}\ge 0$ for all $i\not=k$, and that $A$ is irreducible if the index set $\sss$ cannot be split up in two disjoint nonempty sets ${\mathbb I}$ and ${\mathbb K}$ such that $a_{ik}=0$ for $i\in{\mathbb I}, k\in{\mathbb K}$. And sometimes we need the following assumptions:\vspace{-1.2mm}
\begin{itemize}
	\item[{\bf (L1)}] $L(x,t)=(\ell_{ik}(x,t))_{m\times m}$ is a cooperative matrix for all $(x,t)\in\ol Q_T$;\vvv\vvv
	\item[{\bf (L2)}] $(\bar{\ell}_{ik})_{m \times m}$ is irreducible where $\bar{\ell}_{ik}=\frac1{|\Omega|T}\int_{0}^{T}\int_{\Omega} \ell_{ik} (x,t) \mathrm{d} x \mathrm{d} t$ for all $i,k \in \sss$.
\end{itemize}\vspace{-1.2mm}

We here remark that if $L(x_0,t_0)$ is irreducible at some $(x_0,t_0) \in \bqq$, then {\bf (L2)} holds obviously. Before stating our main result, we provide a brief overview of the study of the principal eigenvalue and the generalized principal eigenvalue. Shen and her collaborators established the principal eigenvalue theory for scalar autonomous and periodic equations and systems (cf. \cite{SZ10JDE, SX15, BS17PAMS}). Coville \cite{Cov10} applied the generalized Krein-Rutman theorem to present several sufficient conditions for the existence of the principal eigenvalue for scalar equations. Li, Coville, and Wang \cite{LCWdcds17} investigated the principal eigenvalue based on the Collatz-Wielandt characterization (maximum-minimum property). There are also many related and subsequent studies, including those involving time-delay \cite{LZZ2019JDE}, age structure \cite{KangR22}, partially degenerate cases \cite{Zhang24}, cross-diffusion \cite{SWZ23, SLLW2023JMPA} as well as more references \cite{BZ07,FLRX24}. Berestycki, Coville and Vo \cite{B-JFA16} defined the generalized principal eigenvalue for a scalar nonlocal dispersal operator using the Collatz-Wielandt characterization. The authors of this paper \cite{WZhang25} proved that the generalized principal eigenvalue of cooperative autonomous nonlocal dispersal systems can be approximated by the principal eigenvalues of monotonic upper and lower control systems. We also demonstrated that the generalized principal eigenvalue serves as a suitable alternative to the principal eigenvalue, regardless of whether the latter exists.

Due to seasonal changes in the environment, including temperature and rainfall, a natural question arises: do these results still hold for the time-periodic case? Specifically, how can the generalized principal eigenvalue be defined in time-periodic environments? Are there suitable perturbations with principal eigenvalues to approximate the generalized principal eigenvalue? Moreover, can the generalized principal eigenvalue play the same role as the principal eigenvalue does in the autonomous case? In this paper, we will provide confirmed answers.

Under the above assumptions {\bf(L1)} and {\bf(L2)}, for $x\in\boo$, we define an operator ${\mathscr P}_x$ by
  \bes
  [{\mathscr P}_{x}\phi](t)=-\frac{{\rm d}\phi(t)}{{\rm d}t}+L(x,t)\phi(t), \;\; \phi \in \mathbb{Y}^m_T.\lbl{z1.6}\ees
Then $\theta(x)=s({\mathscr P}_x)$ is an eigenvalue of ${\mathscr P}_x$, where $s({\mathscr P}_x)$ is the spectral bound of ${\mathscr P}_x$, which is the maximum of the real parts of all eigenvalues of ${\mathscr P}_x$ (see, e.g., \cite[Proposition 2.4]{LZZ2019JDE}). Under the assumption that $L(x,t)$ is irreducible for all $(x,t)\in\bqq$, one of the following conditions (see, e.g., \cite[Theorem 2.2]{BS17PAMS}) is sufficient  to guarantee the existence of $\lm_p({\mathscr L})$: \vspace{-2mm}
 \begin{enumerate}
\item[{\bf(H1)}] There exists an open subset $\Omega_0 \subset \Omega$ such that $[\max_{\ol\oo}\theta(x)-\theta(x)]^{-1}\not\in L^1(\Omega_0)$ ;\vskip 5pt
\item[{\bf(H2)}] $d_i$ are suitable large;\vskip 3pt
\item[{\bf(H3)}] $B(x,t)=B(t)$, i.e., $b_{ik}(x,t)=b_{ik}(t)$ for all $ i,k\in\sss$ ;\vskip 3pt
\item[{\bf(H4)}] $J_i(x,y)=\frac{1}{\delta^N}{J}^*_i(\frac{x-y}{\delta})$ with  ${\rm supp}({J}^*_i)=\{x\in\mathbb{R}^N: |x|<1\}$ and $0<\delta\ll 1$ .
\end{enumerate}\vspace{-2mm}

We also remark that conditions {\bf(H1)}--{\bf(H3)} can also guarantee the existence of the principal eigenvalue when the irreducibility condition is weakened to require that $L(x_0,t_0)$ is irreducible at some point $(x_0,t_0) \in \bqq$ shown in \cite{FLRX24}.

The following theorem is the first main result of this paper, which presents the definition, approximation of the generalized principal eigenvalue.\vspace{-2mm}

\begin{theo}\lbl{t2.1}\, Assume that {\bf(L1)}--{\bf(L2)} hold. Then there exist $\ud L^\ep(x,t)=(\ud\ell^\ep_{ik}(x,t))_{m\times m}$ and $\ol L^\ep(x,t)=(\bar\ell^\ep_{ik}(x,t))_{m\times m}$, with $\ep>0$, satisfying\vspace{-1mm}
 \begin{enumerate}
\item[$\bullet$] $\ud\ell^\ep_{ik}, \bar\ell^\ep_{ik}\in C(\ol Q_T)$, and $\ud\ell^\ep_{ik}$ and $\bar\ell^\ep_{ik}$ are decreasing and increasing in $\ep$, respectively, for all $i,k\in\sss$; and
		\bess
		\ud\ell^\ep_{ik}\le \ell_{ik}\le\bar\ell^\ep_{ik}\;\;\;\mbox{in}\;\;\ol Q_T,\;\;\;
		\lim_{\ep\to 0^+}\ud\ell^\ep_{ik}=\lim_{\ep\to 0^+}\bar\ell^\ep_{ik}=\ell_{ik} \;\;\;\mbox{in}\;\;C(\ol Q_T),\;\;i,k\in\sss,
		\eess
	\end{enumerate}\vspace{-1mm}
such that\vspace{-1mm}
 \begin{enumerate}[$(1)$]
 \item operators $\ud{\mathscr L}^\ep$ and $\ol{\mathscr L}^\ep$
have principal eigenvalues $\lm_p(\ud{\mathscr L}^\ep)$ and $\lm_p(\ol{\mathscr L}^\ep)$, respectively, where $\ud{\mathscr L}^\ep$ and $\ol{\mathscr L}^\ep$ are defined by \qq{1.2} with $\ell_{ik}$ replaced by $\ud\ell^\ep_{ik}$ and $\bar\ell^\ep_{ik}$, respectively.\vskip 4pt
		
\item $\lm_p(\ud{\mathscr L}^\ep)\le\lm_p(\ol{\mathscr L}^\ep)$, and $\lm_p(\ud{\mathscr L}^\ep)$ and $\lm_p(\ol{\mathscr L}^\ep)$ are strictly decreasing and increasing in $\ep$, respectively, and
\[\lim_{\ep\to 0^+}\lm_p(\ud{\mathscr L}^\ep)=\lim_{\ep\to 0^+}\lm_p(\ol{\mathscr L}^\ep)=:\lm({\mathscr L}),\]
 \item  $\lm({\mathscr L})$ belongs to the spectral set of ${\mathscr L}$, and has the characterization:
	\bes
\lm({\mathscr L})=\sup_{\phi\in{\mathbb{P}^m}}\sup_{\lambda\in\mathbb{R}} \{{\mathscr L}\phi\ge\lambda \phi\}= \inf_{\phi\in{\mathbb{P}^m}}\inf_{\lambda\in\mathbb{R}} \{{\mathscr L} \phi\le\lambda \phi\}=s({\mathscr L}).
		\lbl{2b.1}\ees
\item $\lm({\mathscr L})$ is continuous with respect to $L$.
	\end{enumerate}\vspace{-2mm}
\end{theo}\www

Notice that the solution map is eventually strongly positive under the assumptions {\bf(L1)} and  {\bf(L2)}. We first prove that the spectral bound of ${\mathscr L}$ is an eigenvalue corresponding to a strictly positive eigenvector under the assumption {\bf(L1)} and condition {\bf(H1)}, and hence the eigenvector is strongly positive when the assumption {\bf(L2)} holds. This sufficient condition is less stringent than those presented in \cite{BS17PAMS,FLRX24}. The spectral bound of ${\mathscr L}$ is continuous with respect to the matrix-valued function $L(x,t)$ whether the principal eigenvalue exists or not.
We then construct the upper and lower control matrix-valued functions $\ol L^\ep(x,t)$ and $\ud L^\ep(x,t)$ for $L(x,t)$, and these two functions satisfy the assumptions {\bf(L1)}--{\bf(L2)} and condition {\bf(H1)}. Consequently, the corresponding operators $\ol{\mathscr L}^\ep$ and $\ud{\mathscr L}^\ep$ (defined by \qq{1.2} with $L(x,t)$ replaced by $\ol L^\ep(x,t)$ and $\ud L^\ep(x,t)$, respectively) have principal eigenvalues $\lm_p(\ol{\mathscr L}^\ep)$ and $\lm_p(\ud{\mathscr L}^\ep)$, respectively.
Moreover, $\lm_p(\ol{\mathscr L}^\ep)$ and $\lm_p(\ud{\mathscr L}^\ep)$ have the same limit, and this limit satisfies \eqref{2b.1}, which is a variation of Collatz-Wielandt characterization. In this paper, this limit is referred to as the generalized principal eigenvalue of the operator ${\mathscr L}$, which is indeed the spectral bound of ${\mathscr L}$ and also represents the exponential growth bound of the associated linear system.

Next we show this generalized principal eigenvalue play the same role as the principal eigenvalue  by analyzing the dynamics of \eqref{z1.3}.
The basic assumptions on $f=(f_1,\cdots, f_m)$: \vspace{-1.2mm}
\begin{itemize}
	\item[{\bf(F1)}] for each $i\in\sss$, the function $f_i(x,t,u)$ is $T$-periodic in time $t$, $f_i(x,t, 0)=0$ in $\bqq$, $f_i(x,t,u)\in C^{0,0,1}(\bqq\times\mathbb{R}_+^m)$; and $f(x,t,u)$ is cooperative in $u\ge 0$, i.e., $\partial_{u_k}f_i(x,t,u)\ge 0$ for all $(x,t)\in\bqq$, $u\ge 0$ and $k\not=i$;
	\item[{\bf(F2)}] there exists $(x_0,t_0)\in\ol Q_T$ such that
	\bess
	\kk(\partial_{u_k}f_i(x_0,t_0,u)\rr)_{m\times m}\;\;\mbox{is irreducible for all }\;u \ge 0.\eess
	\item[{\bf(F3)}]  $f(x,t,u)$ is strictly subhomogeneous with respect to $u\gg 0$, i.e.,
 \bess
 f(x,t,\delta u)>\rho f(x,t,u),\;\;\forall\,\rho\in (0,1),\; (x,t)\in\bqq,\; u\gg 0\eess
\end{itemize}\vspace{-1mm}

Now we state the second main result of this paper.

\begin{theo}\lbl{th3.6} Assume that {\bf(F1)}--{\bf(F3)} hold. Let $u(x,t;u_0)$ be the unique solution of \qq{z1.3}. Then the following statements are valid:\vspace{-2mm}
	\begin{enumerate}[$(1)$]
		\item If $\lm(\mathscr{L})>0$ and there exists $\ol U\in [C(\ol Q_T)]^m$ with $\ol U\gg 0$ such that
		\bes\begin{cases}
			\ol U_{it}\ge d_i\dd\int_\oo J_i(x,y)\ol U_i(y,t)\dy-d^*_i(x)\ol U_i
			+f_i(x,t,\ol U),\;&(x,t)\in\ol Q_T,\\[3mm]
			\ol U_i(x,0)\ge\ol U_i(x,T),\; &x\in\boo,\\
			i=1,\cdots,m,
		\end{cases}\lbl{3.5}\ees
		then \eqref{3.1} has a unique positive solution $U\in[C(\ol Q_T)]^m$ and $U\le\ol U$ in $\ol Q_T$. Moreover,
		\bes
		\lim\limits_{n\rightarrow +\infty} u(x,t+nT;u_0)= U(x,t) \;\;\text{ uniformly in } \; \ol Q_T.\lbl{3c.15}\ees
		\item If $\lm(\mathscr{L})< 0$, then the system \eqref{3.1} has no positive solution in $[C(\ol Q_T)]^m$. Moreover,  there exist $\sigma>0$ and $C>0$ such that
		\bes
		u(x,t;u_0)\le C{\rm e}^{-\sigma t},\;\;\forall\,x\in\boo,\; t\ge 0.
		\lbl{3.7}\ees
		This shows that $u(x,t;u_0)$ converges exponentially to zero.
		\item If $\lm(\mathscr{L})=0$, and $f$ is strongly subhomogeneous, that is, $f(x,t,\rho u)\gg\rho f(x,t,u)$ for all $(x,t)\in\ol Q_T$, $u\gg 0$ and $\rho\in(0,1)$, then system \eqref{3.1} has no positive solution in $[C(\ol Q_T)]^m$. If, in addition, there exists $\ol U\in [C(\ol Q_T)]^m$ with $\ol U\gg 0$ such that \eqref{3.5} holds, then
		\bes
		\lim_{n\rightarrow +\infty} u(x,t+nT;u_0)=0\;\;\text{ uniformly in } \; \ol Q_T.\lbl{3.8}\ees
	\end{enumerate}\vspace{-2mm}
\end{theo}

The main challenge to study \eqref{z1.3} still arise from the lack of compactness for the solution mapping, which results in the limit of the iteration sequence from upper and lower solutions exhibiting only semi-continuity. To overcome this difficulty, we first investigate the maximum principle and uniqueness of time-periodic positive solutions, which are differentiable with respect to the time variable, in the $L^{\infty}$ sense. Unlike the autonomous case, proving the maximum principle here necessitates simultaneous consideration of both time and space variables. To effectively deal with time variable, we require that the positive periodic solution is differentiable with respect to time variable.

In the case of $\lambda(\mathscr{L}) > 0$, we utilize the lower control matrix-valued function to construct a lower solution, and then construct the continuous positive time-periodic solution by the upper and lower solutions method.
We remark that the limits of the iteration sequence from upper and lower solutions are mild solutions in the $L^{\infty}$ sense, and thus differentiable with respect to the time variable. Moreover, these two limits are equal and upper and lower continuous, respectively, implying that the limit is continuous, which is the unique positive time-periodic solution.  Therefore, the solution from upper and lower solutions converges uniformly to the unique positive continuous time-periodic solution due to the Dini Theorem.  In the case of $\lambda(\mathscr{L}) \le 0$, the non-existence of time-periodic continuous positive solutions can be established using the Collatz-Wielandt characterization. When $\lambda(\mathscr{L}) < 0$, exponential decay can be demonstrated by constructing a suitable upper solution using the upper control matrix-valued function. For the critical case (i.e. $\lambda(\mathscr{L}) = 0$), in order to analyze the global dynamics, we perturb the system upwards and use the conclusion of the case $\lambda(\mathscr{L}) > 0$.

The organization of this paper is as follows. In Section 2, we shall prove the existence of the principal eigenvalue and the continuity of the spectral bound of the nonlocal dispersal cooperative systems. In Section 3, we present the definition and approximation of the generalized principal eigenvalue. In Section 4, we study the threshold dynamics of \qq{z1.3}. The existence, uniqueness and stabilities of positive equilibrium solutions are obtained. The results show that this generalized principal eigenvalue plays the same role as the usual principal eigenvalue. In Section 5, we use the abstract results obtained in Sections 3 and 4 to investigate a West Nile virus model. The section 6 is a brief discussion.

\section{Preliminaries------Existence of the principal eigenvalue and continuity of the spectral bound}

Let $({\mathscr X},{\mathscr X}_+)$ be an ordered Banach space with $\rm{Int}\,({\mathscr X}_+) \neq \emptyset$. We use $\ge$\, ($>$ and $\gg$) to represent the (strict and strong) order relation induced by the cone ${\mathscr X}_+$. For convenience, we say an element is strictly positive if it is in the cone ${\mathscr X}_+$ but not zero, and is strongly positive if it is in $\rm{Int}({\mathscr X}_+)$. For a bounded linear operator ${\mathscr A}$ on ${\mathscr X}$, ${\mathscr A}$ is said to be positive if ${\mathscr A} {\mathscr X}_+\subset {\mathscr X}_+$, strictly positive if ${\mathscr A} {\mathscr X}_+\setminus \{0\}\subset{\mathscr X}_+\setminus\{0\}$ and strongly positive if ${\mathscr A}{\mathscr X}_+\setminus\{0\}\subset\rm{Int}({\mathscr X}_+)$. The spectral bound (or principal spectrum point) of a resolvent positive operator ${\mathscr A}$ is defined by
 \[s({\mathscr A}):=\sup\{\mathrm{Re}\lambda: \lambda\in\sigma({\mathscr A})\}.\]
	
The spectral bound $s({\mathscr A})$ is called the weak principal  (resp. principal) eigenvalue of ${\mathscr A}$ if $s({\mathscr A})$ is an eigenvalue of ${\mathscr A}$ corresponding to a strictly (resp. strongly) positive eigenfunction.

Let $v(t,s,x;v_0)$ be the solution of
 \bes\label{equ:JG}\begin{cases}
	v_{it}= d_i\dd\int_\oo J_i(x,y)v(y,t) \mathrm{d} y+
	\sum_{k=1}^m \ell_{ik}(x,t)v_k, & x\in\boo,\; t>0,\\
	v_i(x,s)=v_{i0}(x)\ge0,\,\not\equiv 0, &x\in\boo,\\
	 i=1,\cdots, m,
\end{cases}
\ees
with initial data $v(s,s,\cdot;v_0)=v_0$. Define $\Phi(t,s)v_0=v(t,s,\cdot;v_0)$, which is a $T$-periodic evolution family on $C(\boo,\mathbb{R}^m)$.
Let $\Psi(t,s)$ be the $T$-periodic  evolution family on $C(\boo,\mathbb{R}^m)$ of
  \bes\label{equ:G}
\begin{cases}
	v_{it}=\dd\sum_{k=1}^m \ell_{ik}(x,t)v_k, &  x\in\boo, \; t>0,\\
	v_i(x,s)=v_{i0}(x)\ge0,\,\not\equiv 0, &x\in\boo,\\
	 i=1,\cdots, m.
\end{cases}
\ees
For any $x\in\boo$, let $\Gamma_x(t,s)$ be the $T$-periodic evolution family of \eqref{equ:G} on $\mathbb{R}^m$. Then we have the following conclusions.\zzz

\begin{lem}\label{lem:A1} Assume {\bf(L1)} holds. Then the following statements are valid:\vspace{-2mm}
\begin{enumerate}[$(1)$]
 \item $\sigma_e(\Phi(T,0))=\sigma_e(\Psi(T,0))=\cup_{x\in\boo}
     \sigma(\Gamma_x(T,0))$.\vvv
\item $r_e(\Phi(T,0))=r_e(\Psi(T,0))=\max_{x\in\boo}r(\Gamma_x(T,0))$.\vvv \item The operator ${\mathscr L}$ is resolvent positive and $s({\mathscr L})=\frac{\ln r(\Phi(T,0))}{T}$, i.e., $r({\rm e}^{-s({\mathscr L})T}\Phi(T,0))=1$.
\end{enumerate}
\end{lem}
\begin{proof} For any $t \in [0,T]$, define an operator
	\bess
 {\cal J}=({\cal J}_1,\cdots,{\cal J}_m)\eess
 with
  \[{\cal J}_i[v](x)=d_i\dd\int_\oo J_i(x,y)v_i(y)\dy,\;\;v\in C(\ol \Omega,\mathbb{R}^m). \]
By the constant-variation formula, we have
  \[\Phi(T,0)v_0=\Psi(T,0)v_0+ \int_{0}^{T} \Psi(T,s) {\cal J} \Phi(s,0)v_0 \mathrm{d} s.\]
Since ${\cal J}_i$ is compact, it is not hard to verify that $\int_{0}^{T} \Psi(T,s) {\cal J} \Phi(s,0) \mathrm{d} s$ is compact. Thanks to \cite[Theorem 7.27]{S1971book} and \cite[Proposition 2.7]{LZZ2019JDE}, $\sigma_e(\Phi(T,0))=\sigma(\Psi(T,0))=\cup_{x\in\boo} \sigma(\Gamma_x(T,0))$. Therefore, the conclusions (1) and (2) hold.
	
The conclusion (3) can be derived by \cite[Proposition 2.10]{FLRX24}.
\end{proof}\www

We remark that the spectral bound of ${\mathscr L}$ is indeed the exponential growth bound of $\Phi$ by Lemma \ref{lem:A1}(3).

The following lemma can be used to estimate spectral bound  $s({\mathscr L})$ of ${\mathscr L}$.\www

\begin{lem}\lbl{le2.2} Assume {\bf (L1)} holds. Then the following statements are valid:\vspace{-2mm}
\begin{enumerate}[$(1)$]
\item If there exist a constant $\beta$ and $\phi\in\mathbb{X}^m$ with $\phi\ge 0$ in $\ol Q_T$ and $\phi(\cdot,0)\not\equiv 0$, such that
 \bess
{\mathscr L}[\phi]\ge \beta\phi,\;\;\;\mbox{and}\;\;\phi(x,T) \, \ge\,\phi(x,0)\, \;\;\mbox{in}\;\;\boo,\eess
then $s({\mathscr L})\ge \beta$.\vvv
 \item If there exist a constant $\beta$ and $\phi\in\mathbb{X}^m$ with $\phi \gg  0$ in $\ol Q_T$, such that
\bess
{\mathscr L}[\phi]\le\beta\phi,\;\;\;\mbox{and}\;\;\phi(x,T)\le\phi(x,0)\, \;\;\mbox{in}\;\;\boo,\eess
then $s({\mathscr L})\le\beta$.
\end{enumerate}
\end{lem}

\begin{proof} For any given $u_0\in C(\boo,\mathbb{R}^m)$, let $u(x,t; u_0)$ be the unique solution of
 \bess\left\{\begin{array}{lll}
u_{it}=d_i\dd\int_\oo J_i(x,y)u_i(y,t)\dy+\sum_{k=1}^n \ell_{ik}(x,t)u_k-\beta u_k,\; &(x,t)\in Q_T,\\[3mm]
u_i(x,0)=u_{i0}(x),\,\not\equiv 0, &x\in\boo,\\[1mm]
i=1,\cdots, m.
 \end{array}\rr.\eess
Define an mapping $\mathscr P:\,C(\boo,\mathbb{R}^m)\to  C(\boo,\mathbb{R}^m)$ by
	\[\mathscr P[u_0]=u(x,T; u_0).\]
	
(1) As $\phi\in \mathbb{X}^m$, $\phi\ge 0$ in $\ol Q_T$ and satisfies ${\mathscr L}[\phi]\ge\beta\phi$, the comparison principle gives $\phi(x,t)\le u(x,t;\phi(x,0))$ in $Q_T$. So
 \[\mathscr P[\phi(\cdot,0)](x)=u(x,T;\phi(x,0))\ge \phi(x,T)\ge\phi(x,0).\]
Then Gelfand's formula (see, e.g., \cite[Theorem VI.6]{ReS80}) implies that $r(\mathscr P)\ge 1$. Noticing that $\mathscr P={\rm e}^{-\beta T}\Phi(T, 0)$, we have ${\rm e}^{-\beta T}r(\Phi(T, 0))=r({\rm e}^{-\beta T}\Phi(T, 0))\ge 1$. By Lemma \ref{lem:A1}(3), ${\rm e}^{-s({\mathscr L})T}r(\Phi(T, 0))=1$. Therefore, $\beta\le s({\mathscr L})$.
	
(2) It is easy to verify that
 \[\mathscr P[\phi(\cdot,0)](x)=u(x,T;\phi(x,0))\le \phi(x,T)\le\phi(x,0).\]
Notice that $C(\boo,\mathbb{R}^m)$ is Banach space with the cone $C(\boo,\mathbb{R}_+^m)$, which is normal and solid. Thanks to \cite[Corollary 2.7]{WZZ2023SIAP}, we obtain that $r(\mathscr P)\le 1$. It then follows from ${\rm e}^{-s({\mathscr L})T}r(\Phi(T, 0))=1$ and ${\rm e}^{-\beta T}r(\Phi(T, 0))\le 1$ that $\beta \ge s({\mathscr L})$.\www
\end{proof}

\begin{lem}\label{lem:PE:exist}	Assume {\bf (L1)} holds. Then $s({\mathscr L})$ is the weak principal eigenvalue of $\mathscr L$ if one of the following statements valids:\vspace{-2mm}
\begin{enumerate}[$(1)$]
\item $r(\Phi(T,0))>r_e(\Phi(T,0))$.\vvv
\item $s({\mathscr L}) >\theta_M=\max_{\boo}\theta(x)=\frac{\ln r_e(\Phi(T,0))}{T}$.\vvv
\item  Define
 \[\widetilde{\mathbb{X}}^m_T=\{\phi:\,\phi_i\in C(\ol Q_T),\;\,\phi_i(x,0)
 =\phi_i(x,T)\;\; {\rm in }\;\;\boo,\,\,\forall\,i\in\mathbb{S}\}.\]
There exists $\beta>\theta_M$ and  $\psi\in\widetilde{\mathbb{X}}^m_T$ with $\psi\ge 0$ in $\ol Q_T$ and $\psi\not\equiv 0$ such that
 \[{\cal J} (\beta I -\cal F)^{-1}\psi\ge\psi,\]
where the operators ${\cal J}$ and ${\cal F}$ are defined by
 \[{\cal J}=({\cal J}_1,\cdots,{\cal J}_m),\;\;\text{with }\;{\cal J}_i[v](x,t)=d_i\dd\int_\oo J_i(x,y)v_i(y,t)\dy,\;\;v\in\widetilde{\mathbb{X}}^m_T,\]
and
  \[{\cal F}=({\cal F}_1,\cdots,{\cal F}_m),\;\;\text{with }\;{\cal F}_i[\phi]=\sum_{k=1}^n\ell_{ik}(x,t)\phi_k(x,t)
  -\phi_{it}(x,t),\;\;\phi\in\widetilde{\mathbb{X}}^m_T\,\]
respectively.
	\end{enumerate}
\end{lem}

\begin{proof} (1) When $r(\Phi(T,0))>r_e(\Phi(T,0))$, the desired conclusion is a straightforward result of the generalized Krein-Rutman theorem (see, e.g., \cite[Corollary 2.2]{N1981FPT} and \cite[Lemma 2.4]{Zhang24}).\vvv
	
(2) Thanks to Lemma \ref{lem:A1}, $s({\mathscr L})= \frac{\ln r(\Phi(T,0))}{T}$. Then the desired conclusion can be derived by the conclusion (1).
	
(3) In view of $\beta>\theta_M$, we have that $r({\rm e}^{-\beta T}\Gamma_x(T,0))<1$ for all $x\in\boo$ and $(\beta I -{\cal F})^{-1}$ is well-defined. Let $\phi=(\beta I -{\cal F})^{-1}\psi$. Then \[\phi (x,0)= (I-{\rm e}^{-\beta T}\Gamma_x(T,0) )^{-1}\int_{0}^{T} {\rm e}^{-\beta (T-s)}\Gamma_x(T,s) \psi (x,s) \mathrm{d} s,\]
and $\phi(x,t)$ is the unique $T$-periodic solution of
 \[\phi_{it}=\sum_{k=1}^{m} \ell_{ik}(x,t)\phi_k-\beta \phi_k +\psi_i(x,t),\; (x,t)\in Q_T.\]
This implies that
 \[\phi(x,t)= {\rm e}^{-\beta t}\Gamma_x(t,0)\phi(x,0)+\int_{0}^{t}{\rm e}^{-\beta (t-s)}\Gamma_x(t,s)\psi(x,s)\mathrm{d} s, \; (x,t)\in \bqq.\]	
So $\psi=\beta \phi -{\cal F}\phi$ and $\phi(\cdot,0)\neq 0$, and hence, ${\mathscr L}\phi ={\cal J}\phi +{\cal F}\phi\ge\beta\phi$. According to Lemma  \ref{le2.2}(1), $s({\mathscr L})\ge\beta>\theta_M$. This completes the proof.
\end{proof}\www\zzz

\begin{lem}\lbl{lA.3} Assume that {\bf (L1)} holds and there exists an open set $\Omega_0\subset\Omega$ such that $[\theta_M-\theta(x)]^{-1}\not \in L^1(\Omega_0)$. Then $s({\mathscr L})$ is the weak principal eigenvalue of ${\mathscr L}$.\www
\end{lem}

\begin{proof}	The desired conclusion can be derived by modifying the arguments to those in \cite[Proposition 3.4 and Lemma 4.1]{BS17PAMS}. For reader's convenience, we provide the details.

Let $\theta(x)$ be the weak principal eigenvalue of ${\mathscr P}_x$ and $w(x,t)$ be the corresponding eigenfunction which is strictly positive and continuous in $\bqq$, and $\max_{i\in\mathbb{S}}\max_{\bqq}w_i=1$.
Choose $\bar{x}\in\overline{\Omega}_0$ such that $\theta(\bar{x})=\theta_M=\max_{x\in\ol\oo} \theta(x)$. Define
 \[\tilde{\mathbb{S}}=\{i\in\mathbb{S}: w_i(\bar{x},t)>0\}.\]
By the continuity, there exists $\sigma>0$ such that $w_i(x,t)>0$ in $\Sigma_\sigma\times[0,T]$ for $i\in\tilde{\mathbb{S}}$, where $\Sigma_\sigma=\overline{B(\bar{x},\sigma)\cap\oo}$. Let
 \[c_1= \min_{i\in\tilde{\mathbb{S}}}\min_{{\Sigma_\sigma}\times [0,T]}w_i.\]
Besides, by the continuity of $J_i$, there exist $r_0>0$ and $c_0>0$ such that
 \[J_i(x,y)>c_0,\;\;\forall\;x,y\in\boo\;\;\mbox{with}
  \;\;|x-y|<r_0, \;t \in [0,T],\; i\in\mathbb{S}.\]
In view of $[\theta_M-\theta(x)]^{-1}\not\in L^1(\Omega_0)$, we can choose $0<\delta<\min\{\sigma,\,r_0/3\}$ and $\beta_0>\theta_M$ such that
  \[\int_{\Sigma_\delta}[\beta_0-\theta(x)]^{-1}\dx \ge\frac 2{c_0c_1}.\]
Take a function $p\in C(\boo)$ with $\max_{\boo} p=1$, and $p(x)=1$ in $\Sigma_\delta$, $p(x)=0$ in $\Sigma_{2\delta}$.

Write $\hat{w}(x,t)=w(x,t)p(x)$, $(x,t)\in\bqq$. We have
 \[[(\beta_0 I-{\cal F})^{-1}\hat{w}](x,t)=[\beta_0 -\theta(x)]^{-1}w(x,t)p(x),\;(x,t)\in\bqq.\]
It then follows that
  \[\begin{aligned}
\int_{\Omega} J_i(x,y)[\beta_0-\theta(y)]^{-1}\hat{w}_i(y,t)\dy
&\ge \int_{{\Sigma_\delta}} J_i(x,y)[\beta_0-\theta(y)]^{-1}\hat{w}_i(y,t)\dy\\
&=\int_{{\Sigma_\delta}} J_i(x,y)[\beta_0-\theta(y)]^{-1}w_i(y,t) \dy\\
&\ge 2\hat{w}_i(x,t)
	\end{aligned}\]
when $i\in\tilde{\mathbb{S}}$, $x\in{\Sigma_\delta}$ and $t\in[0,T]$; and
 \[\int_{\Omega} J_i(x,y)(\beta_0-\theta(y))^{-1}\hat{w}_i(y,t)\dy\ge 0\]
when either $i\in\tilde{\mathbb{S}}$, $x\in\boo \setminus {\Sigma_{2\delta}}$ and $t \in [0,T]$, or $i\in\sss\setminus\tilde{\mathbb{S}}$, $x\in\boo$ and $t \in [0,T]$. This indicates that ${\cal J}(\beta_0 I-{\cal F})^{-1}\hat{w}\ge 2\hat{w}$, and the desired conclusion can be derived by Lemma \ref{lem:PE:exist}(3).\zzz
\end{proof}

\begin{lem}\label{lem:SP} Assume that {\bf (L1)} and {\bf (L2)} hold. Then $\Phi(nT,0)$ is strongly positive for $n>m$, that is, $\Phi(nT,0) u_0 \gg 0$ for all { $u_0\in[C(\boo)]^m$,} $u_0 >0$ and $n > m$.\zzz
	\end{lem}
	
\begin{proof}  Let $u(\cdot, t)=\Phi(t,0)u_0$. Then $u\in[C(\boo\times[0,\yy))]^m$, $u\ge 0$, and for each $i\in\sss$,
 \[u_{it}\ge d_i\dd\int_\oo J_i(x,y)u_i(y,t)\dy-d^*_i(x)u_i(x,t)
 +\ell_{ii}(x)u_i(x,t),\;\; (x,t)\in \boo \times (0,+\infty).\]
Take advantage of Lemma \ref{le3.1} in Section \ref{Ses4}, it follows that for each $i\in\sss$,  either $u_i(x,t)\equiv 0$ in $\boo\times[0,+\infty)$, or there exists $t_0\in (0,+\infty)$ such that $u_i(x,t)>0$ in $\boo\times(t_0,+\infty)$.
		
Define
 \[{\cal I}_n=\{i\in\sss:\,u_i\equiv 0\;\;\mbox{in} \; \boo \times[(n-1)T,nT]\},\]
and ${\cal I}_n^c=\sss\setminus{\cal I}_n$. If ${\cal I}_1=\emptyset$, we can obtain that $u_i(x,t)>0$ in $\boo\times(T,+\infty)$, $i \in \sss$, which yields the desired conclusion.
		
Now we assume that ${\cal I}_1\neq\emptyset$. Since the initial data is nonzero, ${\cal I}_1 \neq \sss$. We then claim that $\#{\cal I}_{n+1}<\#{\cal I}_{n}$ whenever ${\cal I}_{n} \neq \emptyset$, where $\#{\cal I}_{n}$ is the amount of elements in the set ${\cal I}_{n}$. Suppose that $\#{\cal I}_{n_0+1}=\#{\cal I}_{n_0}$, that is, ${\cal I}_{n_0+1}={\cal I}_{n_0}$ for some $n_0$. For convenience, we still write $n =n_0$. Notice that ${\cal I}_{n+1}^c={\cal I}_{n}^c$, we have $u_k>0$ in $\boo\times[nT,+\yy)$ for all $k\in{\cal I}_{n+1}^c$.  An easy computation yields that, for each $i\in{\cal I}_{n+1}$, there holds:
 \bess
0=u_{it}(x,t)&\ge& d_i\dd\int_\oo J_i(x,y)u_i(y,t)\dy-d^*_i(x)u_i(x,t)
+\sum_{k=1}^m \ell_{ik}(x,t)u_k(x,t)\\[1mm]
&=&\sum_{k\in {\cal I}_{n+1}^c} \ell_{ik}(x,t)u_k(x,t),\;\;\forall\,(x,t)\in  \boo \times[nT,(n+1)T].
\eess
This implies that $\ell_{ik}\equiv 0$ in $\boo \times[nT,(n+1)T]$ for all $i\in {\cal I}_{n+1}$ and $k\in {\cal I}_{n+1}^c$, which contradicts the fact that $(\bar{\ell}_{ik})_{m\times m}$ is irreducible. We thus have that $\#{\cal I}_{n+1}<\#{\cal I}_{n}$ whenever ${\cal I}_{n} \neq \emptyset$, and the desired result holds true. \zzz\end{proof}

Making use of Lemmas \ref{lA.3} and \ref{lem:SP}, we have the following sufficient condition to ensure the existence of principal eigenvalue.\zzz

\begin{lem}\lbl{lem:2.1}  Assume {\bf (L1)} and {\bf (L2)} hold. If there exists an open set $\Omega_0\subset\Omega$ such that $[\theta_M-\theta(x)]^{-1}\not \in L^1(\Omega_0)$, then $s({\mathscr L})$ is the principal eigenvalue of ${\mathscr L}$, denoted by $\lm_p({\mathscr L})$.
\end{lem}\www

\begin{proof} For reader's convenience, we provide the outline of the proof. 	 In view of Lemma \ref{lA.3}, $s({\mathscr L})$ is an eigenvalue of ${\mathscr L}$ corresponding to a strictly postive eigenfunction $\phi$. Clearly, $\phi(x,0)$ is strictly positive, and otherwise $\phi(t,0)={\rm e}^{-\lm_p({\mathscr L})t}\Phi(t,0) \phi(\cdot,0) \equiv 0$, a contradiction. We further have $ \phi(t,0)={\rm e}^{-\lm_p({\mathscr L})t}\Phi(t,0) \phi(\cdot,0) \gg 0$ when $t > mT$ by Lemma \ref{lem:SP}.
\end{proof}\www

\begin{lem}	Assume {\bf (L1)} holds. If $s({\mathscr L})$ is the principal eigenvalue of ${\mathscr L }$, then $s({\mathscr L}) >\theta_M$.
\end{lem}

\begin{proof} Let $\phi$ be the principal eigenvector of ${\mathscr L}$ corresponding to $s({\mathscr L}) $, that is,
	\begin{equation}\label{equ:L:phi}
		d_i\dd\int_\oo J_i(x,y)\phi_i(y,t)\dy+\sum_{k=1}^m \ell_{ik}(x,t)\phi_k(x,t)-\phi_{it}(x,t)=s({\mathscr L}) \phi_i (x,t).
	\end{equation}
Choose $\bar{x} \in \boo$ such that $\theta(\bar{x})=\theta_M$. Let $\psi$ be the nonnegative nontrivial eigenfunction of
 \begin{equation}\label{equ:theta_m:psi}
 \sum_{k=1}^{m}\ell_{ki}(\bar{x},t) \psi_k (t)+\psi_{it} (t)= \theta_M \psi_i(t).
	\end{equation}
Multiplying \eqref{equ:L:phi} by $\psi_i$ and \eqref{equ:theta_m:psi} by $\phi_i$, respectively, then integrating the results over $[0,T]$, and making a difference at $\bar{x}$ and summing all $m$ equations together, we obtain that
	\[\sum_{i=1}^{m}d_i\int_{0}^{T}\psi_i(t)\dd\int_\oo J_i(\bar{x},y)\phi_i(y,t)\dy \mathrm{d}t
	=\sum_{i=1}^{m}(s({\mathscr L})-\theta_M)\int_{0}^{T}\phi_i (\bar{x},t) \psi_i(t) \mathrm{d}t.\]
This implies that $s({\mathscr L})>\theta_M$.
\end{proof}\www

Let $\tilde{d}_1$, $\cdots$, $\tilde{d}_m$ be nonnegative numbers. Choose a cooperative $T$-periodic matrix-valued function $\tilde{G}(x,t)=(\tilde{\ell}_{ik}(x,t))_{m \times m}$ and a series of $T$-periodic nonnegative function $\tilde{J}_i(x,y)$.
Define an operator
  \[\tilde{\mathscr L}=(\tilde{\mathscr L}_1,\cdots,\tilde{\mathscr L}_m)\]
with
 \[\tilde{\mathscr L}_i[\phi]=\tilde{d}_i\dd\int_\oo \tilde{J}_i(x,y,t)\phi_i(y,t)\dy+\sum_{k=1}^n\tilde{\ell}_{ik}(x,t)\phi_k(x,t)
 -\phi_{it}(x,t),\;\;\;\phi\in\mathbb{X}^m_T.\]
Then we have the following continuity result.

\begin{lem}\label{lem:conti}  Assume {\bf (L1)} holds. Then $s({\mathscr L})$ is continuous with respect to $L$, $J_i$ and $d_i$, that is, for any $\ep>0$, there exists $\delta>0$ such that $|s({\mathscr L})-s(\tilde{\mathscr L})|\le\ep$ whenever $\|\ell_{ik}-\tilde{\ell}_{ik}\|_{C(\bqq)}\le\delta$, $\|J_i-\tilde{J}_i\|_{C(\boo\times\boo)}\le\delta$ and $|d_i-\tilde{d}_i|\le\delta$ for all $i,k\in\mathbb{S}$.
\end{lem}

\begin{proof} In view of Lemma \ref{lem:A1}(3), $s({\mathscr L})=\frac{\ln r(\Phi(T,0))}{T}$. We only prove the continuity of $r(\Phi(T,0))$ with respect to $L$, and the other conclusions can be derived similarly. Define $\tilde{\Phi}$ the same as $\Phi$ with $L$ replaced by $\tilde{L}$. Let $\ep>0$ be any given number. We proceed our proof into three steps.
	
{\it Step 1.} Consider the case of $\ell_{ik}(x,t) \le\tilde{\ell}_{ik} (x,t)$ in $\bqq$ for all $i,k\in\mathbb{S}$. According to \cite[Theorem IV.3.1 and Remark IV.3.2]{Kato1976Book} and \cite[Theorem 1.1]{B1991AMB}, there exists $\delta_1>0$ such that
	\bes r(\Phi(T,0)) \le r (\tilde{\Phi}(T,0))\le\ep+r(\Phi(T,0)) \label{equ:conti:1}\ees
whenever $\ell_{ik}(x,t) \le\tilde{\ell}_{ik}(x,t)\le\ell_{ik}(x,t)+\delta_1$ for all $i,k\in\mathbb{S}$ and $(x,t)\in\bqq$.
	
{\it Step 2.} Consider the case of $\ell_{ik}(x,t)\ge\tilde{\ell}_{ik}(x,t)$ in $\bqq$ for all $i,k\in\mathbb{S}$. In the case of $r(\Phi(T,0)) >  r_e(\Phi(T,0))$, then $r(\Phi(T,0))$ is an isolated eigenvalue of $r(\Phi(T,0))$ due to the generalized Krein-Rutman theorem (see, e.g., \cite[Corollary 2.2]{N1981FPT} and \cite[Lemma 2.4]{Zhang24}). By the perturbation theory of isolated eigenvalue (see, e.g., \cite[Section IV.3.5]{Kato1976Book}), there exists $\delta_2>0$ such that
	\bes \vert r(\tilde{\Phi}(T,0))-r(\Phi(T,0)) \vert\le\ep
	\label{equ:conti:2}\ees
whenever $\|\tilde{\ell}_{ik}-\ell_{ik}\|_{C(\bqq)}\le\delta_2$ for all $i,k\in\mathbb{S}$.
	
In the case of $r(\Phi(T,0))=r_e(\Phi(T,0))$, by Lemma \ref{lem:A1}(2), there exists $\delta_3>0$ such that $\vert r_e(\tilde{\Phi}(T,0))-r_e(\Phi(T,0)) \vert\le\ep$ whenever $\|\tilde{\ell}_{ik}-\ell_{ik}\|_{C(\bqq)}\le\delta_3$  for all $i,k\in\mathbb{S}$. According to \cite[Theorem 1.1]{B1991AMB}, $r(\tilde{\Phi}(T,0))\le r(\Phi(T,0))$.	Notice that $r(\tilde{\Phi}(T,0)) \ge r_e(\tilde{\Phi}(T,0))$, we then have
	\bes 0 \ge r(\tilde{\Phi}(T,0))-r(\Phi(T,0)) \ge r_e(\tilde{\Phi}(T,0))-r_e(\Phi(T,0))  \ge -\ep
	\label{equ:conti:3}\ees
whenever $\|\tilde{\ell}_{ik}-\ell_{ik}\|_{C(\bqq)}\le\delta_3$ for all $i,k\in\mathbb{S}$.
	
{\it Step 3.} Prove the conclusion in general case. Let $\delta=\min(\delta_1,\delta_2,\delta_3)>0$. Assume that $\ell_{ik}(x,t)-\delta\le\tilde{\ell}_{ik}(x,t)\le \ell_{ik}(x,t)+\delta$ in $\bqq$ for all $i,k\in\mathbb{S}$. Define
 \[\overline{\ell}_{ik}(x,t)= \max(\ell_{ik}(x,t),\tilde{\ell}_{ik}(x,t)),\;\;\text{ and }\; \underline{\ell}_{ik}(x,t)= \min(\ell_{ik}(x,t),\tilde{\ell}_{ik}(x,t))\]
for $ i,k\in\mathbb{S}$, $(x,t)\in\bqq$. It is easy to verify that
 \[\ell_{ik}(x,t)\le\overline{\ell}_{ik}(x,t)\le \ell_{ik}(x,t)+\delta\;\;\;\mbox{in}\;\; \bqq,\;\forall\,i,k\in\mathbb{S},\]
and
 \[\ell_{ik}(x,t)-\delta\le\underline{\ell}_{ik}(x,t)\le \ell_{ik}(x,t)\;\;\;\mbox{in}\;\; \bqq,\; \forall\,i,k\in\mathbb{S}.\]
	
Define $\overline{\Phi}$ and $\underline{\Phi}$ the same as $\Phi$ with $L(x,t)$ replaced by $\overline{L}(x,t)=(\overline{\ell}_{ik}(x,t))_{m\times m}$ and $\underline{L}(x,t)=(\underline{\ell}_{ik}(x,t))_{m\times m}$, respectively. Then
 \bess
 r(\tilde{\Phi}(T,0))-\ep\le r(\overline{\Phi}(T,0))-\ep
	\le r(\Phi(T,0))\le r(\underline{\Phi}(T,0)) +\ep
	\le r(\tilde{\Phi}(T,0)) +\ep,\eess
which completes the proof.
\end{proof}

\section{Approximation and characterization of the generalized principal eigenvalue}\setcounter{equation}{0}
{\setlength\arraycolsep{2pt}

We are now in a position to prove our first main result.\www

\begin{proof}[Proof of Theorem $\ref{t2.1}$] We first construct the upper and lower control systems with principal eigenvalues, and then prove that the principal eigenvalues of the upper and lower control systems have the same limit to obtain the generalized principal eigenvalue.

Notice that $\theta\in C(\boo)$. Set
 \[\theta_M=\max_{\boo}\theta(x).\]

{\it Step 1}: The construction of the lower control matrix $\ud L^\ep(x,t)$ of $L(x,t)$.
For the given $0<\ep\ll 1$, we define
 \[\Sigma_\ep=\big\{x\in\boo:\,\theta(x)\ge\theta_M-\ep\big\}.\]
Then $\Sigma_\ep$ is a closed subset of $\ol\oo$, and $\theta(x)=\theta_M-\ep$ on $\partial\Sigma_\ep$. For any fixed $t\in[0,T]$, we define
 \bess
\ud \ell_{ik}^\ep(x,t)=\ell_{ik}(x,t),\;\;x\in\boo\;\;\;\mbox{for}\;\;i\not=k,\eess
and
 \bess
 \ud \ell_{ii}^\ep(x,t)=\left\{\begin{array}{ll}
 \ell_{ii}(x,t)-2\ep+\theta_M-\theta(x),\; &x\in\Sigma_\ep,\\[1mm]
 \ell_{ii}(x,t)-\ep,\; &x\in\boo\setminus\Sigma_\ep. \end{array}\right.\eess
Then $\ud \ell_{ii}^\ep(x,t)\le \ell_{ii}(x,t)-\ep$ in $\ol Q_T$. Set $\ud L^\ep(x,t)=(\ud \ell_{ik}^\ep(x,t))_{m\times m}$, i.e., for each $t\in[0,T]$,
 \bess
 \ud L^\ep(x,t)=\left\{\begin{array}{ll}
 L(x,t)+[\theta_M-2\ep-\theta(x)]I,\; &x\in\Sigma_\ep,\\[1mm]
  L(x,t)-\ep I,\; &x\in\boo\setminus\Sigma_\ep. \end{array}\right.
   \eess
Then $\ud L^\ep(x,t)$ is a continuous cooperative matrix-valued function, and $\ud L^\ep(x_0,t_0)$ is irreducible. For any $x\in\boo$, define an operator
$\ud{\mathscr P}_{x}^{\ep}$ by
 \[[\ud{\mathscr P}_{x}^{\ep}\phi](t):=-\frac{{\rm d}\phi(t)}{{\rm d}t}
  +\ud L^\ep(x,t)\phi(t), \;\; \phi \in \mathbb{Y}^m_T.\]
An easy computation yields that
 \[{\ud\theta_\ep(x)}:=s(\ud{\mathscr P}_{x}^{\ep})
=\begin{cases}
	\theta(x)+\theta_M-2\ep-\theta(x)=\theta_M-2\ep,& x\in\Sigma_\ep,\\
	\theta(x)-\ep<\theta_M-2\ep, & x\in\boo\setminus\Sigma_\ep.
\end{cases}\]
Obviously,
 \[\Big(\max_{\ol\oo}\ud\theta_\ep(x)-\ud\theta_\ep(x)\Big)^{-1}\not\in L^1(\oo).\]
According to Lemma \ref{lem:2.1}, $\ud{\mathscr L}^\ep$ has the principal eigenvalue  $\lm_p(\ud{\mathscr L}^\ep)$.

{\it Step 2}: The construction of the upper control matrix $\ol L^\ep(x,t)$ of $L(x,t)$. Similar to the above, for any fixed $t\in[0,T]$, we define
 \bess
 \bar \ell_{ik}^\ep(x,t)=\ell_{ik}(x,t),\;\;x\in\boo\;\;\;\mbox{for}\;\;i\not=k,\eess
and
 \bess
\bar \ell_{ii}^\ep(x,t)=\left\{\begin{array}{ll}
	\ell_{ii}(x,t)+\ep+\theta_M-\theta(x),\; &x\in\Sigma_\ep,\\[1mm]
	\ell_{ii}(x,t)+2\ep,\; &x\in\boo\setminus\Sigma_\ep. \end{array}\right.\eess
Then $\bar \ell_{ii}^\ep(x,t)\ge \ell_{ii}(x,t)+\ep$ in $\ol Q_T$. Set $\ol L^\ep(x,t)=(\bar \ell_{ik}^\ep(x,t))_{m\times m}$. Then
$\ol L^\ep(x,t)$ is a continuous cooperative matrix-valued function, and $\ol L^\ep(x_0,t_0)$ is irreducible. Similar to the above, for any $x\in\boo$, we define $\ol{\mathscr P}_{x}^{\ep}$ by
  \[[\ol{\mathscr P}_{x}^{\ep}\phi](t):=-\frac{{\rm d}\phi(t)}{{\rm d}t}+\ol L^\ep(x,t)\phi(t), \;\; \phi\in \mathbb{Y}^m_T.\]
Then
  \[\bar\theta_\ep(x):=s(\ol{\mathscr P}_{x}^{\ep})
 =\begin{cases}
	\theta(x)+\ep+\theta_M-\theta(x)=\ep+\theta_M,& x\in\Sigma_\ep,\\
	\theta(x)+2\ep<\ep+\theta_M, & x\in\boo\setminus\Sigma_\ep.
  \end{cases}\]
 Certainly,
  \[\Big(\max_{\ol\oo}\bar\theta_\ep(x)-\bar\theta_\ep(x)\Big)^{-1}\not\in L^1(\oo).\]
By Lemma \ref{lem:2.1}, $\ol{\mathscr L}^\ep$ has the principal eigenvalue  $\lm_p(\ud{\mathscr L}^\ep)$.

{\it Step 3}: From the constructions of $\ud L^\ep$ and $\ol L^\ep$ we easily see
that
 \[\ol L^\ep(x,t)= \ud L^\ep(x,t)+3\ep I,\]
and $\lm_p(\ud{\mathscr L}^\ep)$ and $\lm_p(\ol{\mathscr L}^\ep)$ are strictly decreasing and increasing in $\ep$, respectively. Therefore, the limits $\lim_{\ep\to 0}\lm_p(\ud{\mathscr L}^\ep)$ and $\lim_{\ep\to 0}\lm_p(\ol{\mathscr L}^\ep)$ exist, and they are equal, denoted by $\lm({\mathscr L})$, the generalized principal eigenvalue of ${\mathscr L}$.

Notice that for any given $\phi \in\mathbb{P}^m$, it is easy to see that
  \[\sup_{\lambda\in\mathbb{R}}\{{\ud{\mathscr L}}^{\ep}\phi\ge\lambda\phi\}\le \inf_{\lambda\in\mathbb{R}}\{{\ud{\mathscr L}}^{\ep}\phi\le\lambda \phi\}.\]
Since $\lm_p(\ud{\mathscr L}^\ep)$ is the principal eigenvalue of $\ud{\mathscr L}^{\ep}$ corresponding to an eigenfunction in $\mathbb{P}^m$, we have
  \[\lm_p(\ud{\mathscr L}^\ep)=\sup_{\phi\in\mathbb{P}^m}\sup_{\lambda\in\mathbb{R}}
  \{{\ud{\mathscr L}}^{\ep}\phi\ge\lambda\phi\}
 = \inf_{\phi\in\mathbb{P}^m}\inf_{\lambda\in\mathbb{R}}\{{\ud{\mathscr L}}^{\ep}
 \phi\le\lambda \phi\}=s(\ud{\mathscr L}^\ep).\]
Similarly,
  \[\lm_p(\ol{\mathscr L}^\ep)=\sup_{\phi \in\mathbb{P}^m}\sup_{\lambda\in\mathbb{R}}
  \{{\ol{\mathscr L}}^{\ep}\phi\ge\lambda\phi\}
  =\inf_{\phi\in\mathbb{P}^m}\inf_{\lambda\in\mathbb{R}}\{{\ol{\mathscr L}}^{\ep}\phi\le\lambda\phi\}=s(\ol{\mathscr L}^\ep).\]
Letting $\ep\to 0^+$, we conclude that \qq{2b.1} holds.

Clearly, $\lm({\mathscr L})\ge\theta_M$. If $\lm({\mathscr L})=\theta_M$, then $\lm({\mathscr L}) \in \sigma({\mathscr L})$. If $\lm({\mathscr L})>\theta_M$, then $\lm({\mathscr L})$ is an eigenvalue of ${\mathscr L}$ corresponding to a strongly positive eigenfunction, which means that $\lm({\mathscr L}) \in \sigma({\mathscr L})$. According to Lemma \ref{lem:conti}, $\lm({\mathscr L})$ is continuous in $L(x,t)$. The proof is complete.
\end{proof}

\begin{remark}	In Theorem $\ref{t2.1}$, it is shown that  $\lm({\mathscr L})=s({\mathscr L})$. Thanks to Lemma $\ref{lem:A1}$, it is known that $s({\mathscr L})=\frac{\ln r(\Phi(T,0))}{T}$. This implies that the generalized principal eigenvalue of ${\mathscr L}$ is also the exponential growth bound of $\Phi$ $($see, e.g., {\rm\cite[Proposition A.2]{Thieme2009SIAP})}.
\end{remark}

\section{Threshold dynamics for cooperative systems \qq{z1.3}}\lbl{Ses4}\setcounter{equation}{0}

As the applications of Theorem \ref{t2.1}, in this section we study the threshold dynamics for cooperative systems \qq{z1.3}. Its corresponding time-periodic problem is
 \bes\begin{cases}\label{3.1}
U_{it}=d_i\dd\int_\oo J_i(x,y) U_i(y,t)\dy-d_i^*(x)U_i+f_i(x,t,U),\;\;&(x,t)\in Q_T,\\[2mm]
  U_i(x,0)=U_i(x,T),&x\in\boo,\\
i=1,\cdots,m.
 \end{cases}\ees

Throughout this section, $J(x,y)$ and $J_i(x,y)$ satisfy the condition {\bf(J)}, $d^*(x)$ and $d_i^*(x)$ are defined by the manner {\bf(D)}.

\subsection{Maximum principle, positivity and uniqueness of solutions of \qq{z1.3} and \qq{3.1}}

We first investigate the maximum principle of \qq{z1.3}.\zzz

\begin{lem}\lbl{le3.1} Assume that $p\in L^\yy(Q_T)$ and $u\in\mathbb{Z}^1$ satisfies
 \bess\begin{cases}
 u_t\ge\dd\int_\oo J(x,y)u(y,t)\dy+p(x,t)u,\;\;& x\in\boo,\;0<t\le T,\\[1mm]
 u(x,0)\ge 0,&x\in\boo.
 \end{cases}\eess
 Then the following statements are valid:\vspace{-2mm}
 \begin{itemize}
 \item[\rm (i)] $u\ge 0$ in $Q_T$.\vvv
 \item[\rm (ii)] If $ u(x_0,0)>0$ for some $x_0\in\boo$, then $u(x_0,t)>0$ in $[0,T]$.\vvv
 \item[\rm (iii)] If there exists $\Omega_1\subset\oo$ with $|\Omega_1|>0$ such that $u(x,0)>0$ in $\Omega_1$, then $\inf_{x\in \Omega}u(x,t)>0$ for any $t \in (0,T]$. In particular, if $u\in\mathbb{X}^1$ and $u(x_0,0)>0$ for some $x_0\in\boo$, then $\inf_{x\in \Omega}u(x,t)>0$ for any $t\in 0,T]$.\www
 \end{itemize}
\end{lem}

\begin{proof}(i) The proof of the conclusion $u\ge 0$ in $Q_T$ is standard we omit the details.

(ii) Take $k>0$ large enough such that $k+p(x,t)>0$ in $\ol Q_T$. Then the function $v(x,t)={\rm e}^{kt}u(x,t)$ satisfies $v\ge 0$, $v(x_0,0)>0$ and
 \bess
 v_t(x_0,t)\ge\dd\int_\oo J(x_0,y)v(y,t)\dy+[k+p(x_0,t)]v(x_0,t)\ge 0,\;\;\forall\,0<t\le T.\eess
It follows that $v(x_0, t)>0$, i.e., $u(x_0,t)>0$ in $[0,T]$.

(iii) Without loss of generality, we assume that $p(x,t)>0$, and otherwise, $p(x,t)$ can be replaced by $p(x,t)+k$ by the same arguments as (ii).

We first consider the special case that $u\in\mathbb{X}^1$ and $u(x_0,0)>0$ for some $x_0\in\boo$. As $z(t)=u(x_0,t)$ satisfies
 \[z'(t)\ge\dd\int_\oo J(x_0,y)u(y,t)\dy+p(x_0,t)z(t)\ge p(x_0,t)z(t), \;\;0<t\le T\]
and $z(0)>0$, it yields $z(t)>0$, i.e., $u(x_0,t)>0$ in $[0,T]$.

Suppose that there exists $(x',t')\in\boo\times(0,T]$ such that $u(x',t')=0$. As $u(\cdot,t')\in C(\boo)$ and $u(x_0,t')>0$, we can find a $\sigma>0$ and $x^*\in\partial(B_\sigma(x_0)\cap\oo)$, such that $u(\cdot,t')>0$ in $B_\sigma(x_0)\cap\oo$ and $u(x^*,t')=0$. Then there exists $\delta>0$ such that $u(\cdot,t')>0$ in $B_\delta(x^*)\cap B_\sigma(x_0)\cap\oo$. Since $u(x^*,\cdot)\in C^1([0,T])$, $u(x^*,\cdot)\ge 0$ in $[0,t']$ and $u(x^*,t')=0$, we obtain that
 \[0\ge u(x^*,t')-u(x^*,0)\geq\int_{0}^{t'}\int_\oo J(x^*,y)u(y,s)\dy\mathrm{d}s>0.\]
This contradiction implies that $u(x,t)>0$ for all $(x,t)\in\ol\Omega\times(0,T]$, and hence, $\inf_{x \in \Omega} u(x,t)>0$.

We finally consider the general case that there exists $\Omega_1\subset\oo$ with  $|\oo_1|>0$ such that $ u(x,0)>0$ in $\Omega_1$. Integrating the above equation over $(0,t)$, we have
\bes
u(x,t)-u(x,0)\geq \int_{0}^{t}\int_{\Omega}J(x,y)u(y,t)\mathrm{d}y\dt.
\lbl{4d.2}\ees
Write $w(x,t)=\int_{\Omega} J(x,y) u(y,t) \mathrm{d}y, (x,t)\in\ol\Omega\times (0,T]$, which is a continuous function in $\ol\Omega\times[0,T]$.
Multiplying \qq{4d.2} by $J(z,x)$ and integrating the results over $\Omega$, we obtain that (for convenience, we still use $x$ and $y$ to represent $z$ and $x$, respectively)
 \[w(x,t)-w(x,0)\geq\int_{\Omega}\int_{0}^{t}J(x,y)w(y,t)\mathrm{d}y=\int_{0}^{t} \int_{\Omega}J(x,y)w(y,t)\mathrm{d}y.\]
Noticing that $w(x,0)$ is continuous on $\ol \Omega$, which is nonzero as $|\oo_1|>0$ and $w(x,0)>0$ in $\Omega_1$. By analogous argument to deal with the above special case, we have $w(x,t)>0$ in $\ol\Omega\times(0,T]$. This implies 
  \[u_t \geq w(x,t) >0,\;\;\; x\in\boo,\;0<t\le T.\]
We integrate the above inequality over $(0,t)$ to obtain that
\[u(x,t) -u(x,0) \geq \int_{0}^{t}w(x,s) \mathrm{d}s >0,\qquad x\in\boo,\;0<t\le T.\]
Since $w(x,t)$ is continuous on $\bqq$,  we can further derive that $\inf_{x \in\Omega}u(x,t)>0$. 
\zzz\end{proof}

\begin{col}\lbl{col4.1} Let $p\in L^\yy(Q_T)$. Assume that $u\in \mathbb{Z}^1$ and $u\ge 0$ in $Q_T$. If $u$ satisfies
 \bess
 u_t\ge\dd\int_\oo J(x,y)u(y,t)\dy+p(x,t)u,\;\;&(x,t)\in Q_T,
 \eess
and $u(x_0, t_0)=0$ for some $(x_0,t_0)\in Q_T$, then $u(x_0, t)\equiv 0$ in $[0,t_0]$.\www
\end{col}

Using Lemma \ref{le3.1} we have the following comparison principle.\www

\begin{theo}\lbl{t3c.1}{\rm(Comparison principle)} Assume that $\bar u,\ud u\in \mathbb{Z}^m$, $\bar u,\ud u\ge 0$ and  satisfy
 \bes\begin{cases}
 \bar u_{it}\ge d_i\dd\int_\oo J_i(x,y)\bar u_i(y,t)\dy-d_i^*(x)\bar u_i+f_i(x,t, \bar u),\;\;&(x,t)\in Q_T,\\[3mm]
 \ud u_{it}\le d_i\dd\int_\oo J_i(x,y)\ud u_i(y,t)\dy-d_i^*(x)\ud u_i+f_i(x,t, \ud u),\;\;&(x,t)\in Q_T,\\[1mm]
 \bar u(x,0)\ge \ud u(x,0),&(x,t)\in\ol Q_T,\\
i=1,\cdots,m.
 \end{cases}\lbl{3c.3}\ees
If $f_i(x,t,u)\in C(\bqq\times\mathbb{R}_+^m)$, $f_i(x,t,u)$ is locally Lipschitz continuous with respect to $u$ for all $i\in\sss$ and $f(x,t,u)$ is cooperative for all $u \ge 0$, then $\bar u\ge\ud u$ in $\ol Q_T$.\zzz
\end{theo}

Then we study the strong maximum principle for the systems. To achieve this aim, we first prove a conclusion about the scalar equation.\zzz

\begin{lem}\lbl{le3.3} Let $p\in L^\yy(Q_T)$ and $U\in \mathbb{Z}^1$. If $U\ge 0$ and satisfies
 \bes\begin{cases}
 U_t\ge d\dd\int_\oo J(x,y)U(y,t)\dy-d^*(x)U(x,t)+p(x,t)U(x,t),&(x,t)\in Q_T,\\[3mm]
 U(x,0)\ge U(x,T),&(x,t)\in\ol Q_T.
 \end{cases} \lbl{3c.7}\ees
Then either $U\equiv 0$, or $U>0$ in $\ol Q_T$ and $\inf_{Q_T}U>0$.\zzz
\end{lem}

\begin{proof} Define
  \[h(x,t)=\int_\oo J(x,y)U(y,t)\dy.\]
Then $h\in C(\ol Q_T)$ and $h$ is nonnegative in $Q_T$.
We claim that if there exists $(x_0,t_0)\in\ol Q_T$ such that $h(x_0,t_0)>0$, then $U>0$ in $\ol Q_T$ and $\inf_{Q_T}U>0$. This implies the desired results.

We first show $U(x_0,t_0)>0$. If $U(x_0,t_0)=0$, we may think of $t_0>0$ since $0\le U(x_0,T)\le U(x_0,0)$. Take advantage of the facts that $U(x_0,\cdot)\in C^1([0,T])$ and $U(x_0,t)\ge 0$, it derives that $U_t(x_0,t_0)\le 0$. Recalling the differential equation of \qq{3c.7}, this is impossible.

We next prove that $h>0$ in $\ol Q_T$. By the continuity of $h(x,t)$, the set
 \[{\cal O}=\{(x,t)\in\ol Q_T:\, h(x,t)>0\}\]
is an open subset of $\ol Q_T$, and $U>0$ in ${\cal O}$ by the above claim. We will prove that ${\cal O}$ is a closed subset of $\ol Q_T$. Let $(x_l,t_l)\in{\cal O}$ and $(x_l,t_l)\to (x^*,t^*)\in\ol Q_T$, then $U(x_l,t_l)>0$. Since ${\cal O}$ is an open subset of $\ol Q_T$, there exists  $\delta_l>0$ such that $U(x,t)>0$ in $B_{\delta_l}((x_l,t_l))\cap\ol Q_T$. If $h(x^*,t^*)=0$, then $U(x,t^*)=0$ in a neighborhood $B_\sigma(x^*)\cap\boo$ for some $\sigma>0$ since $J(x^*,x^*)>0$ and $J(x^*,y)$ is continuous in $y\in\ol Q_T$. In the case of $t^*>0$, it follows from Corollary \ref{col4.1} that $U(x,t)=0$ in $(B_\sigma(x^*)\cap\boo)\times[0,t^*]$. Certainly, $U(x,0)=0$, which implies $U(x,T)=0$ in $B_\sigma(x^*)\cap\boo$. Thanks to Corollary \ref{col4.1} again, $U(x,t)=0$ in $(B_\sigma(x^*)\cap\boo)\times[0,T]$. In the case of $t^*=0$, we have $U(x,T)=0$ in $B_\sigma(x^*)\cap\boo$. Similarly, $U(x,t)=0$ in $(B_\sigma(x^*)\cap\boo)\times[0,T]$. So, there exists $\delta>0$ such that $U(x,t)=0$ in $B_\delta((x^*,t^*))\cap\ol Q_T$. This is impossible as $(x_l,t_l)\to (x^*,t^*)\in\ol Q_T$ and $U(x_l,t_l)>0$. Thus, $h(x^*,t^*)>0$ and ${\cal O}$ is a closed subset of $\ol Q_T$. It then follows that ${\cal O}=\ol Q_T$ and hence $h>0$, $U>0$ in $\ol Q_T$.

Now we prove $\inf_{Q_T}U>0$. Assume $\inf_{Q_T}U=0$. We first have the following claim.
 \begin{enumerate}\vspace{-2mm}
\item[{\bf Claim}] There is no sequence $\{(x_l,t_l)\}\subset\ol Q_T$ such that $U_t(x_l,t_l)\le 0$ and $U(x_l,t_l)\to 0$.
   \end{enumerate}\vspace{-2mm}
In fact, by choosing a subsequence if necessary, we may assume that $(x_l,t_l)\to (x_*,t_*)\in\ol Q_T$. Take $(x,t)=(x_l,t_l)$ in \qq{3c.7} to derive
 \[0\ge U_t(x_l,t_l)\ge d h(x_l,t_l)-d^*(x_l)U(x_l,t_l)+p(x_l,t_l)U(x_l,t_l),\]
and let $l\to+\yy$ to deduce $h(x_*,t_*)\le 0$ as $h(x,t)$ is continuous. A contradiction is obtained because we have known that $h>0$ in $\ol Q_T$.

As $\inf_{Q_T}U=0$, by the above Claim, there exist $(x_l,t_l)\in\ol Q_T$ such that $U(x_l,t_l)\to 0$ as $l\to+\yy$ and $U_t(x_l,t_l)>0$ for all $l$. In this part of the proof, only the subsequence is sufficient, so we no longer emphasize the subsequence but write it as the sequence itself. If $t_l=0$, then
$U(x_l,T)\le U(x_l,t_l)\to 0$. So, we can think of $t_l>0$. Since $U_t(x_l,t)$ is continuous in $t$, the following number
 \[\bar t_l=\inf\{0\le t< t_l:\, U_t(x_l,s)>0,\;\;\forall\,t\le s\le t_l\}\]
is well defined, and $U(x_l,\bar t_l)\le U(x_l,t_l)\to 0$. If $\bar t_l>0$, then $U_t(x_l,\bar t_l)=0$. This is impossible by the above Claim. Therefore, $\bar t_l=0$ and $U(x_l,T)\le U(x_l,0)\to 0$. So $U_t(x_l,T)>0$, i.e., $t_l=T$ and $U_t(x_l,t)>0$ in $(0, T]$. This implies
 \bes
 U(x_l,t)<U(x_l,T)\to 0\;\;\mbox{ as }\; l\to+\yy,\;\; \forall \;0\le t\le T.
\lbl{3c.8}\ees
However, by \qq{3c.7} we have
 \[U(x_l,T)\ge U(x_l,0)+\int_0^T\big[d h(x_l,t)-d^*(x_l)U(x_l,t)+p(x_l,t)U(x_l,t)\big]\dt.\]
Assume that $x_l\to x_*\in\ol Q_T$. Letting $l\to+\yy$ in the above inequality, and using \qq{3c.8} and the dominated convergence theorem, we have
 \[\int_0^T  h(x_*,t)\dt\le 0.\]
This is a contradiction as $h(x_*, t)>0$ and $h(x_*, t)$ is continuous in $t\in[0,T]$. The conclusion $\inf_{Q_T}U>0$ is proved.\zzz
\end{proof}

\begin{col}\lbl{c3.1} Let $p_{ik}\in L^\yy(Q_T)$ and $(p_{ik})_{m\times m}$ be cooperative in $Q_T$, i.e. $p_{ik}\ge 0$ in $Q_T$ when $i\not=k$.  Assume that $U\in \mathbb{Z}^m$, $U\ge 0$ and satisfies
 \bess\begin{cases}
U_{it}>d_i\dd\int_\oo J_i(x,y)U_i(y,t)\dy-d^*_i(x)U_i(x,t)
+\sum_{k=1}^mp_{ik}(x,t)U_k(x,t),\;& (x,t)\in\ol Q_T, \\
U_i(x,0)\ge U_i(x,T), &x\in\boo,\\
 i=1,\cdots, m.
 \end{cases}\eess
Then $U_i>0$ in $\ol Q_T$ and $\inf_{Q_T}U_i>0$ for all $i\in\sss$.
\end{col}

\begin{theo}\lbl{th3.2}{\rm(Strong maximum principle)} Let $p_{ik}\in L^\yy(Q_T)$ and $(p_{ik})_{m\times m}$ be cooperative in $Q_T$, i.e. $p_{ik}\ge 0$ in $Q_T$ when $i\not=k$. Assume that $U\in \mathbb{Z}^m$, $U\ge 0$ and satisfies
 \bess\begin{cases}
U_{it}\ge d_i\dd\int_\oo J_i(x,y)U_i(y,t)\dy-d^*_i(x)U_i(x,t)
+\sum_{k=1}^mp_{ik}(x,t)U_k(x,t),\;& (x,t)\in\ol Q_T, \\
U_i(x,0)\ge U_i(x,T), &x\in\boo,\\
 i=1,\cdots, m.
 \end{cases}\eess
Define $\tilde{p}_{ik}=\frac1{|\Omega|T}\int_0^T\int_{\Omega}p_{ik}(x,t) \mathrm{d}x\mathrm{d}t$. If $(\tilde{p}_{ik})_{m\times m}$ is irreducible, then either $U_i\equiv 0$ in $\ol Q_T$ for all $i\in\sss$, or $U_i>0$ in $\ol Q_T$ and $\inf_{Q_T}U_i>0$ for all $i\in\sss$.
\end{theo}

\begin{proof} For any $i\in\sss$, there holds:
 \[U_{it}\ge d_i\dd\int_\oo J_i(x,y)U_i(x,t)\dy-d^*_i(x)U_i(x,t)+p_{ii}(x)U_i(x,t),\;\; (x,t)\in\ol Q_T.\]
By Lemma \ref{le3.3}, we have that either $U_i\equiv 0$ in $\ol Q_T$, or $U_i>0$ in $\ol Q_T$ and $\inf_{Q_T}U_i>0$.

Assume that there exists $\bar i\in\sss$ such that $U_{\bar i}\equiv 0$ in $\ol Q_T$. Then we have
 \[{\cal I}=\{i\in\sss:\,U_i\equiv 0\;\;\mbox{in}\,\;\ol Q_T\}\not=\emptyset.\]
Define ${\cal I}^c=\sss\setminus{\cal I}$. We next prove that ${\cal I}^c$ is an empty set. Suppose that ${\cal I}^c$ is not empty, then $U_k>0$ in $\ol Q_T$ for all $k\in{\cal I}^c$, and $U_i \equiv 0$ in $\ol Q_T$ for all $i\in {\cal I}$ by Lemma \ref{le3.3} again. Thus we have
 \bess
 0=U_{it}(x,t)&\ge& d_i\dd\int_\oo J_i(x,y)U_i(y,t)\dy-d^*_i(x)U_i(x,t)
+\sum_{k=1}^m p_{ik}(x,t)U_k(x,t)\\[1mm]
&=&\sum_{k\in {\cal I}^c}p_{ik}(x,t)U_k(x,t),\;\;\forall\,i\in{\cal I},\;  (x,t)\in\bqq.
 \eess
This implies that $p_{ik}\equiv 0$ in $\bqq$ for all $i\in {\cal I}$ and $k\in {\cal I}^c$, which contradicts the fact that $(\tilde{p}_{ik})_{m\times m}$ is irreducible. We thus have that ${\cal I}^c$ is empty. Hence, $U_i\equiv 0$ in $\ol Q_T$ for all $i\in\sss$. \end{proof}\www

Then we investigate the positivity of non-negative and nontrivial solutions of \qq{3.1}.

\begin{theo}\lbl{th3.3}{\rm(Positivity of non-negative and nontrivial solutions)} Assume that {\bf(F1)} and {\bf(F2)} hold. Let $U\in \mathbb{Z}^m$ and $U\ge 0$ be a solution of \qq{3.1}. Then either $U_i\equiv 0$ in $\ol Q_T$ for all $i\in\sss$, or $U_i>0$ in $\ol Q_T$ and $\inf_{\oo}U_i>0$ for all $i\in\sss$.\zzz
\end{theo}

\begin{proof} Noticing that $U_i$ satisfies
\bess\begin{aligned}
U_{it}&=d_i\dd\int_\oo J_i(x,y)U_i(y,t)\dy-d^*_i(x)U_i(x,t)+f_i(x,t,U)- f_i(x,t,0)\\
	&=d_i\dd\int_\oo J_i(x,y)U_i(y,t)\dy-d^*_i(x)U_i(x,t)+\sum_{k=1}^mp_{ik}(x,t)U_k(x,t),
	\end{aligned}\eess
where
 \[p_{ik}(x,t)=\int_0^1\partial_{u_k}f_i(x,t, sU(x,t)){\rm d}s\ge 0.\]
By {\bf(F2)}, there exist $\delta>0$ and $M>\sup_{i\in\sss}\sup_{\bqq} U_i(x,t)$ such that $(\partial_{u_k}f_i(x,t,u))_{m\times m}$ is irreducible for all $(x,t,u)\in [B_{\delta}((x_0,t_0))\cap\bqq]\times[0,M]$. Therefore,
and $(p_{ik}(x,t))_{m\times m}$ is irreducible for all $(x,t)\in B_{\delta}((x_0,t_0)) \cap \bqq$. According to Theorem \ref{th3.2}, we can get the desired results.\zzz
\end{proof}

Now we use Theorem \ref{th3.3} to derive the uniqueness of bounded positive solutions of \qq{3.1}.\zzz

\begin{theo}\lbl{th3.4}{\rm(Uniqueness of bounded positive solutions)} Assume that {\bf(F1)}--{\bf(F3)} hold. Then \qq{3.1} has at most one bounded positive solution $U$ in $\mathbb{Z}^m$ {\rm(}no need for continuity with respect to $x)$.\zzz
\end{theo}

\begin{proof} Let $U,V\in \mathbb{Z}^m$ and $U, V\gg 0$ in $\ol Q_T$ be two solutions  of \qq{3.1}. By Theorem \ref{th3.3}, $\inf_{\ol Q_T}U_i>0$ and  $\inf_{\ol Q_T}V_i>0$ for all $i\in\sss$.

As $U,V\in[L^\yy(Q_T)]^m$, we can find $0<s\ll 1$ such that $V\ge sU$ in $\ol Q_T$.
Hence, the set
 \[\Sigma=\kk\{0<s\le 1: V(x,t)\ge sU(x,t),\;\;\forall\,(x,t)\in\ol Q_T\rr\}\]
is nonempty. So $\bar{s}:=\sup\Sigma$ exists, and $V(x,t)\ge\bar sU(x,t)$ in $\ol Q_T$. Moreover, there exists $\bar i$ such that $\inf_{Q_T}(V_{\bar i}(x,t)-\bar sU_{\bar i}(x,t))=0$.

If $\bar{s}=1$, then $U(x,t)=V(x,t)$ in $\ol Q_T$. The proof is complete. 	
Assume that $\bar{s}<1$. Since $f$ is strictly subhomogeneous, we have
 \[f(x,t,\overline{s}U(x,t))>\overline{s}f(x,t,U(x,t)), \;\; (x,t)\in\overline Q_T.\]
Write $W(x,t)=(W_1(x,t),\cdots,W_m(x,t))= V(x,t)-\overline{s}U(x,t)$, and define
 \[a_{ik}(x,t)=\int_0^1\partial_{u_k}f_i(x,t,\bar s\ol U(x,t)
 +\tau W(x,t)){\rm d}\tau\ge 0.\]
Then $(a_{ik}(x,t))_{m\times m}$ is irreducible at the neighborhood of $(x_0,t_0)$ by the condition {\bf(F2)}. Moreover,
 \bes
W_{it}&=&d_i\dd\int_\oo J_i(x,y)W_i(y,t)\dy-d^*_i(x)W_i+f_i(x,t, V)- \overline{s}f_i(x,t,U)\nm\\
	&\ge& d_i\dd\int_\oo J_i(x,y)W_i(y,t)\dy-d^*_i(x)W_i+f_i(x,t, V)- f_i(x,t,\overline{s}U)\lbl{3c.10}\\
	&=&d_i\dd\int_\oo J_i(x,y)W_i(y,t)\dy-d^*_i(x)W_i+\sum_{k=1}^ma_{ik}(x,t)W_k\nm\\
&\ge& d_i\dd\int_\oo J_i(x,y)W_i(y,t)\dy-d^*_i(x)W_i+a_{ii}(x,t)W_i,\;\;\forall\,(x,t)\in Q_T,
	\lbl{3c.11}\ees
and for any given $(\hat x,\hat t)\in Q_T$, the strictly inequality of \qq{3c.10}  holds with $(x,t)$ replaced by $(\hat x,\hat t)$ for some $\hat i\in\sss$. Noticing that $W_{\hat i}(\hat x,0)\ge 0$, and
 \bess
 W_{\hat it}(\hat x,t)&\ge& d_{\hat i}\dd\int_\oo J_{\hat i}(\hat x,y)W_{\hat i}(y,t)\dy -d^*_{\hat i}(\hat x)W_{\hat i}(\hat x,t)
 +a_{\hat i\hat i}(\hat x,t)W_{\hat i}(\hat x,t),\;\;\forall\,t\in [0,T],\eess
and the above inequality is strict at $t=\hat{t}$, we conclude that $W_{\hat i}(\hat x,\hat t)>0$ by Corollary \ref{col4.1}. So $W_{\hat i}>0$ in $\ol Q_T$ by Lemma \ref{le3.3}. On the other hand, since $\inf_{Q_T}W_{\bar i}=0$, it deduces that $W_{\bar i}\equiv 0$ in $\ol Q_T$ by Lemma \ref{le3.3}, which is a contradiction. So $\bar s=1$ and the proof is complete.\zzz
\end{proof}

In the following of this section we always assume that {\bf(F1)}--{\bf(F3)} hold.

\subsection{Upper and lower solutions method for \qq{3.1}}

\begin{defi} A function $U\in \mathbb{X}^m$ with $U\ge 0$ is called an upper solution $($a lower solution$)$ of \qq{3.1} if
 \bes\begin{cases}
 U_{it}\ge\, (\le)\, d_i\dd\int_\oo J_i(x,y)U_i(y,t)\dy-d_i^*(x)U_i+f_i(x,t, U),\;\;&(x,t)\in Q_T,\\[1mm]
 U(x,0)\ge \,(\le) \,U(x,T),&x\in\boo,\\
i=1,\cdots,m.
 \end{cases}\lbl{3c.5}\ees
\end{defi}\www

\begin{theo}\lbl{th3.5} Assume that $\ud U,\ol U\in \mathbb{X}^m$ with $0\le\ud U\le\ol U$ are the upper and lower solutions of \qq{3.1}, respectively. Then \qq{3.1} has a unique bounded positive solution $U$, and $U\in[C(\ol Q_T)]^m$ and satisfies $\ud U\le U\le\ol U$.
\zzz\end{theo}

\begin{proof} We first consider the initial value problem
\bes\begin{cases}
  U_{it}=d_i\dd\int_\oo J_i(x,y)U_i(y,t)\dy-d_i^*(x)U_i+f_i(x,t,U),\;\;&x\in Q_T,\\[1mm]
 U_i(x,0)=\ud U_i(x,T),&x\in\ol\oo,\\
 i=1,\cdots,m.
 \end{cases}\lbl{3c.13}\ees
It is easy to see that $\ol U$ and $\ud U$ are the upper and lower solutions of  \qq{3c.13}, respectively. So, \qq{3c.13} has a unique positive solution $\ud U^1\in[C(\ol Q_T)]^m$ and $\ud U\le\ud U^1\le\ol U$ in $\ol Q_T$. Certainly,
 \[\ud U(x,0)\le\ud U(x,T)\le\ud U^1(x,T)\le\ol U(x,T)\le\ol U(x,0),\;\;x\in\ol\oo.\]
Of course, $\ol U$ and $\ud U$ are, respectively, the upper and lower solutions of
the following initial value problem
 \bess\begin{cases}
  U^2_{it}=d_i\dd\int_\oo J_i(x,y)U^2_i(y,t)\dy-d_i^*(x)U^2_i+f_i(x,t,U^2),\;\;&x\in Q_T,\\[1mm]
 U^2_i(x,0)=\ud U^1_i(x,T),&x\in\ol\oo,\\
i=1,\cdots,m.
 \end{cases}\eess
Hence the above problem has a unique solution $\ud U^2\in[C(\ol Q_T)]^m$ and $\ud U\le\ud U^2\le\ol U$ in $\ol Q_T$. Since $\ud U^2(x,0)=\ud U^1(x,T)\ge\ud U(x,T)=\ud U^1(x,0)$, the comparison principle gives $\ud U^2\ge\ud U^1$. Define $\ud U^2={\mathscr P}(\ud U^1)$ and $\ud U^{n+1}={\mathscr P}(\ud U^n)$ for $n\ge 2$. Then $\ol U$ and $\ud U$ are also the upper and lower solutions of the problem solving $\ud U^{n+1}$. Therefore, $\ud U\le\ud U^{n+1}\le\ol U$ in $\ol Q_T$. On the other hand, Using $\ud U^{n+1}(x,0)=\ud U^{n}(x,T)\ge \ud U^{n-1}(x,T)$ we have
 \[\ud U(x,t)\le \ud U^{n}(x,t)\le\ud U^{n+1}(x,t)\le\ol U(x,t),\;\;\forall\,(x,t)\in\ol Q_T,\; n\ge 1.\]
Consequently, there exists a function $\widetilde U$ satisfying $\ud U\le\widetilde U\le\ol U$ in $\ol Q_T$ such that $\dd\lim_{n\to+\yy}\ud U^{n+1}(x,t)=\widetilde U(x,t)$ for all $(x,t)\in\ol Q_T$. It is clear that $\widetilde U$ is semi-lower continuous and
 \[\widetilde U(x,0)=\lim\limits_{n\rightarrow +\infty} \ud U^{n+1}(x,0)= \lim\limits_{n\rightarrow +\infty}\ud U^{n}(x,T)=\widetilde U(x,T).\]
Write the differential equation of $\ud U^n$ as an integral equation we have
 \[\ud U^n(x,t)=\ud U^{n-1}(x,T)+\int_0^t\kk(d_i\dd\int_\oo J_i(x,y)\ud U^n_i(y,s)\dy-d_i^*(x)\ud U^n_i(x,s)+f_i(x,s,\ud U^n)\rr){\rm d}s.\]
Take $n\to+\yy$ and use the dominated convergence theorem to derive that $\widetilde U(x,t)$ satisfies
 \[\widetilde U(x,t)=\widetilde U(x,T)+\int_0^t\kk(d_i\dd\int_\oo J_i(x,y)\widetilde U_i(y,s)\dy-d_i^*(x)\widetilde U_i(x,s)+f_i(x,s,\widetilde U)\rr){\rm d}s.\]
This implies that $\widetilde U\in \mathbb{Z}^m$ is a bounded positive solution of \qq{3.1}.

Replacing $\ud U_i(x,T)$ by $\ol U_i(x,T)$ and repeating the above discussions we can construct a monotone decreasing sequence $\{\ol U^{n}\}$ satisfying $\ol U^{n}\in[C(\ol Q_T)]^n$ and
 \[\ud U(x,t)\le \ol U^{n+1}(x,t)\le\ol U^{n}(x,t)\le\ol U(x,t),\;\;\forall\,(x,t)\in\ol Q_T,\; n\ge 1.\]
Moreover, there exists a function $\widehat U$ satisfying $\ud U\le\widehat U\le\ol U$ in $\ol Q_T$ such that $\dd\lim_{n\to+\yy}\ol U^{n+1}(x,t)=\widehat U(x,t)$ for all $(x,t)\in\ol Q_T$. It is clear that $\widehat U$ is semi-upper continuous and
 \[\widehat U(x,0)=\widehat U(x,T).\]
Similar to the above, $\widehat U\in \mathbb{Z}^m$ is a bounded positive solution of \qq{3.1}. Hence, $\widehat U=\widetilde U:=U$ by Theorem \ref{th3.4}. Of course, $U\in[C(\ol Q_T)]^m$ since $\widetilde U$ and $\widehat U$ are semi-lower and semi-upper continuous, respectively. Make use of Theorem \ref{th3.4} again, the bounded positive solution of \qq{3.1} is unique.\www\end{proof}

\subsection{The global dynamics of \qq{z1.3}}

We are now in a position to prove our second main result.\www

\begin{proof}[Proof of Theorem \ref{th3.6}]  As in the introduction, we set $b_{ik}(x,t)=\frac{\partial}{\partial u_k}f_i(x,t,0)$, and \[\ell_{ii}(x,t)=b_{ii}(x,t)-d_i^*(x),\;\;\ell_{ik}(x,t)=b_{ik}(x,t),\;\; k\not=i.\]  Let ${\mathscr L}$ be defined by \qq{1.2}. Then $L(x_0, t_0)=(\ell_{ik}(x_0,t_0))_{m\times m}$ is also irreducible, and the condition {\bf (L2)} holds. Therefore, Theorem \ref{t2.1} is true.

(1) Assume that $\lm(\mathscr{L})>0$ and \qq{3.5} holds.

{\it Step 1: The existence and uniqueness of continuous positive solution of \qq{3.1}}. Let $\ud L^\ep(x,t)$ be the lower control matrix of $L(x,t)$ as in Theorem \ref{t2.1}. Then $\lm_p(\ud{\mathscr L}^\ep)\to \lm(\mathscr{L})$ as $\ep\to 0^+$, so $\lm_p(\ud{\mathscr L}^\ep)>0$ when $0<\ep\ll1$. Let $\phi(x,t)$ with $\|\phi\|_{L^\yy(Q_T)}=1$ be a positive eigenfunction corresponding to $\lm_p(\ud{\mathscr L}^\ep)$. Set $\ud U=\rho\phi$ with $0<\rho\ll 1$ being determined later. Noticing that
 \bess
 f_i(x,t,\rho\phi)=f_i(x,t,\rho\phi)-f_i(x,t,0)=
 \sum_{k=1}^m\kk(\int_0^1\partial_{u_k}f_i(x,t, s\rho\phi){\rm d}s\rr)\rho\phi_k=:\sum_{ k=1}^mb_{ik}^\rho(x,t)\rho\phi_k,
 \eess
and $b_{ik}^\rho\to b_{ik}$ uniformly in $\ol Q_T$ as $\rho\to 0$. Then there exists $\rho_0>0$ such that
 \[\lm_p(\ud{\mathscr L}^\ep)\phi_i+\sum_{k=1}^m[b_{ik}^\rho(x,t)-b_{ik}(x,t)]\phi_k>0,
 \;\;(x,t)\in\ol Q_T\]
provided $0<\rho\le\rho_0$ because of $\lm_p(\ud{\mathscr L}^\ep)>0$ and $\phi$ is positive and continuous in $\ol Q_T$. Therefore,
 \bess
 &&d_i\dd\int_\oo J_i(x,y)\ud U_i(y,t)\dy-d^*_i(x)\ud U_i+f_i(x,t,\ud U)-\ud U_{it}\\
 &=&\rho d_i\dd\int_\oo J_i(x,y)\phi_i(y,t)\dy-\rho d^*_i(x)\phi_i
  +\sum_{k=1}^mb_{ik}^\rho(x,t)\rho\phi_k-\rho\phi_{it}\\
 &=&\rho d_i\dd\int_\oo J_i(x,y)\phi_i(y,t)\dy+\rho\sum_{k=1}^m\ud\ell^\ep_{ik}(x,t)\phi_k-\rho\phi_{it}+\rho \sum_{k=1}^m\kk\{(\ell_{ik}-\ud \ell_{ik}^\ep)\phi_k
 +(b_{ik}^\rho-b_{ik})\phi_k\rr\}\\
 &\ge&\rho \lm_p(\ud{\mathscr L}^\ep)\phi_i+\rho\sum_{k=1}^m[b_{ik}^\rho(x,t)
 -b_{ik}(x,t)]\phi_k\\
 &>&0,\;\;(x,t)\in\ol Q_T,\; 0<\rho\le\rho_0.
  \eess
Certainly, $\ud U\le\ol U$ in $\ol Q_T$ when $0<\rho\ll1$. By Theorem \ref{th3.5}, the system \eqref{3.1} has a unique positive solution $U$, and $U\in[C(\ol Q_T)]^m$ and  $\ud U\le U\le\ol U$ in $\ol Q_T$.

{\it Step 2: The proof of \qq{3c.15}.} Notice that $u(x,t;u_0)$ is continuous
in $\boo\times[0,+\yy)$ and $u(x,t;u_0)\gg 0$ in $\boo\times(0,+\yy)$. Without loss of generality we may assume that $u_0(x)\gg 0$ in $\boo$. Then there exist $0<\rho\le\rho_0$ and $\gamma>1$ such that
 \[\ud U(x,0):=\rho\phi(x,0)\le u_0(x)\le\gamma U(x,0),\;\;x\in\boo.\]
As $f$ is strictly subhomogeneous, we have $f(x,t,\gamma u)<\gamma f(x,t,u)$ for all $x\in\boo\times[0,+\yy)$, $u\gg 0$ and $\gamma>1$. Remind this fact, it is easy to verify that $\gamma U$ still satisfies \qq{3.5}.

On the other hand, by the comparison principle,
 \bes
 V(x,t):=u(x,t;\ud U(x,0))\le u(x,t;u_0)\le u(x,t;\gamma U(x,0))=:W(x,t)
  \lbl{3c.18}\ees
in $\boo\times[0,+\yy)$, and
\bess
 \ud U(x,t)\le V(x,t)\le U(x,t)\le W(x,t)\le\gamma U(x,t)
  \lbl{3c.18x}\eess
in $\ol Q_T$. For any integer $n\ge 0$, let us define
 \[V^n(x,t)=V(x,t+nT),\;\;\;(x,t)\in\overline Q_T.\]
Since $f(x,t,u)$ is time-periodic with period $T$, we see that $V^n$ satisfies
  \begin{eqnarray*}\left\{\begin{array}{lll}
 V^n_{it}=d_i\dd\int_\oo J_i(x,y)V^n_i(y,t)\dy-d^*_i(x)V^n_i+f_i(x,t,V^n)=0,\;\;&(x,t)\in Q_T,\\[1mm]
  V^n(x,0)=V(x,nT),&x\in\ol\Omega,\\
i=1,\cdots,m.
  \end{array}\right.\end{eqnarray*}
Note that
 \begin{eqnarray*}
 V(x,0)=\ud U(x,0)=\ud U(x,T)\le V(x,T)=V^1(x,0)\le U(x,T)=U(x,0),\;\;x\in\boo.\end{eqnarray*}
We can adopt the comparison principle to produce $\ud U\le V\le V^1\le U$ in $\overline Q_T$ which in turn asserts
  $$ V^1(x,0)=V(x,T)\le V^1(x,T)=V^2(x,0)\le U(x,T)=U(x,0),\;\;x\in\boo.$$
As above, $V^1\le V^2\le U$ in $\overline Q_T$.
Applying the inductive method, we can prove that $V^n$ is monotonically increasing in $n$ and $\ud U\le V^n\le U$ in $\overline Q_T$ for all $n$. Consequently, there exists a positive function $\widetilde U$ satisfying $\ud U\le\widetilde U\le U$ such that $V^n\to \widetilde U$ pointwisely in $\overline Q_T$ as $n\to+\yy$, and $\widetilde U$ is semi-lower continuous in $\ol Q_T$. Clearly, $\widetilde U(x,0)=\widetilde U(x,T)$ since $V^{n+1}(x,0)=V^n(x,T)$. Similar to the proof of Theorem \ref{th3.5} we can show that $\widetilde U \in \mathbb{Z}^m$ is a solution of \eqref{3.1}. Consequently, $\widetilde U=U$ by Theorem \ref{th3.4}.

In view of \qq{3c.18}, it follows that
 \bes
  \liminf_{n\rightarrow +\infty} u(x,t+nT;u_0)\ge \lim_{n\to+\yy}V(x,t+nT)=U(x,t) \;\;\text{ uniformly in } \; \ol Q_T.\lbl{3c.19}\ees

Similar to the above, we can prove that $W^n(x,t)=W(x, t+nT)$ is monotonically decreasing in $n$ and $\ud U\le W^n\le\gamma U$ in $\overline Q_T$ for all $n$. And $W^n\to U$ pointwisely in $\overline Q_T$ as $n\to+\infty$. Moreover,
 \bess
  \limsup_{n\rightarrow +\infty} u(x,t+nT;u_0)\le \lim_{n\to+\yy}W(x,t+nT)=U(x,t) \;\;\text{ uniformly in } \; \ol Q_T.\eess
Combining this with \qq{3c.19} and using the Dini Theorem, we see that the limit \qq{3c.15} holds.

(2) Assume that $\lm(\mathscr{L})<0$.

{\it Step 1: The nonexistence of continuous positive solution of \qq{3.1}}. Assume on the contrary that $U\in[C(\ol Q_T)]^m$ is a positive solution of \qq{3.1}. Then, for $0<\rho\ll 1$,
\bess
 f_i(x,t,\rho U)=f_i(x,t,\rho U)-f_i(x,t,0)=\sum_{k=1}^m\kk(\int_0^1\partial_{u_k}f_i(x,t, s\rho U){\rm d}s\rr)\rho U_k=:\sum_{k=1}^mb^\rho_{ik}(x,t)\rho U_k,
 \eess
and $b_{ik}^\rho\to b_{ik}$ uniformly in $\ol Q_T$ as $\rho\to 0$.
Let $L_\rho(x,t)=(\ell^\rho_{ik}(x,t))_{m\times m}$ with $\ell_{ik}^{\rho}(x,t)=b^\rho_{ik}(x,t)$ for $i\neq k$ and $\ell_{ii}^{\rho}(x,t)=b^\rho_{ii}(x,t)-d_i^*(x)$, and let $\mathscr{L}_{\rho}$ be defined by \eqref{1.2} with $\ell_{ik}$ replaced by $\ell_{ik}^{\rho}$. By the continuity, there exists $\rho_0>0$ such that $\lm(\mathscr{L}_{\rho})<0$ for all $0<\rho\le\rho_0$. Let $\ol L^\ep_\rho(x,t)$ be the upper control matrix of $L_\rho(x,t)$ as in Theorem \ref{t2.1}. Then $\lm_p(\ol{\mathscr L}^\ep_\rho)<0$ when $0<\ep\ll1$. Besides, as $f$ is strictly subhomogeneous, we have that, by \qq{3.1},
 \bess
\rho U_{it}&=&d_i\dd\int_\oo J_i(x,y)\rho U_i(y,t)\dy-d^*_i(x)\rho U_i
+\rho f_i(x,U)\\
&\le &d_i\dd\int_\oo J_i(x,y)\rho U_i(y,t)\dy-d^*_i(x)\rho U_i
+f_i(x,\rho U)\\
&=&\rho d_i\dd\int_\oo J_i(x,y) U_i(y,t)\dy-\rho d^*_i(x) U_i
+\rho\sum_{k=1}^m b^\rho_{ik}(x,t) U_k\\
&=&\rho d_i\dd\int_\oo J_i(x,y) U_i(y,t)\dy
+\rho\sum_{k=1}^m \ell^\rho_{ik}(x,t) U_k\\
&\le&\rho d_i\dd\int_\oo J_i(x,y)U_i(y,t)\dy
+\rho\sum_{k=1}^m\bar\ell^{\rho,\ep}_{ik}(x,t)U_k,\;\;\; (x,t)\in\ol Q_T,
 \eess
i.e.,
 \bess
 U_{it}\le d_i\dd\int_\oo J_i(x,y)U_i(y,t)\dy
+\sum_{k=1}^m\bar \ell^{\rho,\ep}_{ik}(x,t)U_k,\;\;\; (x,t)\in\ol Q_T,\;\, i=1,\cdots,m.
 \eess
Since $U(x,t)$ is positive in $\ol Q_T$, it follows by Lemma \ref{le2.2} that $\lm_p(\ol{\mathscr L}^\ep_\rho)\ge 0$, and we have a contradiction. So, the system \eqref{3.1} has no continuous positive solution.

{\it Step 2: The proof of \qq{3.7}.} Let $\ol L^\ep(x,t)$ be the upper control matrix of $L(x,t)$ as in Theorem \ref{t2.1} with $0<\ep\ll1 $. Since $\lm(\mathscr{L})< 0$, there exists $0<\ep\ll1$ such that $\lm_p(\ol{\mathscr L}^\ep)<0$. Let $0\ll\phi\in [C(\ol Q_T)]^m$, with $\|\phi\|_{L^\yy(Q_T)}=1$, be a positive eigenfunction corresponding to $\lm_p(\ol{\mathscr L}^\ep)$, and $0<\rho\ll 1$ being determined later. Then
 \bes
&&-\rho d^*_i(x)\phi_i(x,t)+f_i(x,t,\rho\phi(x,t))\nm\\
&=&-\rho d^*_i(x)\phi_i(x,t)+\rho\sum_{k=1}^m h^\rho_{ik}(x,t)\phi_k(x,t)\nm\\
 &=&\rho\sum_{k=1}^m\kk\{\bar\ell^\ep_{ik}(x,t)+
 [\ell_{ik}(x,t)-\bar\ell^\ep_{ik}(x,t)]+[h^\rho_{ik}(x,t)-b_{ik}(x,t)]\rr\}\phi_k(x,t)\nm\\
&\le&\rho\sum_{k=1}^m\bar\ell^\ep_{ik}(x,t)\phi_k(x,t)+
\rho\sum_{k=1}^m[h^\rho_{ik}(x,t)-b_{ik}(x,t)]\phi_k(x,t),
 \lbl{z3.15}\ees
where
 \[h_{ik}^\rho(x,t)=\int_0^1\partial_{u_k}f_i(x,t,s\rho\phi(x,t)){\rm d}s,\]
and $h_{ik}^\rho\to b_{ik}$ uniformly in $\ol Q_T$ as $\rho\to 0$. Choose $0<\rho_1\ll 1$ such that, for all $0<\rho\le\rho_1$, there holds:
 \[\sum_{k=1}^m[h^\rho_{ik}(x,t)-b_{ik}(x,t)]\phi_k(x,t)<-\frac 12\lm_p(\ol{\mathscr L}^\ep)\phi_i(x,t),
 \;\;(x,t)\in\ol Q_T,\; i=1,\cdots,m.\]
This combines with \qq{z3.15} allows us to derive
 \bes
 &&d_i\dd\int_\oo J_i(x,y)\rho\phi_i(y,t)\dy-d^*_i(x)\rho\phi_i+f_i(x,t,\rho\phi)-\rho\phi_{it}\nm\\
 &<&\rho d_i\dd\int_\oo J_i(x,y)\phi_i(y,t)\dy+\rho\sum_{k=1}^m\bar\ell^\ep_{ik}(x,t)\phi_k(x,t)-\rho\phi_{it}-\frac\rho 2\lm_p(\ol{\mathscr L}^\ep)\phi_i\nm\\
&=&\frac1 2\lm_p(\ol{\mathscr L}^\ep)\rho\phi_i,\;\;\; (x,t)\in\ol Q_T.
 \lbl{z3.16}\ees
Certainly, the above inequality holds in $\boo\times(0,+\yy)$ since all functions are time-periodic with period $T$.

Set $\sigma=-\frac 14\lm_p(\ol{\mathscr L}^\ep)$ and $V(x,t)=\rho{\rm e}^{-\sigma t}\phi(x,t)$.  As $\sigma>0$, we have $\rho{\rm e}^{-\sigma t}\le\rho$ for all $t\ge 0$. Of course, $V(x,t)$ satisfies, by \qq{z3.16},
 \bes
 &&d_i\dd\int_\oo J_i(x,y)V_i(y,t)\dy-d^*_i(x)V_i(x,t)+f_i(x,t,V)-V_{it}\nm\\[1mm]
 &=&d_i\dd\int_\oo J_i(x,y)V_i(y,t)\dy-d^*_i(x)V_i(x,t)+f_i(x,t,V)-\rho{\rm e}^{-\sigma t}\phi_{it}+\rho\sigma{\rm e}^{-\sigma t}\phi_i\nm\\
 &\le&\frac1 2\lm_p(\ol{\mathscr L}^\ep)\rho{\rm e}^{-\sigma t}\phi_i+\rho\sigma{\rm e}^{-\sigma t}\phi_i\nm\\
 &<&0,\;\; x\in\boo,\;t\ge 0,\, i=1,\cdots,m.\lbl{z3.17}\ees

Take $\gamma >1$ such that $u_0(x)\le \gamma\rho\phi(x,0)$ in $\boo$.
Then, by the comparison principle,
 \bes
 u(x,t;u_0)\le u(x,t;\gamma\rho\phi(x,0)),\;\;x\in\boo,\;t\ge 0.
   \lbl{3.10}\ees
Set $\bar u(x,t)=\gamma V(x,t)$. Noticing that $f$ is strictly subhomogeneous, i.e., $f(x,t,\gamma u)<\gamma f(x,t,u)$ for all $(x,t)\in\ol Q_T$, $u \gg 0$ and $\gamma>1$. Making use of \qq{z3.17} we have
 \bess
 &&d_i\dd\int_\oo J_i(x,y)\bar u_i(y,t)\dy-d^*_i(x)\bar u_i(x,t)+f_i(x,t,\bar u(x,t))-\bar u_{it}(x,t)\\
 &\le&d_i\dd\int_\oo J_i(x,y)\bar u_i(y,t)\dy-d^*_i(x)\bar u_i(x,t)+\gamma f_i(x,t,V(x,t))-\bar u_{it}(x,t)\\
 &=&\gamma\kk(d_i\dd\int_\oo J_i(x,y)V_i(y,t)\dy-d^*_i(x)V_i(x,t)+f_i(x,t,V)-V_{it}\rr)\\
 &<&0,\;\;\;x\in\boo,\;t\ge 0,\, i=1,\cdots,m.
 \eess
Since $\bar u(x,0)=\gamma \rho\phi(x,0)$, the comparison principle yields
 \[ u(x,t;\gamma\rho\phi(x,0))\le\bar u(x,t)=\gamma\rho{\rm e}^{-\sigma t}\phi(x,t),\;\;x\in\boo,\;t\ge 0.\]
This combines with \qq{3.10} implies \qq{3.7}.

(3)  Assume that $\lm(\mathscr{L})=0$.

{\it Step 1: The nonexistence of continuous positive solution of \qq{3.1}}. Assume on the contrary that $U\in[C(\ol Q_T)]^m$ is a positive solution of \qq{3.1}. Since $f$ is strongly subhomogeneous, we have $\rho f_i(x,t,U)\ll f_i(x,t,\rho U)$ in $\bqq$ for all $\rho\in (0,1)$. Fix $\rho_0$ small enough. Then there exists $\sigma>0$ such that $\rho_0 f_i(x,t,U)\le f_i(x,t,\rho_0 U)-\sigma\rho_0 U_i$ in $\ol Q_T$. Therefore,
  \bes
  \rho f_i(x,t,U)\le \frac{\rho}{\rho_0} f_i(x,t,\rho_0 U)-\sigma\rho U_i \le  f_i(x,t,\rho U)-\sigma\rho U_i,\;\;\forall\,\rho \in (0,\rho_0).
  \lbl{3.11}\ees
For any given $\delta>0$ small enough, there exists $\rho \in (0,\rho_0)$ small enough such that
 \[f_i(x,t,\rho U)=f_i(x,t,\rho U)-f_i(x,t,0) \le \rho \sum_{k=1}^m(b_{ik}(x,t) +\delta) U_k,\;\; (x,t)\in\ol Q_T.\]
It then follows that
	\bes
\rho U_{it}&=&d_i\dd\int_\oo J_i(x,y)\rho U_i(y,t)\dy-d^*_i(x)\rho U_i+\rho f_i(x,t,U)\nm\\
&\le& d_i\dd\int_\oo J_i(x,y)\rho U_i(y,t)\dy-d^*_i(x)\rho U_i
+f_i(x,t,\rho U)-\sigma\rho U_i\nm\\
&\le&  d_i\dd\int_\oo J_i(x,y)\rho U_i(y,t)\dy-d^*_i(x)\rho U_i+\rho\sum_{k=1}^m (b_{ik}(x,t)+\delta) U_k-\sigma\rho U_i,\;\; (x,t)\in\ol Q_T.\qquad
	\lbl{3.12}\ees
Consequently $\lambda(\widetilde{\mathscr L}_{\delta})\ge\sigma$ by Lemma  \ref{le2.2}, where $\widetilde{L}_{\delta}(x,t)=(\ell_{ik}(x,t)+\delta)_{m \times m}$ and $\widetilde{\mathscr L}_{\delta}$ is defined by \qq{1.2} with $L(x,t)$ replaced by $\widetilde{L}_{\delta}(x,t)$. Letting $\delta \rightarrow 0$, we can obtain that $\lm(\mathscr{L})\ge\sigma$. This contradiction indicates that \eqref{3.1} has no positive solution in $[C(\ol Q_T)]^m$.

{\it Step 2: The proof of \qq{3.8}}.

(i) \ud{The upper control function of $u(x,t;u_0)$}. We can find a constant $\gamma >1$ such that $u_0(x)\le\gamma\ol U(x,0)$ in $\boo$. Then
 \bes
 u(x,t; u_0)\le u(x,t; \gamma\ol U(x,0)),\;\;x\in\boo,\; t\ge 0.
 \lbl{3.13}\ees

(ii) \ud{The upper control function of $u(x,t; \gamma\ol U(x,0))$}. Let $\ol L^\ep(x,t)$ be the upper control matrix of $L(x,t)$ given in Theorem \ref{t2.1}. We will use $\ol L^\ep(x,t)$ to construct an upper control problem for $u(x,t; \gamma\ol U)$ to control it. Set
 \[f^\ep_i(x,t,u)=[\bar \ell_{ii}^\ep(x,t)-\ell_{ii}(x,t)]u_i+f_i(x,t,u).\]
Then $f^\ep(x,t,u)=(f^\ep_1(x,t,u),\cdots,f^\ep_m(x,t,u))$ is also cooperative, strongly subhomogeneous, and $\kk(\partial_{u_k}f^\ep_i(x_0,t_0,u)\rr)_{m\times m}$ is irreducible for all $u \ge 0$.

Recalling that $0\ll \ol U\in [C(\ol Q_T)]^m$ satisfies \qq{3.5}. In consideration of $\gamma>1$ and the strongly subhomogeneity of $f$, there exists $\ep_0>0$ small enough such that $f_i(x,t,\gamma\ol U)\le\gamma f_i(x,t,\ol U)-\ep\gamma\ol U_i$ for all $0<\ep\le\ep_0$ (refer to the derivation of \qq{3.11}). As $\bar \ell_{ii}^\ep(x,t)-\ell_{ii}(x,t)\le\ep$, we have
 \[f^\ep_i(x,t,\gamma\ol U)\le f_i(x,t,\gamma\ol U)+\ep\gamma \ol U_i\le\gamma f_i(x,t,\ol U),\;\;\;i=1,\cdots,m.\]
This together with \qq{3.5} implies that, for $0<\ep<\ep_0$,
 \bes
\gamma\ol U_{it}&\ge&d_i\dd\int_\oo J_i(x,y)\gamma\ol U_i(y,t)\dy
 -\gamma d^*_i(x)\ol U_i+\gamma f_i(x,t,\ol U)\nm\\
&\ge&d_i\dd\int_\oo J_i(x,y)\gamma\ol U_i(y,t)\dy-\gamma d^*_i(x)\ol U_i
+f^\ep_i(x,t,\gamma\ol U),\;\; (x,t)\in\ol Q_T\lbl{3.15}
 \ees
and $\gamma\ol U_i(x,0)\ge\gamma\ol U_i(x,T)$ in $\boo$ for all $i\in\sss$.

Owing to
 \bess
  b^\ep_{ii}(x,t)&:=&\partial_{u_i}f^\ep_i(x,t,0)=\bar \ell_{ii}^\ep(x,t)-\ell_{ii}(x,t)
 +\partial_{u_i}f_i(x,t,0)=\bar\ell_{ii}^\ep(x,t)+d_i^*(x),\\
  b^\ep_{ik}(x,t)&:=&\partial_{u_k}f^\ep_i(x,t,0)
 =b_{ik}(x,t)=\ell_{ik}(x,t),\;\;k\not=i,
 \eess
we have the corresponding
 \[\ell^\ep_{ii}(x,t)=b^\ep_{ii}(x,t)-d_i^*(x)=\bar \ell_{ii}^\ep(x,t),\;\;
 \ell^\ep_{ik}(x,t)=b^\ep_{ik}(x,t)=\bar\ell^\ep_{ik}(x,t),\;\;k\not=i. \]
Set $L^\ep(x,t)=(\ell^\ep_{ik}(x,t))_{m \times m}$. Then $L^\ep(x,t)=\ol L^\ep(x,t)$, and $\lm_p({\mathscr L}^\ep)=\lm_p(\ol{\mathscr L}^\ep)>\lm({\mathscr L})=0$ by Theorem \ref{t2.1}, where ${\mathscr L}^\ep$ is defined by \qq{1.2} with $L(x,t)$ replaced by $ L^\ep(x,t)$.

Taking advantage of the facts that $\lm_p({\mathscr L}^\ep)>0$ and \qq{3.15} we have that, by the conclusion (1), the problem \bes\begin{cases} \label{3.14}
U^\ep_{it}=d_i\dd\int_\oo J_i(x,y)U^\ep_i(y,t)\dy-d^*_i(x)U^\ep_i+f^\ep_i(x,t,U^\ep)=0,\;&(x,t)\in\ol Q_T,\\[1mm]
 U^\ep_i(x,0)=U^\ep_i(x,T),&x\in\boo,\\
i=1,\cdots,m
 \end{cases}\ees
has a unique positive solution $U^\ep\in [C(\ol Q_T)]^m$ satisfying   $U^\ep\le\gamma\ol U$ in $\ol Q_T$, and the solution $u^\ep(x,t;\gamma\ol U(x,0))$ of
 \bess\begin{cases}
u^\ep_{it}=d_i\dd\int_\oo J_i(x,y)u^\ep_i(y,t)\dy-d^*_i(x)u^\ep_i+f^\ep_i(x,t, u^\ep),&x\in\boo, \; t>0,\\[3mm]
u^\ep_i(x,0)=\gamma\ol U_i(x,0), &x\in\boo,\\[1mm]
 i=1,\cdots,m
 \end{cases}\eess
satisfies
 \bes
 \lim_{n\to+\yy}u^\ep(x,t+nT;\gamma\ol U(x,0))=U^\ep(x,t) \;\;\text{ uniformly in } \; \ol Q_T.\lbl{3.16}\ees
Moreover,
 \bes
 u(x,t;\gamma\ol U(x,0))\le u^\ep(x,t;\gamma\ol U(x,0)),\;\;x\in\boo,\; t\ge 0
 \lbl{3.17}\ees
by the comparison principle as $f_i\le f^\ep_i$.

(iii) \ud{Estimate of the limit $\lim_{\ep\to 0^+}U^\ep(x,t)$}. Noticing that $f^\ep$ is increasing in $\ep>0$, so is $U^\ep$ by the comparison principle. Therefore, the limit
 \[\lim_{\ep\to 0^+}U^\ep(x,t)=U(x,t)\ge 0\]
exists and is a nonnegative solution of \qq{3.1}. By Theorem \ref{th3.3},
either $U\equiv 0$ in $\ol Q_T$, or $U_i> 0$ in $\ol Q_T$ and $\inf_\oo U_i>0$ for all $i\in\sss$.

(iv) \ud{Prove $U\equiv 0$ in $\ol Q_T$}. If $ U \gg 0$ in $\ol Q_T$,
then $\theta_i:=\inf_{\ol Q_T}U_i(x,t)>0$ for all $i\in\sss$. It is deduced that
 \[U^\ep_i(x,t)\ge U_i(x,t)\ge\theta_i,\;\;\forall\,(x,t)\in\ol Q_T,\;i\in\sss,\]
where $U^\ep\in [C(\ol Q_T)]^m$ is the unique positive solution of \qq{3.14}. As $U^\ep\le\gamma\ol U$ in $\ol Q_T$, we can find $\beta_i>0$ such that $U_i^\ep\le\beta_i$ in $\ol Q_T$. Take $B=\prod_{i=1}^m[\theta_i,\beta_i]$. 
Then $U^\ep(x,t)\in B$ for all $(x,t)\in\ol Q_T$. Noticing that $f^\ep(x,t,u)$ and $f(x,t,u)$ are  strongly subhomogeneous and
 \[\lim_{\ep\to 0^+}f^\ep(x,t,u)=f(x,t,u)\;\; \mbox{uniformly in }\ol Q_T\times B.\]
If we denote $f^0_i=f_i$, then $f^\ep(x,t,u)$ is continuous in $\ep\ge 0$, $(x,t)\in\ol Q_T$ and $u\ge 0$; and
  \[f^\ep(x,t,\rho u)\gg\rho f^\ep(x,t,u),\;\;\forall\,\ep\ge 0, \; (x,t)\in\ol Q_T, \; u\gg 0, \; \rho\in (0,1).\]

For the given $0<\rho\ll 1$, the function
 \[h_i(x,t,w,\ep)=f_i^\ep(x,t, \rho w)-\rho f_i^\ep(x,t,w)\]
is continuous and positive in $\ol Q_T\times B\times[0,1]$. There exists $\gamma_i>0$ such that $h_i(x,t,w,\ep)\ge\gamma_i$ for all $(x,t,w,\ep)\in\ol Q_T\times B\times[0,1]$. Denote $U^0=U$. Then $U^\ep(x,t)\in B$ for all $(x,t)\in\ol Q_T$ and $0\le\ep\le 1$. So  $h_i(x,t,U^\ep(x,t),\ep)\ge\gamma_i$ in $\ol Q_T$ for all $0\le\ep\le 1$. Then there exists $\sigma_i>0$ such that 
$h_i(x,t,U^\ep(x,t),\ep)\ge\sigma_i\rho U^\ep_i(x,t)$ in $\ol Q_T$ 
for all $0\le\ep\le 1$. Take $\sigma=\min\{\sigma_1,\cdots,\sigma_m\}$.
Then we have $h_i(x,t,U^\ep(x,t),\ep)\ge\sigma\rho U^\ep_i(x,t)$ in $\ol Q_T$ for all $0\le\ep\le 1$, i.e.,
 \[\rho f_i^\ep(x,t,U^\ep(x,t))\le f_i^\ep(x,t,\rho U^\ep(x,t))-\sigma\rho U^\ep_i(x,t),\;\;\forall\,(x,t)\in\ol Q_T,\;0\le\ep\le 1.\]
Using $f^\ep_i$ instead of $f_i$, similar to the discussion in Step 1 we can show that \qq{3.12} holds. Then $\lambda_p(\widetilde{\mathscr L}^\ep_\delta)\ge\sigma$ for all $0\le\ep\ll 1$, where $\widetilde{\mathscr L}^\ep_\delta$ is defined by \eqref{1.2} with $\ell_{ik}(x,t)$ replaced by $\bar\ell^\ep_{ik}(x,t)+\delta$. Letting $\delta\rightarrow 0$, we can obtain that $\lm_p(\ol{\mathscr L}^\ep)=\lm_p(\widetilde {\mathscr L}^\ep)\ge\sigma$ for all $0\le\ep\ll 1$. This contradicts the fact that $\lim_{\ep\to 0^+}\lm_p(\ol{\mathscr L}^\ep)=\lm(\mathscr{L})=0$. Therefore, $U\equiv 0$ in $\ol Q_T$. This together with \qq{3.16}, \qq{3.17} and \qq{3.13} derives \qq{3.8}. The proof is complete.\zzz
\end{proof}

\begin{remark} In Theorem {\rm\ref{th3.6}(1)}, the existence of upper solution $\ol U$ of \qq{3.1} is necessary to obtain the positive solutions of \qq{3.1}. The role of strict subhomogeneity conditions is manifested in the following aspects:
\begin{enumerate}[$(1)$]\vspace{-2mm}
\item If $\ud U$ is a lower solution of \qq{3.1}, then for any $0<\rho<1$, $\rho\ud U$ is still a lower solution of \qq{3.1}; If $\ol U$ is an upper solution of \qq{3.1}, then for any $\rho>1$, $\rho\ol U$ is still an upper  solution of \qq{3.1}. This provides convenience for constructing appropriate upper and lower solutions and estimating solutions.
\item The strict homogeneous condition plays a crucial role in proving the uniqueness of positive equilibrium solutions.
\item Especially, when $\lm(\mathscr{L})<0$, the nonexistence of the positive solution can be proved by using strict subhomogeneity. At the same time, the strict subhomogeneous condition plays an important role in constructing the upper solution of the initial value problem by using the positive eigenfunction corresponding to the principal eigenvalue of the control problem.
\item The case $\lm(\mathscr{L})=0$ indicates that the zero solution is degenerate, which is a critical case.	Due to the loss of compactness for the solution maps, it is challenging to prove the nonexistence of positive solutions and the stability of the zero solution without additional conditions. Here, we prove these conclusions under a strong subhomogeneous condition.\zzz
\end{enumerate}
\end{remark}

\begin{remark}\lbl{r3.2} In Theorem {\rm\ref{th3.6}}, it is assumed that $f(x,t,u)$ is strictly subhomogeneous with respect to all $u\gg 0$, i.e., condition {\bf(F3)} holds. Especially, in Theorem {\rm\ref{th3.6}(3)}, it is required that $f$ is strongly subhomogeneous, i.e., $f(x,t,\rho u) \gg \rho f(x,t,u)$ for all $(x,t)\in\ol Q_T$, $u \gg 0$ and $\rho\in (0,1)$. However, some models do not meet these conditions $($for example, the following system \qq{4.11}$)$. Based on the proof of Theorem {\rm\ref{th3.6}} we can see that these can be weakened.
\begin{enumerate}[$(1)$]\vspace{-2mm}
\item In Theorem {\rm\ref{th3.6}(1)}, the conditions {\bf(F3)} and \qq{3.5} can be replaced by the following assumptions:
\begin{enumerate}\vspace{-2mm}
\item [{\rm(1a)}] $\ol U$ is a strict upper solution, that is, for all $i\in\sss$, there holds:
  \bess\begin{cases}
  \ol U_{it}>d_i\dd\int_\oo J_i(x,y)\ol U_i(y,t)\dy-d^*_i(x)\ol U_i
+f_i(x,t,\ol U),&(x,t)\in\ol Q_T,\\[2mm]
 \ol U_i(x,0)\ge \ol U_i(x,T),&x\in\boo;\end{cases}\eess
\item[{\rm(1b)}] $f(x,t,u)$ is subhomogeneous with respect to $u\ge 0$,  that is, $f(x,t,\delta u)\ge\delta f(x,t,u)$  for all $(x,t)\in\ol Q_T, u \ge 0$ and $\delta\in (0,1)$;\vskip 2pt
\item[{\rm(1c)}] $f(x,t,u)$ is strictly subhomogeneous with respect to $u$ with $0\ll u\le\overline{\phi}$, that is, $f(x,t,\delta u)>\delta f(x,t,u)$ for all $(x,t)\in\ol Q_T, 0\ll u \le  \overline{\phi}(x,t)$ and $\delta\in (0,1)$.
  \end{enumerate}\www
These show that when \qq{3.1} has a strict upper solution $\ol U$, it only needs $f(x,t,u)$ to be subhomogeneous with respect to all $u\gg 0$ and strictly homogeneous with respect to $0\ll u\le\overline{\phi}$, without the need for $f(x,t,u)$ to be strictly homogeneous with respect to all $u\gg 0$.\vskip 2pt
\item In Theorem {\rm\ref{th3.6}(2)}, the condition {\bf(F3)} can be replaced by that $f(x,t,u)$ is subhomogeneous with respect to all $u$.\vskip 2pt
\item In Theorem {\rm\ref{th3.6}(3)}, the conditions \qq{3.5} and that $f(x,t,u)$ is strongly subhomogeneous with respect to all $u\gg 0$ can be replaced by the following assumptions{\rm:}

\hspace{3mm} The problem \qq{3.1} has a strongly strict upper solution $\ol U$, that is, for all $i\in\sss$, there holds:
  \bess\begin{cases}
  \ol U_{it}\gg d_i\dd\int_\oo J_i(x,y)\ol U_i(y,t)\dy-d^*_i(x)\ol U_i
+f_i(x,t,\ol U),&(x,t)\in\ol Q_T,\\[2mm]
  \ol U_i(x,0)\ge\ol U_i(x,T),&x\in\boo;\end{cases}\zzz\eess
 $f(x,t,u)$ is subhomogeneous with respect to all $u\gg 0$, and $f(x,t,u)$ is strongly subhomogeneous with respect to $0\ll u \le \rho_0\overline{\phi}$ for some $0\le\rho_0\ll 1$, i.e., $f(x,t,\delta u)\gg \delta f(x,t,u)$ for all $(x,t)\in\ol Q_T, \;0\ll u \le\rho_0\overline{\phi}(x,t)$ and $\delta\in (0,1)$.\www
\end{enumerate}
\end{remark}

Before concluding this section, we will present the conclusion regarding the logistic equation equation, as it is a highly anticipated equation.
We study the following initial value problems
 \bes\begin{cases}
 u_t=d\dd\int_{\Omega}J(x,y) u(y,t){\rm d}y-d^*(x)u+uf(x,t,u),\;\;&x\in\ol\oo,\;t>0\\[1mm]
 u(x,0)=u_0(x,t)>0,&x\in\ol\oo,\zzz
  \end{cases}\lbl{3.18}\ees
where, either $d^*(x)=d$ or $d^*(x)=dj(x)$. Assume that $f\in C^{0,0,1}(\ol Q_T\times[0,+\yy))$ is $T$-periodic in time $t$, $f(x,t,u)$ is strictly decreasing in $u\ge 0$ for $(x,t)\in\ol Q_T$, and that there exists $M>0$ such that $f(x,t,u)<0$ for all $(x,t)\in\ol Q_T$ and $u>M$. Then  $uf(x,t,u)$ is strictly subhomogeneous with respect to $u\gg 0$, and $M$ is an upper solution of \eqref{3.18} if one of the following holds:
 \begin{enumerate}[$(A)$]\vspace{-2mm}
\item either $d^*(x)=d$, or $d^*(x)=dj(x)$ and $J(x,y)$ is symmetric in $x,y$, i.e., $J(x,y)=J(y,x)$;\vskip 2pt
\item $J(x,y)$ is not symmetric in $x,y$, but $d\int_{\Omega}J(x,y){\rm d}y-d^*(x)+f(x,t,M)\le 0$ in $\ol Q_T$.
\end{enumerate}\www

Set $\ell(x,t)=f(x,t,0)-d^*(x)$. Then the conclusions of Theorem \ref{t2.1} holds for $\ell(x,t)$. Let $\lm({\mathscr L})$ be the generalized principal eigenvalue of ${\mathscr L}$. Making use of Theorem \ref{th3.6}, we have the following theorem.\www

\begin{theo}\lbl{t3.3}\, Assume that one of the above conditions $(A)$ and $(B)$ holds. Let $u(x,t;u_0)$ be unique solution of \qq{3.18}. \www
  \begin{enumerate}
\item[{\rm(1)}]\, If $\lm({\mathscr L})>0$, then the time-periodic problem
 \bes\begin{cases}
 U_t=d\dd\int_{\Omega}J(x,y)U(y){\rm d}y-d^*(x)U+Uf(x,U)=0,&(x,t)\in\ol Q_T,\\[1mm]
 U(x,0)=U(x,T),&x\in\boo
 \end{cases}
  \lbl{3.19}\ees
has a unique positive solution $U\in C(\ol Q_T)$ and $U(x,t)$ is globally asymptotically stable in the time-periodic sense, i.e., $\dd\lim_{n\to+\yy}u(x,t+nT;u_0)=U(x,t)$ uniformly in $\ol Q_T$.
\item[{\rm(2)}]\, If $\lm({\mathscr L})\le 0$, then the problem \qq{3.19} has no continuous positive solution, and $\dd\lim_{n\to+\yy}u(x,t+nT;u_0)=0$ uniformly  in $\ol Q_T$.
Especially, $u(x,t;u_0)$ converges exponentially to zero when $\lm({\mathscr L})<0$.
  \end{enumerate}\www
\end{theo}

\section{A West Nile virus model}\lbl{S5}
\setcounter{equation}{0} {\setlength\arraycolsep{2pt}

In this section we study a  West Nile virus model with nonlocal dispersal, standard incidences and effects of recovered hosts.

Let $H_u(x,t)$, $H_i(x,t)$, $V_u(x,t)$ and $V_i(x,t)$ be the densities of susceptible birds, infected birds, susceptible mosquitoes, and infected mosquitoes at location $x$ and time $t$, respectively. Then $H(x,t)=H_u(x,t)+H_i(x,t)$ and $V(x,t)=V_u(x,t)+V_i(x,t)$ are the total densities of birds and mosquitoes. Recently, the authors of this paper (\cite{WZhang24}) proposed and studied the following West Nile (WN) virus model in the spatiotemporal heterogeneity with random diffusion processes and general boundary conditions
\bes\left\{\begin{aligned}
&H_{u,t}=\nabla\cdot d_1\nabla H_u+a_1(H_u+H_i)-b_1H_u-c_1(H_u+H_i)H_u
-\mu_1H_uV_i,&&x\in\oo,\; t>0,\\
&H_{it}=\nabla\cdot d_1\nabla H_i+\mu_1H_uV_i-b_1H_i-c_1(H_u+H_i)H_i,\!\!&&x\in\oo,\; t>0,\\
&V_{u,t}=\nabla\cdot d_2\nabla V_u+a_2(V_u+V_i)-b_2V_u
-c_2(V_u+V_i)V_u-\mu_2H_iV_u,&&x\in\oo,\; t>0,\\
&V_{it}=\nabla\cdot d_2\nabla V_i+\mu_2H_iV_u-b_2V_i-c_2(V_u+V_i)V_i,&&x\in\oo,\; t>0,\\
&\alpha_1\partial\nu H_u+\beta_1H_u=\alpha_1\partial_\nu H_i+\beta_1H_i=
\alpha_2\partial_\nu V_u+\beta_2V_u=\alpha_2\partial_\nu V_i+\beta_2V_i=0,&&x\in\partial\oo,\; t>0,\\
&\big(H_u, H_i, V_u, V_i\big)=\big(H_{u0}(x,t), H_{i0}(x,t), V_{u0}(x,t), V_{i0}(x,t)),&&x\in\oo,\; t=0,\vspace{-2mm}
   \end{aligned}\rr.\quad\lbl{4.1}\ees
where $\nu$ is the outward normal vector of $\partial\Omega$, all coefficients are continuous functions of $(x,t)$ and $T$-periodic in time $t$; $d_1$ and $d_2$ are the diffusion rates of birds and mosquitoes, $a_1$ and $a_2$ ($b_1$ and $b_2$)  are the birth (death) rates of susceptible birds and mosquitoes, $c_1$ and $c_2$ are the loss rates of birds and mosquitoes population due to environmental crowding, $\mu_1$ and $\mu_2$ are the transmission rates of uninfected birds and uninfected mosquitoes, respectively; $\alpha_k, \beta_k\ge 0$ and $\alpha_k+\beta_k>0$ in $\overline\oo\times[0, T]$, $k=1,2$; they are all H\"{o}lder continuous, nonnegative and nontrivial; $b_k, \mu_k\ge 0$, $d_k, a_k, c_k>0$ in $\overline\oo\times[0, T]$. By the upper and lower solutions method, it was proved that the problem \qq{4.1} has a positive time periodic solution if and only if birds  and mosquitoes persist and the basic reproduction ratio is greater than one, and moreover the positive time periodic solution is unique and globally asymptotically stable when it exists.

In the model \qq{4.1}, the diffusion is random (local) diffusion, and effects of recovered hosts is neglected, and terms $H_uV_i$ and $H_iV_u$ represent bilinear incidences. However, because the movements of birds and mosquitoes are mainly through flight,  it is more reasonable to use nonlocal dispersal (long-distance diffusion) instead of local diffusion. Moreover, we know that the probability of an infected mosquito biting healthy birds and the probability of a healthy mosquito biting infected birds are $\frac{H_u}{H_u+H_i}$ and $\frac{H_i}{H_u+H_i}$, respectively. It is more reasonable to replace bilinear incidences $H_uV_i$ and $H_iV_u$ with standard incidences $\frac{H_u}{H_u+H_i}$ and $\frac{H_i}{H_u+H_i}$, and consider the effects of recovered hosts (after the virus of the infected host disappears, the infected host becomes a susceptible host again). Under the above considerations, the corresponding West Nile virus model with nonlocal dispersal, standard incidences and effects of recovered hosts are the following
 \bes\left\{\begin{aligned}
&H_{ut}=d_1\!\dd\int_\oo\! J_1(x,y)H_u (y,t) \dy-d^*_1(x)H_u+a_1(x,t)(H_u+H_i)-b_1(x,t)H_u\\[0.5mm]
 &\hspace{17mm}-c_1(x,t)(H_u+H_i)H_u
 -\mu_1(x,t)\dd\frac{H_u}{H_u+H_i}V_i+\gamma(x,t)H_i,&&x\in\oo,\; t>0,\\[1mm]
 &H_{it}=d_1\!\dd\int_\oo\! J_1(x,y)H_i (y,t) \dy-d^*_1(x) H_i+\mu_1(x,t)\dd\frac{H_u}{H_u+H_i}V_i\\[0.5mm]
 &\hspace{17mm}-b_1(x,t)H_i-c_1(x,t)(H_u+H_i)H_i-\gamma(x,t)H_i,\!\!&&x\in\oo,\; t>0,\\[1mm]
 &V_{ut}=d_2\!\dd\int_\oo\! J_2(x,y)V_u (y,t) \dy-d^*_2(x) V_u+a_2(x,t)(V_u+V_i)\\[0.5mm]
 &\hspace{17mm}-b_2(x,t)V_u-c_2(x,t)(V_u+V_i)V_u-\mu_2(x,t)\dd\frac{H_i}{H_u+H_i}V_u,&&x\in\oo,\; t>0,\\[1mm]
 &V_{it}=d_2\!\dd\int_\oo\! J_2(x,y)V_i (y,t) \dy-d^*_2(x) V_i+\mu_2(x,t)\dd\frac{H_i}{H_u+H_i}V_u-b_2(x,t)V_i\\
	&\hspace{17mm}-c_2(x,t)(V_u+V_i)V_i,&&x\in\oo,\; t>0,\\
	&\big(H_u, H_i, V_u, V_i\big)=\big(H_{u0}(x), H_{i0}(x), V_{u0}(x), V_{i0}(x)\big)>0,&&x\in\oo,\; t=0,\vspace{-2mm}
 \end{aligned}\rr.\lbl{4.2}\ees
where $\gamma(x,t)\ge 0$ is the recovery rate from the loss of virus in infected hosts to uninfected hosts; $d_k$ is positive constant, $J_k$ satisfies the condition {\bf(J)}, and $d^*_k(x)$ is defined by the manner {\bf(D)}; $\oo\subset\mathbb{R}^N$ is a bounded domain whose boundary has zero measure. Same as in \qq{4.1}, all coefficients are continuous functions of $(x,t)$ and $T$-periodic in time $t$.

Throughout this section, we always assume that\vspace{-2mm}
 \begin{quote}$\bullet$\; Initial functions are continuous in $\ol\oo$; and $a_k, c_k>0$, $b_k, \mu_k\ge 0$ in $\ol Q_T$, and there exists $x_0\in\boo$ such that $\mu_k(x_0,t)>0$ in $[0, T]$, $k=1,2$. \vspace{-2mm}\end{quote}

The corresponding time-periodic problem of \qq{4.2} is
 \bes\left\{\!\begin{aligned}
&\mathsf{H}_{ut}=d_1\!\int_\oo\! J_1(x,y)\mathsf{H}_u(y,t)\dy-[d_1^*(x)+b_1(x,t)]\mathsf{H}_u+
a_1(x,t)(\mathsf{H}_u+\mathsf{H}_i)\\
&\hspace{17mm}-c_1(x,t)(\mathsf{H}_u+\mathsf{H}_i)\mathsf{H}_u
-\mu_1(x,t)\dd\frac{\mathsf{H}_u}{\mathsf{H}_u+\mathsf{H}_i}\mathsf{V}_i
+\gamma(x,t)\mathsf{H}_i,\!&&(x,t)\in\ol Q_T,\\
&\mathsf{H}_{it}=d_1\!\int_\oo\! J_1(x,y)\mathsf{H}_i(y,t)\dy-[d_1^*(x)+b_1(x,t)]\mathsf{H}_i+
\mu_1(x,t)\dd\frac{\mathsf{H}_u}{\mathsf{H}_u+\mathsf{H}_i}\mathsf{V}_i\\
&\hspace{17mm}-c_1(x,t)(\mathsf{H}_u+\mathsf{H}_i)\mathsf{H}_i
-\gamma(x,t)\mathsf{H}_i,\!&&(x,t)\in\ol Q_T,\\
&\mathsf{V}_{ut}=d_2\!\int_\oo\! J_2(x,y)\mathsf{V}_u(y,t)\dy-[d_2^*(x)+b_2(x,t)]\mathsf{V}_u+
a_2(x,t)(\mathsf{V}_u+\mathsf{V}_i)\\
&\hspace{17mm}-c_2(x,t)(\mathsf{V}_u+\mathsf{V}_i)\mathsf{V}_u
-\mu_2(x,t)\dd\frac{\mathsf{H}_i}{\mathsf{H}_u+\mathsf{H}_i}\mathsf{V}_u,\!&&(x,t)\in\ol Q_T,\\
&\mathsf{V}_{it}=d_2\!\int_\oo\! J_1(x,y)\mathsf{V}_i(y,t)\dy-[d_2^*(x)+b_2(x,t)]\mathsf{V}_i
+\mu_2(x,t)\dd\frac{\mathsf{H}_i}{\mathsf{H}_u+\mathsf{H}_i}\mathsf{V}_u\\
&\hspace{17mm}-c_2(x,t)(\mathsf{V}_u+\mathsf{V}_i)\mathsf{V}_i,\!&&(x,t)\in\ol Q_T,\\[1mm]
&(\mathsf{H}_u,\mathsf{H}_i,\mathsf{V}_u,\mathsf{V}_i)(x,0)=
(\mathsf{H}_u,\mathsf{H}_i,\mathsf{V}_u,\mathsf{V}_i)(x,T),&&x\in\boo.
 \end{aligned}\rr.\;\;\lbl{4.3}\ees

 \begin{theo}\lbl{t5.2} The problem \qq{4.2} has a unique global solution $(H_u(x,t), H_i(x,t), V_u(x,t), V_i(x,t)\big)$, which is positive and bounded.
\end{theo}

The proof of Theorem \ref{t5.2} is standard and we omit the details.

\subsection{The existence and uniqueness of positive solutions of \qq{4.3}}\lbl{S4.1}
{\setlength\arraycolsep{2pt}

We first consider the logistic type equations, with $k=1,2$,
 \[\begin{cases}
 \mathsf{U}_t=d_k\dd\int_\oo J_k(x,y)\mathsf{U}(y,t)\dy
 +[a_k(x,t)-d_k^*(x)-b_k(x,t)]\mathsf{U}-c_k(x,t)\mathsf{U}^2,&(x,t)\in\ol Q_T\\[2mm]
 \mathsf{U}(x,0)=\mathsf{U}(x,T),&x\in\boo.
 \end{cases}\eqno(4.4_k)\]
If $\big(\mathsf{H}_u(x,t), \mathsf{H}_i(x,t), \mathsf{V}_u(x,t), \mathsf{V}_i(x,t)\big)$ is a nonnegative solution of \qq{4.3}, then $\mathsf{H}(x,t)=\mathsf{H}_u(x,t)+\mathsf{H}_i(x,t)$ and $\mathsf{V}(x,t)=\mathsf{V}_u(x,t)+\mathsf{V}_i(x,t)$ satisfy $(4.4_1)$ and $(4.4_2)$, respectively. Define operators ${\mathscr G}_k$, for $k=1,2$, by
\bess
 {\mathscr G}_k[\mathsf{U}]=d_k\dd\int_\oo J_k(x,y)\mathsf{U}(y,t)\dy+g_k(x,t)\mathsf{U}(x,t)
 -\mathsf{U}_t(x,t),\;\;\mathsf{U}\in \mathbb{X}^1,
 \eess
and let $\lm({\mathscr G}_k)$ be the generalized principal eigenvalue of ${\mathscr G}_k$, where
 \[g_k(x,t)=a_k(x,t)-d_k^*(x)-b_k(x,t),\;\;\;k=1,2.\]

As $c_k(x,t)>0$ in $\ol Q_T$, by Theorem \ref{t3.3}, we know that
the problem $(4.4_k)$ has a unique positive solution which is globally asymptotically stable when $\lm({\mathscr G}_k)>0$, while the problem $(4.4_k)$ has no positive solution and $0$ is asymptotically stable when $\lm({\mathscr G}_k)\le 0$.\setcounter{equation}{4}

In the following we assume that $\lm({\mathscr G}_k)>0$, $k=1,2$. Then $(4.4_1)$ and $(4.4_2)$  have unique positive solutions $\mathsf{H}(x,t)$ and $\mathsf{V}(x,t)$, respectively, and $\mathsf{H}(x,t)$ and $\mathsf{V}(x,t)$ are globally asymptotically stable in the time-periodic sense. Set
 \bess
 &\ell_{11}(x,t)=-[d_1^*(x)+b_1(x,t)+\gamma(x,t)+c_1(x,t)\mathsf{H}(x,t)],\;\;
 \ell_{12}(x,t)=\mu_1(x,t),&\\[1mm]
 &\ell_{21}(x,t)=\mu_2(x,t)\dd\frac{\mathsf{V}(x,t)}{\mathsf{H}(x,t)},\;\;
 \ell_{22}(x,t)=-[d_2^*(x)+b_2(x,t)+c_2(x,t)\mathsf{V}(x,t)],&
 \eess
Investigating the positive solutions of \qq{4.3} can be transformed into finding the positive solutions $\big(\mathsf{H}_i(x,t), \mathsf{V}_i(x,t)\big)$ of
 \bes\left\{\!\begin{aligned}
&\mathsf{H}_{it}=d_1\!\int_\oo\! J_1(x,y)\mathsf{H}_i(y,t)\dy
+\ell_{11}(x,t)\mathsf{H}_i+\mu_1(x,t)\dd\frac{\mathsf{H}(x,t)-\mathsf{H}_i}{\mathsf{H}(x,t)}
\mathsf{V}_i,&&(x,t)\in\ol Q_T,\\[1mm]
&\mathsf{V}_{it}=d_2\!\int_\oo\! J_2(x,y)\mathsf{V}_i(y,t)\dy+\ell_{22}(x,t)\mathsf{V}_i+\mu_2(x,t)\dd\frac{\mathsf{V}(x,t)-\mathsf{V}_i}{\mathsf{H}(x,t)}
\mathsf{H}_i,&&(x,t)\in\ol Q_T,\\[1mm]
&\mathsf{H}_i(x,0)=\mathsf{H}_i(x,T),\;\;\mathsf{V}_i(x,0)=\mathsf{V}_i(x,T),\;\;
&&x\in\boo
 \end{aligned}\rr.\qquad\lbl{4.7}\ees
satisfying $\mathsf{H}_i<\mathsf{H}$, $\mathsf{V}_i<\mathsf{V}$ in $\ol Q_T$. We should emphasize that \qq{4.7} is not a cooperative system within the range of $(\mathsf{H}_i, \mathsf{V}_i)\ge (0,0)$, but only within the range of $(0,0)\le(\mathsf{H}_i, \mathsf{V}_i)\le(\mathsf{H}, \mathsf{V})$.
To utilize the abstract conclusion presented earlier, it is necessary to consider an auxiliary system
\bes\left\{\!\begin{aligned}
&\mathsf{H}_{it}=d_1\!\int_\oo\! J_1(x,y)\mathsf{H}_i(y,t)\dy
+\ell_{11}(x,t)\mathsf{H}_i+\mu_1(x,t)\dd\frac{(\mathsf{H}(x,t)-\mathsf{H}_i)^+}{\mathsf{H}(x,t)}
\mathsf{V}_i,&&(x,t)\in\ol Q_T,\\[1mm]
&\mathsf{V}_{it}=d_2\!\int_\oo\! J_2(x,y)\mathsf{V}_i(y,t)\dy+\ell_{22}(x,t)\mathsf{V}_i+\mu_2(x,t)\dd\frac{(\mathsf{V}(x,t)-\mathsf{V}_i)^+}{\mathsf{H}(x,t)}
\mathsf{H}_i,&&(x,t)\in\ol Q_T,\\[1mm]
&\mathsf{H}_i(x,0)=\mathsf{H}_i(x,T),\;\;\mathsf{V}_i(x,0)=\mathsf{V}_i(x,T),\;\;
&&x\in\boo,
 \end{aligned}\rr.\qquad\lbl{4.7x}\ees
which is a cooperative system within the range of $(\mathsf{H}_i, \mathsf{V}_i)\ge (0,0)$. In fact, we will show that \qq{4.7} and \qq{4.7x} are equivalent.

Linearize the left-hand side of system \qq{4.7} at $(\mathsf{H}_i, \mathsf{V}_i)=(0,0)$ to obtain an operator ${\mathscr L}=({\mathscr L}_1, {\mathscr L}_2)$:
 \bess\left\{\begin{aligned}
&{\mathscr L}_1[\phi]=d_1\!\int_\oo\!J_1(x,y)\phi_1(y,t)\dy
+\ell_{11}(x,t)\phi_1(x,t)+\ell_{12}(x,t)\phi_2(x,t)-\phi_{1t}(x,t),\\[1mm]
&{\mathscr L}_2[\phi]=d_2\!\int_\oo\! J_2(x,y)\phi_2(y,t)\dy
+\ell_{22}(x,t)\phi_2(x,t)+\ell_{21}(x,t)\phi_1(x,t)-\phi_{2t}(x,t),
  \end{aligned}\rr.\eess
where $\phi=(\phi_1,\phi_2)\in \mathbb{X}^2_T:=\mathbb{X}^m_T\big|_{m=2}$. Set
 \[L(x,t)=(\ell_{ik}(x,t))_{2\times 2}.\]
Then $L(x,t)$ is a cooperative matrix, and $L(x_0,t)$ is irreducible for all $0\le t\le T$. Thus, Theorem \ref{t2.1} holds. Let $\lm({\mathscr L})$ be the generalized principal eigenvalue of ${\mathscr L}$ given in Theorem \ref{t2.1}. The main result of this section is the following theorem.\vspace{-2mm}

\begin{theo}\lbl{t5.1} Assume that $\lm({\mathscr G}_k)>0$, $k=1,2$.\www
\begin{enumerate}
\item[{\rm(1)}]\; If $\lm({\mathscr L})>0$, then problem \qq{4.7} has a unique continuous positive solution $\big(\mathsf{H}_i, \mathsf{V}_i\big)$ and satisfies
 \[\mathsf{H}_i<\mathsf{H},\;\;\;\mathsf{V}_i<\mathsf{V}
 \;\;\;\mbox{in}\;\;\ol Q_T.\]
So \qq{4.3} has a unique continuous positive solution
 \[\big(\mathsf{H}_u, \mathsf{H}_i, \mathsf{V}_u, \mathsf{V}_i\big)=\big(\mathsf{H}-\mathsf{H}_i,\, \mathsf{H}_i,\,  \mathsf{V}-\mathsf{V}_i, \,\mathsf{V}_i\big).\]
\item[{\rm(2)}]\;If $\lm({\mathscr L})\le 0$, then problem \qq{4.7} has no bounded solution with a positive lower bound, and so \qq{4.3} has no bounded solution with a positive lower bound.\www
 \end{enumerate}
 \end{theo}

Before giving the proof of Theorem \ref{t5.1}, we first prove the a general conclusion (the following Lemma \ref{le4.1}), which will be used in the sequel. To that end, let's make some preparations.

Let $\ud g_k^\ep(x,t)$ be the lower control function of $g_k(x,t)$ given in Theorem \ref{t2.1}, and $\phi^\ep_k(x,t)$ be the positive eigenfunction corresponding to $\lm_p(\ud{\mathscr G}_k^\ep)$, with $\|\phi^\ep_k\|_{L^\infty(\oo)}=1$. Then $(\lm_p(\ud{\mathscr G}_k^\ep),\,\phi^\ep_k)$ satisfies
  \[d_k\dd\int_\oo J_k(x,y)\phi^\ep_k(y,t)\dy+\ud g_k^\ep\phi^\ep_k(x,t)-\phi^\ep_{kt}(x,t)=\lm_p(\ud{\mathscr G}_k^\ep)\phi^\ep_k(x,t),\;\;(x,t)\in\ol Q_T,\;\;k=1,2.
 \eqno(4.6_k)\]
As $\lm({\mathscr G}_k)>0$, there exists $0<\ep_0\ll 1$ such that\setcounter{equation}{6}
 \bess
 \lm_p(\ud{\mathscr G}_k^\ep)>0,\;\;\;\forall\,0<\ep\le\ep_0.
 \eess
For the given $|\sigma|\ll 1$, we set
 \bes\begin{cases}
 \ell_{11}^\sigma(x,t)=-[d_1^*(x)+b_1(x,t)+\gamma(x,t)
 +c_1(x,t)(\mathsf{H}(x,t)-\sigma\phi^\ep_1(x,t))],\\[1mm]
 \ell_{12}^\sigma(x,t)=\mu_1(x,t)\dd
 \frac{\mathsf{H}(x,t)+\sigma\phi^\ep_1(x,t)}{\mathsf{H}(x,t)
 -\sigma\phi^\ep_1(x,t)},\\[3mm]
 \ell_{21}^\sigma(x,t)=\mu_2(x,t)\dd\frac{\mathsf{V}(x,t)+\sigma\phi^\ep_2(x,t)}
 {\mathsf{H}(x,t)-\sigma\phi^\ep_1(x,t)},\\[1mm] \ell_{22}^\sigma(x,t)=-[d_2^*(x)+b_2(x,t)+c_2(x,t)(\mathsf{V}(x,t)
 -\sigma\phi^\ep_2(x,t))].&
 \end{cases}\lbl{4.9}\ees
Then $\ell_{11}^\sigma(x,t)<0$, $\ell_{22}^\sigma(x,t)<0$ in $\ol Q_T$ and
 \bes
 L_\sigma(x,t)=(\ell_{kl}^\sigma(x,t))_{2\times 2}
  \lbl{4.10}\ees
is cooperative provided $|\sigma|\ll 1$. Moreover, $L_\sigma(x_0,t)$ is irreducible for all\zzz $0\le t\le T$.

\begin{lem}\lbl{le4.1} Assume that $\lm({\mathscr G}_k)>0$, $k=1,2$, and $\lm({\mathscr L})>0$. Then there is a $0<\sigma_0\ll 1$ such that, when $|\sigma|\le\sigma_0$, the problems
 \bes\left\{\begin{aligned}
&\mathsf{H}_{it}=d_1\!\int_\oo\! J_1(x,y)\mathsf{H}_i(y,t)\dy+\ell_{11}^\sigma(x,t)\mathsf{H}_i
+\mu_1(x,t)\dd\frac{\big(\mathsf{H}+\sigma\phi^\ep_1-\mathsf{H}_i\big)^+}
{\mathsf{H}-\sigma\phi^\ep_1}\mathsf{V}_i,\!&&(x,t)\in\ol Q_T,\\
&\mathsf{V}_{it}=d_2\!\int_\oo\! J_2(x,y)\mathsf{V}_i(y,t)\dy+\ell_{22}^\sigma(x,t)\mathsf{V}_i
+\mu_2(x,t)\dd\frac{\big(\mathsf{V}+\sigma\phi^\ep_2-\mathsf{V}_i\big)^+}
{\mathsf{H}-\sigma\phi^\ep_1}\mathsf{H}_i,\!&&(x,t)\in\ol Q_T,\\
&\mathsf{H}_i(x,0)=\mathsf{H}_i(x,T),\;\;\mathsf{V}_i(x,0)=\mathsf{V}_i(x,T),\;\;
&&x\in\boo
  \end{aligned}\rr.\quad\;\;\lbl{4.11}\ees
and
  \bes\left\{\begin{aligned}
&\mathsf{H}_{it}=d_1\!\int_\oo\! J_1(x,y)\mathsf{H}_i(y,t)\dy+\ell_{11}^\sigma(x,t)\mathsf{H}_i
+\mu_1(x,t)\dd\frac{\mathsf{H}+\sigma\phi^\ep_1-\mathsf{H}_i}
{\mathsf{H}-\sigma\phi^\ep_1}\mathsf{V}_i, \!&&(x,t)\in\ol Q_T,\\
&\mathsf{V}_{it}=d_2\!\int_\oo\! J_2(x,y)\mathsf{V}_i(y,t)\dy+\ell_{22}^\sigma(x,t)\mathsf{V}_i
+\mu_2(x,t)\dd\frac{\mathsf{V}+\sigma\phi^\ep_2-\mathsf{V}_i}
{\mathsf{H}-\sigma\phi^\ep_1} \mathsf{H}_i,\!&&(x,t)\in\ol Q_T,\\
&\mathsf{H}_i(x,0)=\mathsf{H}_i(x,T),\;\;\mathsf{V}_i(x,0)=\mathsf{V}_i(x,T),\;\;
&&x\in\boo
 \end{aligned}\rr.\quad\;\;\lbl{4.12}\ees
have, respectively, unique continuous positive solutions $(\widehat{\mathsf{H}}_i^\sigma,\widehat{\mathsf{V}}_i^\sigma)$ and $(\mathsf{H}_i^\sigma, \mathsf{V}_i^\sigma)$. Moreover,
 \bess
 \widehat{\mathsf{H}}_i^\sigma, \,\mathsf{H}_i^\sigma<\mathsf{H}+\sigma\phi^\ep_1, \;\;\;\widehat{\mathsf{V}}_i^\sigma,\, \mathsf{V}_i^\sigma<\mathsf{V}+\sigma\phi^\ep_2\;\;\;{\rm in}\;\;\ol Q_T.
 \eess
Therefore, $(\widehat{\mathsf{H}}_i^\sigma, \widehat{\mathsf{V}}_i^\sigma)
 =(\mathsf{H}_i^\sigma, \mathsf{V}_i^\sigma)$ in $\ol Q_T$,
and \qq{4.11} and \qq{4.12} are equivalent.\zzz
\end{lem}

\begin{proof} Here, we only discuss the existence and uniqueness of continuous  positive solutions for problem \qq{4.11}. The existence and uniqueness of positive solutions for problem \qq{4.12} can be addressed in a similar manner, which is simpler than that of \qq{4.11}.

{\it Step 1: The construction of upper solution}. Owing to $a_1, a_2, \mathsf{H}, \mathsf{V}>0$ in $\ol Q_T$, we can choose $0<\sigma_0\ll 1$ such that, when $|\sigma|\le\sigma_0$,
 \bes\left\{\begin{aligned}
&\mathsf{H}+\sigma\phi^\ep_1>0, \;\; \mathsf{V}+\sigma\phi^\ep_2>0&&\mbox{in}\;\;\ol Q_T, \\[.1mm]
&a_1(\mathsf{H}+\sigma\phi^\ep_1)>\sigma\phi^\ep_1\big(\lm_p(\ud{\mathscr G}_1^\ep)+g_1-\ud g_1^\ep+\sigma c_1\phi^\ep_1\big)&&\mbox{in}\;\;\ol Q_T,\\[.1mm]
&a_2(\mathsf{V}+\sigma\phi^\ep_2)>\sigma\phi^\ep_2\big(\lm_p(\ud{\mathscr G}_2^\ep)+g_2-\ud g_2^\ep+\sigma c_2\phi^\ep_2\big)&&\mbox{in}\;\;\ol Q_T.
 \end{aligned}\rr.\lbl{4.13}\ees

Now we verify that
 \bess
 (\ol{\mathsf{H}}_i(x,t), \ol{\mathsf{V}}_i(x,t))=\big(\mathsf{H}(x,t)+\sigma\phi^\ep_1(x,t), \mathsf{V}(x,t)+\sigma\phi^\ep_2(x,t)\big)\eess
is a strict upper solution of \qq{4.11}. By the direct calculations we have that (For clarity and space saving, we omitted $(x,t)$ in this calculation process)
\bess
&&d_1\!\int_\oo\! J_1(x,y)\ol{\mathsf{H}}_i(y,t)\dy+\ell_{11}^\sigma\ol{\mathsf{H}}_i
+\mu_1\dd\frac{\big(\mathsf{H}+\sigma\phi^\ep_1-\ol{\mathsf{H}}_i\big)^+}
{\mathsf{H}-\sigma\phi^\ep_1}
\ol{\mathsf{V}}_i-\ol{\mathsf{H}}_{it}\\[.1mm]
&=&d_1\!\int_\oo\! J_1(x,y)[\mathsf{H}(y,t)+\sigma\phi^\ep_1(y,t)]\dy
+\big[g_1-a_1-c_1(\mathsf{H}-\sigma\phi^\ep_1)\big](\mathsf{H}
+\sigma\phi^\ep_1)-{\mathsf{H}}_{t}-\sigma\phi^\ep_{1t}\\[.1mm]
&=&d_1\!\int_\oo\! J_1(x,y)\mathsf{H}(y,t)\dy+g_1\mathsf{H}-c_1\mathsf{H}^2
-{\mathsf{H}}_{t}+\sigma\kk(d_1\!\int_\oo\! J_1(x,y)\phi^\ep_1(y,t)\dy+\ud g_1^\ep\phi^\ep_1-\phi^\ep_{1t}\rr)\\[.1mm]
&&-a_1(\mathsf{H}+\sigma\phi^\ep_1)+\sigma(g_1-\ud g_1^\ep)\phi^\ep_1+c_1\sigma^2(\phi^\ep_1)^2\\[.1mm]
&=&\sigma\lm_p(\ud{\mathscr G}_1^\ep)\phi^\ep_1
-a_1(\mathsf{H}+\sigma\phi^\ep_1)+\sigma(g_1-\ud g_1^\ep)\phi^\ep_1+c_1\sigma^2(\phi^\ep_1)^2\\[.1mm]
&=&\sigma\phi^\ep_1(\lm_p(\ud{\mathscr G}_1^\ep)+g_1-\ud g_1^\ep+\sigma c_1\phi^\ep_1)-a_1(\mathsf{H}+\sigma\phi^\ep_1)\\[.1mm]
 &<&0\;\;\;\mbox{in}\;\;\ol Q_T
 \eess
by the second inequality of \qq{4.13}. Likewise,
 \bess
 d_2\!\int_\oo\! J_2(x,y)\ol{\mathsf{V}}_i(y,t)\dy
 +\ell_{22}^\sigma\ol{\mathsf{V}}_i+\mu_2\dd\frac{\big(\mathsf{V}
+\sigma\phi^\ep_2-\ol{\mathsf{V}}_i\big)^+}{\mathsf{H}-\sigma\phi^\ep_1}
\ol{\mathsf{H}}_i-\ol{\mathsf{V}}_{it}<0\;\;\;\mbox{in}\;\;\ol Q_T.
\eess

{\it Step 2: The existence of positive solutions}. In view of $\lm({\mathscr L})>0$, we can find a $0<\sigma_0\ll 1$ such that $\lm({\mathscr L}_\sigma)>0$ for all $|\sigma|\le\sigma_0$ by the continuity. Set
 \bess
 f_1(x,t,\mathsf{H}_i,\mathsf{V}_i)&=&d^*_1(x) H_i+\ell_{11}^\sigma(x,t)\mathsf{H}_i
+\mu_1(x,t)\dd\frac{\big(\mathsf{H}(x,t)+\sigma\phi^\ep_1(x,t)-\mathsf{H}_i\big)^+}
{\mathsf{H}(x,t)-\sigma\phi^\ep_1(x,t)}\mathsf{V}_i,\\[.2mm]
f_2(x,t,\mathsf{H}_i,\mathsf{V}_i)&=&d^*_2(x) H_i+\ell_{22}^\sigma(x,t)\mathsf{V}_i
+\mu_2(x,t)\dd\frac{\big(\mathsf{V}(x,t)+\sigma\phi^\ep_2(x,t)-\mathsf{V}_i\big)^+}
{\mathsf{H}(x,t)-\sigma\phi^\ep_1(x,t)}\mathsf{H}_i.
 \eess
It is easy to see that $f(x,t,\mathsf{H}_i,\mathsf{V}_i)=
(f_1(x,t,\mathsf{H}_i,\mathsf{V}_i),f_2(x,t,\mathsf{H}_i,\mathsf{V}_i))$ satisfies {\bf(F1)}--{\bf(F3)} in the region of $\mathsf{H}_i \in (0,\mathsf{H}+\sigma\phi^\ep_1)$, $\mathsf{V}_i \in (0,\mathsf{V}+\sigma\phi^\ep_2)$. Noticing that $L_\sigma(x_0,t)$ is irreducible for all $0\le t\le T$. By repeating the arguments to those in the proofs of Theorem \ref{th3.6}(1) and Theorem \ref{th3.5}, locating between $(0,0)$ and $(\mathsf{H}+\sigma\phi^\ep_1, \mathsf{V}+\sigma\phi^\ep_2)$, the problem \qq{4.11} has a positive solution $(\widehat{\mathsf{H}}_i^\sigma, \widehat{\mathsf{V}}_i^\sigma)\in [C(\ol Q_T)]^2$. Moreover, it follows from Corollary \ref{c3.1} that any positive solution $(\widetilde{\mathsf{H}}_i^\sigma, \widetilde{\mathsf{V}}_i^\sigma)$ of \qq{4.11} locating between $(0,0)$ and $(\mathsf{H}+\sigma\phi^\ep_1, \mathsf{V}+\sigma\phi^\ep_2)$ satisfies
$\widetilde{\mathsf{H}}_i^\sigma<\mathsf{H}+\sigma\phi^\ep_1$ and $ \widetilde{\mathsf{V}}_i^\sigma<\mathsf{V}+\sigma\phi^\ep_2$ in $\ol Q_T$
since $\big(\mathsf{H}+\sigma\phi^\ep_1, \mathsf{V}+\sigma\phi^\ep_2\big)$ is a strict upper solution of \qq{4.11}. Therefore, the positive solution of \qq{4.11} locating between $(0,0)$ and $(\mathsf{H}+\sigma\phi^\ep_1, \mathsf{V}+\sigma\phi^\ep_2)$ is unique (Theorem \ref{th3.4}), which is exactly $\big(\widehat{\mathsf{H}}_i^\sigma, \widehat{\mathsf{V}}_i^\sigma\big)$.
Certainly, $\big(\widehat{\mathsf{H}}_i^\sigma, \widehat{\mathsf{V}}_i^\sigma\big)$ satisfies \qq{4.12}.

{\it Step 3: The uniqueness of the positive solution}. Until now, we already established the uniqueness of the continuous positive solution of \qq{4.11} which does not exceed $\big(\mathsf{H}+\sigma\phi^\ep_1,\,\mathsf{V}+\sigma\phi^\ep_2\big)$.
The global uniqueness can also be obtained by repeating the arguments used in Theorem \ref{th3.4}. For reader's convenience, we provide the details here.

Let $(\widetilde{\mathsf{H}}_i^\sigma, \widetilde{\mathsf{V}}_i^\sigma)$ be another continuous positive solution of \qq{4.11}. We can find a constant $0<s<1$ such that $s(\widetilde{\mathsf{H}}_i^\sigma, \widetilde{\mathsf{V}}_i^\sigma)\le(\widehat{\mathsf{H}}_i^\sigma, \widehat{\mathsf{V}}_i^\sigma)$ in $\ol Q_T$. Set
\[\bar s=\sup\kk\{0<s\le 1: s\big(\widetilde{\mathsf{H}}_i^\sigma, \widetilde{\mathsf{V}}_i^\sigma\big)\le\big(\widehat{\mathsf{H}}_i^\sigma, \widehat{\mathsf{V}}_i^\sigma\big) \;{\rm ~ in  ~ } \ol Q_T\rr\}.\]
Then $\bar s$ is well defined and $0<\bar s\le1$, and $\bar s(\widetilde{\mathsf{H}}_i^\sigma, \widetilde{\mathsf{V}}_i^\sigma)\le\big(\widehat{\mathsf{H}}_i^\sigma, \widehat{\mathsf{V}}_i^\sigma)$ in $\ol Q_T$. We shall prove $\bar s=1$. Assume on the contrary that $\bar s<1$. Then $\mathsf{U}:=\widehat{\mathsf{H}}_i^\sigma-\bar s\widetilde{\mathsf{H}}_i^\sigma\ge 0$ and $\mathsf{Z}:=\widehat{\mathsf{V}}_i^\sigma-\bar s\widetilde{\mathsf{V}}_i^\sigma\ge 0$, as well as $\bar s\widetilde{\mathsf{H}}_i^\sigma\le\widehat{\mathsf{H}}_i^\sigma<\mathsf{H}+\sigma\phi^\ep_1$ and $\bar s\widetilde{\mathsf{V}}_i^\sigma\le\widehat{\mathsf{V}}_i^\sigma<\mathsf{V}+\sigma\phi^\ep_2$ in $\ol Q_T$. It follows that
\bess
&\big(\mathsf{H}+\sigma\phi^\ep_1-\widehat{\mathsf{H}}_i^\sigma\big)^+
=\mathsf{H}+\sigma\phi^\ep_1-\widehat{\mathsf{H}}_i^\sigma,&\\[0.5mm]
&\big(\mathsf{H}+\sigma\phi^\ep_1-\widetilde{\mathsf{H}}_i^\sigma\big)^+
<\mathsf{H}+\sigma\phi^\ep_1-\bar s\widetilde{\mathsf{H}}_i^\sigma,&\\[0.5mm]
& \big(\mathsf{H}+\sigma\phi^\ep_1-\widehat{\mathsf{H}}_i^\sigma\big)^+
-\big(\mathsf{H}+\sigma\phi^\ep_1-\widetilde{\mathsf{H}}_i^\sigma\big)^+
>\bar s\widetilde{\mathsf{H}}_i^\sigma-\widehat{\mathsf{H}}_i^\sigma,&\\[0.5mm]
&\big(\mathsf{V}+\sigma\phi^\ep_2-\widehat{\mathsf{V}}_i^\sigma\big)^+
=\mathsf{V}+\sigma\phi^\ep_2-\widehat{\mathsf{V}}_i^\sigma,&\\[0.5mm] &\big(\mathsf{V}+\sigma\phi^\ep_2-\widetilde{\mathsf{V}}_i^\sigma\big)^+
<\mathsf{V}+\sigma\phi^\ep_2-\bar s\widetilde{\mathsf{V}}_i^\sigma,&\\[0.5mm]
&\big(\mathsf{V}+\sigma\phi^\ep_2-\widehat{\mathsf{V}}_i^\sigma\big)^+
-\big(\mathsf{V}+\sigma\phi^\ep_2-\widetilde{\mathsf{V}}_i^\sigma\big)^+
>\bar s\widetilde{\mathsf{V}}_i^\sigma-\widehat{\mathsf{V}}_i^\sigma&
\eess
in $\ol Q_T$. Denote
\bess
\alpha_1(x,t)=\frac{\mu_1(x,t)\widetilde{\mathsf{V}}_i^\sigma(x,t)}{\mathsf{H}(x,t)
-\sigma\phi^\ep_1(x,t)}-\ell_{11}^\sigma(x,t),\;\;\;
\alpha_2(x,t)=\frac{\mu_2(x,t)\widetilde{\mathsf{H}}_i^\sigma(x,t)}{\mathsf{H}(x,t)
-\sigma\phi^\ep_1(x,t)}-\ell_{22}^\sigma(x,t).
  \eess
After careful calculation, we can show that $\big(\mathsf{U}, \mathsf{Z}\big)$ satisfies, in $\ol Q_T$,
\bes\left\{\!\!\begin{aligned}
&\mathsf{U}_t-d_1\!\int_\oo\!J_1(x,y)\mathsf{U}(y,t)\dy	 +\alpha_1(x,t)\mathsf{U}>\mu_1(x,t)\dd\frac{\mathsf{H}(x,t)+\sigma\phi^\ep_1(x,t)
	-\widehat{\mathsf{H}}_i^\sigma(x,t)}{\mathsf{H}(x,t)
-\sigma\phi^\ep_1(x,t)}\mathsf{Z}\ge 0,\\
&\mathsf{Z}_t-d_2\!\int_\oo\!J_2(x,y)\mathsf{Z}(y,t)\dy +\alpha_2(x,t)\mathsf{Z}>\mu_2(x,t)\dd\frac{\mathsf{V}(x,t)+\sigma\phi^\ep_2(x,t)
	-\widehat{\mathsf{V}}_i^\sigma(x,t)}{\mathsf{H}(x,t)
-\sigma\phi^\ep_1(x,t)}\mathsf{U}\ge 0.
\end{aligned}\right.\qquad\lbl{4.14}\ees
It follows that $\mathsf{U},\mathsf{Z}>0$ in $\ol Q_T$ by the maximum principle (Lemma \ref{le3.3}) as $\mathsf{U},\mathsf{Z}\ge 0$ in $\ol Q_T$. Then there exists $0<\varrho<1-\bar s$ such that $(\mathsf{U}, \mathsf{Z})\ge\varrho(\widetilde{\mathsf{H}}_i^\sigma, \widetilde{\mathsf{V}}_i^\sigma)$, i.e., $(\bar s+\varrho)(\widetilde{\mathsf{H}}_i^\sigma, \widetilde{\mathsf{V}}_i^\sigma)\le(\widehat{\mathsf{H}}_i^\sigma, \widehat{\mathsf{V}}_i^\sigma)$ in $\ol Q_T$. This contradicts the definition of $\bar s$. Hence $\bar s=1$, i.e.,
 \[(\widetilde{\mathsf{H}}_i^\sigma,\, \widetilde{\mathsf{V}}_i^\sigma\big)\le\big(\widehat{\mathsf{H}}_i^\sigma,\, \widehat{\mathsf{V}}_i^\sigma\big)\;\;\;\mbox{in}\;\;\ol Q_T.\]
Certainly, $\widetilde{\mathsf{H}}_i^\sigma<\mathsf{H}+\sigma\phi^\ep_1$, $\widetilde{\mathsf{V}}_i^\sigma<\mathsf{V}+\sigma\phi_2^\ep$, and then
\bess
\big(\mathsf{H}+\sigma\phi^\ep_1-\widetilde{\mathsf{H}}_i^\sigma\big)^+
=\mathsf{H}+\sigma\phi^\ep_1-\widetilde{\mathsf{H}}_i^\sigma,\;\;\;
\big(\mathsf{V}+\sigma\phi^\ep_2-\widetilde{\mathsf{V}}_i^\sigma\big)^+
=\mathsf{V}+\sigma\phi^\ep_2-\widetilde{\mathsf{V}}_i^\sigma.\eess

On the other hand, we can find $\gamma >1$ such that
$\gamma \big(\widetilde{\mathsf{H}}_i^\sigma, \widetilde{\mathsf{V}}_i^\sigma\big)\ge\big(\widehat{\mathsf{H}}_i^\sigma, \widehat{\mathsf{V}}_i^\sigma\big)$ in $\ol Q_T$. Set
\[\ud \gamma=\inf\kk\{k\ge1: k\big(\widetilde{\mathsf{H}}_i^\sigma, \widetilde{\mathsf{V}}_i^\sigma\big)\ge\big(\widehat{\mathsf{H}}_i^\sigma, \widehat{\mathsf{V}}_i^\sigma\big)\;{\rm ~ in  ~ }\, \ol Q_T\rr\}.\]
Then $\ud \gamma$ is well defined, $\ud \gamma\ge1$ and $\ud \gamma\big(\widetilde{\mathsf{H}}_i^\sigma, \widetilde{\mathsf{V}}_i^\sigma\big)\ge\big(\widehat{\mathsf{H}}_i^\sigma, \widehat{\mathsf{V}}_i^\sigma\big)$ in $\ol Q_T$. If $\ud \gamma>1$, then $\mathsf{P}:=\ud \gamma\widetilde{\mathsf{H}}_i^\sigma-\widehat{\mathsf{H}}_i^\sigma\ge 0$ and $\mathsf{Q}:=\ud \gamma\widetilde{\mathsf{V}}_i^\sigma-\widehat{\mathsf{V}}_i^\sigma\ge 0$ in $\ol Q_T$. Similarly to the above, we can verify that $\big(\mathsf{P}, \mathsf{Q}\big)$ satisfies a system of differential inequalities similar to \qq{4.14} and derive $\mathsf{P}, \mathsf{Q}>0$ in $\ol Q_T$ by the maximum principle, and there exists $0<r<\ud \gamma-1$ such that $\big(\mathsf{P}, \mathsf{Q}\big)\ge r\big(\widetilde{\mathsf{H}}_i^\sigma, \widetilde{\mathsf{V}}_i^\sigma\big)$, i.e., $(\ud \gamma-r)\big(\widetilde{\mathsf{H}}_i^\sigma, \widetilde{\mathsf{V}}_i^\sigma\big)\ge\big(\widehat{\mathsf{H}}_i^\sigma, \widehat{\mathsf{V}}_i^\sigma\big)$ in $\ol Q_T$. This contradicts the definition of $\ud \gamma$. Hence $\ud \gamma=1$ and
\[\big(\widetilde{\mathsf{H}}_i^\sigma, \widetilde{\mathsf{V}}_i^\sigma\big)\ge\big(\widehat{\mathsf{H}}_i^\sigma, \widehat{\mathsf{V}}_i^\sigma\big)\;\;\;\mbox{in}\;\;\ol Q_T.\]
The uniqueness is obtained and the proof is complete.\zzz\end{proof}

\begin{proof}[Proof of Theorem \ref{t5.1}] (1) Taking $\sigma=0$ in Lemma \ref{le4.1}, the existence and uniqueness of continuous positive solutions are proved.

(2)\, Assume $\lm({\mathscr L})\le 0$ and prove that \qq{4.7} has no bounded solution with a positive lower bound. Assume on the contrary that $(\mathsf{H}_i, \mathsf{V}_i)$ is a such solution of \qq{4.7}. Set
 \[\rho_1=\inf_{Q_T}\frac{\mu_1(x,t)\mathsf{V}_i(x,t)}{\mathsf{H}(x,t)},\;\;
 \rho_2=\inf_{Q_T}\frac{\mu_2(x,t)\mathsf{H}_i(x,t)}{\mathsf{H}(x,t)}.\]
Then $\rho_1, \rho_2>0$. By the direct calculations we have that, for $\rho=\min\{\rho_1, \rho_2\}$,
 \bess
 \mathsf{H}_{it}&=&d_1\!\int_\oo\! J_1(x,y)\mathsf{H}_i(y,t)\dy+\ell_{11}(x,t)\mathsf{H}_i+\mu_1(x,t)\mathsf{V}_i
 -\frac{\mu_1(x,t)\mathsf{V}_i}{\mathsf{H}(x,t)}\mathsf{H}_i\\
&\le&d_1\!\int_\oo\! J_1(x,y)\mathsf{H}_i(y,t)\dy+\ell_{11}(x,t)\mathsf{H}_i+\ell_{12}(x,t)
\mathsf{V}_i-\rho_1\mathsf{H}_i\\
&\le&d_1\!\int_\oo\! J_1(x,y)\mathsf{H}_i(y,t)\dy+\bar \ell_{11}^\ep(x,t)\mathsf{H}_i-\rho\mathsf{H}_i+\bar \ell_{12}^\ep(x,t)\mathsf{V}_i,\;\;\;(x,t)\in\ol Q_T,\\
\mathsf{V}_{it}&=&d_2\!\int_\oo\! J_2(x,y)\mathsf{V}_i(y,t)\dy+\ell_{22}(x,t)\mathsf{V}_i+\mu_2(x,t)
\frac{\mathsf{V}(x,t)-\mathsf{V}_i}{\mathsf{H}(x,t)}\mathsf{H}_i\\
&\le&d_2\!\int_\oo\! J_2(x,y)\mathsf{V}_i(y,t)\dy+\bar \ell_{22}^\ep(x,t)\mathsf{V}_i-\rho\mathsf{V}_i+\bar \ell_{21}^\ep(x,t)\mathsf{H}_i,\;\;\;(x,t)\in\ol Q_T.
 \eess
It follows that $\lm_p(\ol{\mathscr L}^\ep)\ge\rho$ for all $0<\ep\ll 1$, which implies $\lm({\mathscr L})\ge\rho>0$. This is a contradiction. The proof is complete.
\end{proof}

\subsection{Dynamical properties of \qq{4.2}}\lbl{s3}

In this section we study the stabilities of nonnegative equilibrium solutions.

\begin{theo}\lbl{t4.3} Let $(H_u, H_i, V_u, V_i\big)$ be the unique positive solution of \qq{4.2}. \vspace{-2mm}
\begin{enumerate}[$(1)$]
\item\; If $\lm({\mathscr G}_1)>0$, $\lm({\mathscr G}_2)>0$ and $\lm({\mathscr L})>0$, then, in $C(\ol Q_T)$,
 \bes
 &&\lim_{n\to+\yy}\big(H_u(x,t+nT), H_i(x,t+nT), V_u(x,t+nT), V_i(x,t+nT)\big)\nm\\[0.5mm]
 &=&\big(\mathsf{H}(x,t)-\mathsf{H}_i(x,t),\mathsf{H}_i(x,t), \mathsf{V}(x,t)-\mathsf{V}_i(x,t),\mathsf{V}_i(x,t)\big),\qquad
 \lbl{4.17}\ees
where $\mathsf{H}(x,t)$, $\mathsf{V}(x,t)$ and $\big(\mathsf{H}_i(x,t), \mathsf{V}_i(x,t)\big)$ are the unique positive solutions of $(4.4_1)$, $(4.4_2)$ and \qq{4.7}, respectively, and $\mathsf{H}_i(x,t)<\mathsf{H}(x,t), \mathsf{V}_i(x,t)<\mathsf{V}(x,t)$.

\item\; If $\lm({\mathscr G}_1)>0$, $\lm({\mathscr G}_2)>0$ and $\lm({\mathscr L})\le 0$, then, in $C(\ol Q_T)$,
 \bes
 \lim_{n\to+\yy}\big(H_u(x,t+nT), H_i(x,t+nT), V_u(x,t+nT), V_i(x,t+nT)\big)=\big(\mathsf{H}(x,t), 0, \mathsf{V}(x,t), 0\big).
 \qquad\lbl{4.18}\ees

\item\; If $\lm({\mathscr G}_1)\le 0$, $\lm({\mathscr G}_2)>0$, then, in $C(\ol Q_T)$,
 \bess
\lim_{n\to+\yy}\big(H_u(x,t+nT), H_i(x,t+nT), V_u(x,t+nT), V_i(x,t+nT)\big)=\big(0, 0, \mathsf{V}(x,t), 0\big).\eess

If $\lm({\mathscr G}_1)>0$, $\lm({\mathscr G}_2)\le 0$, then, in $C(\ol Q_T)$,
 \bess
\lim_{n\to+\yy}\big(H_u(x,t+nT), H_i(x,t+nT), V_u(x,t+nT), V_i(x,t+nT)\big)=\big(\mathsf{H}(x,t), 0, 0, 0\big).\eess

If $\lm({\mathscr G}_1) \le 0$, $\lm({\mathscr G}_2) \le0$, then, in $C(\ol Q_T)$,
  \bess
  \lim_{n\to+\yy}\big(H_u(x,t+nT), H_i(x,t+nT), V_u(x,t+nT), V_i(x,t+nT)\big)=(0, 0, 0, 0).\eess
 \end{enumerate}\vspace{-3mm}
  \end{theo}

\begin{proof} Throughout this proof, the matrix $L_\sigma(x,t)$ is defined by \qq{4.10}. Let $(H_u, H_i, V_u, V_i)$ be the unique solution of \qq{4.2}. Then $H=H_u+H_i$ and $V=V_u+V_i$ satisfy
 \bess\left\{\!\begin{aligned}
 &H_t=d_1\!\dd\int_\oo\! J_1(x,y)H(y,t)\dy+g_1(x,t)H-c_1(x,t)H^2,&&x\in\boo,\; t>0,\\
&H(x,0)=H_u(x,0)+H_i(x,0)>0,&&x\in\boo
	\end{aligned}\rr.\eess
and
 \bess\left\{\begin{aligned}
&V_t=d_2\!\dd\int_\oo\! J_2(x,y)V(y,t)\dy+g_2(x,t)V-c_2(x,t)V^2,&&x\in\oo,\; t>0,\\
&V(x,0)=V_u(x,0)+V_i(x,0)>0,&&x\in\oo,
	\end{aligned}\rr.\eess
respectively. As $\lm({\mathscr G}_1)>0$, $\lm({\mathscr G}_2)>0$, we have that, by Theorem \ref{t3.3},
 \bes
 \lim_{n\to+\yy}H(x,t+nT)=\mathsf{H}(x,t),\;\;\;\lim_{n\to+\yy}V(x,t+nT)=\mathsf{V}(x,t)
 \;\;\;{\rm uniformly\,\, in}\;\;\ol Q_T.\lbl{4.19}\ees

Similar to the above, let $\phi^\ep_k(x,t)>0$ with $\|\phi^\ep_k\|_{L^\infty(Q_T)}=1$ be the unique solution of $(4.6_k)$. For any given $0<\sigma\le\sigma_0$, there exists $N_\sigma\gg 1$ such that
 \bes\left\{\begin{aligned}
&0<{\mathsf H}(x,t)-\sigma\phi^\ep_1(x,t)\le H(x,t+nT)\le {\mathsf H}(x,t)+\sigma\phi^\ep_1(x,t),&&(x,t)\in Q_T,\; n\ge N_\sigma,\\
&0<{\mathsf V}(x,t)-\sigma\phi^\ep_2(x,t)\le V(x,t+nT)\le {\mathsf V}(x,t)+\sigma\phi^\ep_2(x,t),&&(x,t)\in Q_T,\; n\ge N_\sigma.
  \end{aligned}\rr.\lbl{4.20}\ees
Since ${\mathsf H}, {\mathsf V}, \phi^\ep_1, \phi^\ep_2$ and all coefficient functions  are $T$-periodic in time $t$. Making use of $H_u=H-H_i$, $V_u=V-V_i$ and \qq{4.20}, we see that $\big(H_i, V_i\big)$ satisfies
 \bes\left\{\!\begin{aligned}
&H_{it}\le d_1\!\dd\int_\oo\! J_1(x,y)H_i(y,t)\dy-d^*_1(x) H_i+ \mu_1(x,t)\dd\frac{(\mathsf{H}+\sigma\phi^\ep_1-H_i)^+}{\mathsf{H}
 -\sigma\phi^\ep_1}V_i\\
 &\hspace{16mm}-\big[b_1(x,t)+c_1(x,t)(\mathsf{H}-\sigma\phi^\ep_1)\big]H_i\\
&\hspace{4.5mm}=d_1\!\dd\int_\oo\! J_1(x,y)H_i(y,t)\dy+\ell_{11}^\sigma(x,t) H_i+ \mu_1(x,t)\dd\frac{(\mathsf{H}+\sigma\phi^\ep_1-H_i\big)^+}{\mathsf{H}
 -\sigma\phi^\ep_1}V_i,
\!&&x\in\ol\oo,\; t\ge n T,\\
&V_{it}\le d_2\!\dd\int_\oo\! J_2(x,y)V_i(y,t)\dy-d^*_2(x)V_i+ \mu_2(x,t)\dd\frac{(\mathsf{V}+\sigma\phi^\ep_2-V_i)^+}{\mathsf{H} -\sigma\phi^\ep_1}H_i\\
 &\hspace{16mm}-\big[b_2(x,t)+c_2(x,t)(\mathsf{V}-\sigma\phi^\ep_2)\big]V_i\\
&\hspace{4.5mm}=d_2\!\dd\int_\oo\! J_2(x,y)V_i(y,t)\dy+\ell_{22}^\sigma(x,t)V_i+ \mu_2(x,t)\dd\frac{(\mathsf{V}+\sigma\phi^\ep_2-V_i)^+}{\mathsf{H} -\sigma\phi^\ep_1}H_i,
\!&&x\in\ol\oo,\; t\ge n T,\\
&H_i\le H\le \mathsf{H}(x,0)+\sigma\phi^\ep_1(x,0),\;\;
V_i\le V\le \mathsf{V}(x,0)+\sigma\phi^\ep_2(x,0),&&x\in\boo,\,t=n T
 \end{aligned}\rr.\lbl{4.21}\ees
for all $n\ge N_\sigma$, where $\ell_{11}^\sigma(x,t)$ and $\ell_{22}^\sigma(x,t)$ are defined by \qq{4.9}.

(1)\, Assume that $\lm({\mathscr G}_1)>0$, $\lm({\mathscr G}_2)>0$ and $\lm({\mathscr L})>0$.

{\it Step 1}. Since $\lm({\mathscr L})>0$, there exists a $0<\sigma_0\ll 1$ such that $\lm({\mathscr L}_\sigma)>0$ for all $|\sigma|\le\sigma_0$ by the continuity. Let $(U_\sigma,Z_\sigma)$ be the solution of
\bes\left\{\!\begin{aligned}
 &U_{\sigma,t}=d_1\!\dd\int_\oo\!J_1(x,y)U_\sigma(y,t)\dy
 +\ell_{11}^\sigma(x,t)U_\sigma+ \mu_1(x,t)\dd\frac{(\mathsf{H}
 +\sigma\phi^\ep_1-U_\sigma\big)^+}{\mathsf{H} -\sigma\phi^\ep_1}Z_\sigma,\!&&x\in\oo,\; t>0,\\[1mm]
 &Z_{\sigma,t}=d_2\!\dd\int_\oo\! J_2(x,y)Z_\sigma(y,t)\dy+\ell_{22}^\sigma(x,t)Z_\sigma+ \mu_2(x,t)\dd\frac{(\mathsf{V}+\sigma\phi^\ep_2-Z_\sigma\big)^+}{\mathsf{H} -\sigma\phi^\ep_1}U_\sigma,\!&&x\in\oo,\; t>0,\\[1mm]
&(U_\sigma(x,0), Z_\sigma(x,0))=(U_0(x,0), Z_0(x,0)),\!&&x\in\Omega.
  \end{aligned}\rr.\quad\;\;\;\;\lbl{4.22}\ees
with $U_0(x,t)=\mathsf{H}(x,t)+\sigma\phi^\ep_1(x,t)$ and $Z_0(x,t)=\mathsf{V}(x,t)+\sigma\phi^\ep_2(x,t)$.  By the comparison principle,
 \bes
 (H_i(x,t+n T),V_i(x,t+n T))\le(U_\sigma(x,t), Z_\sigma(x,t))\;\;\;\mbox{in}\;\; \boo\times[0,+\yy),\;\;\forall\,n\ge N_\sigma.
 \lbl{4.24}\ees

On the other hand, we have shown that $\big(\mathsf{H}+\sigma\phi^\ep_1, \mathsf{V}+\sigma\phi^\ep_2\big)$ is a strict upper solution of \qq{4.11} provided $|\sigma|\ll 1$ in Step 1 of the proof of Lemma \ref{le4.1}. In view of $\lm({\mathscr L}_\sigma)>0$, by repeating the arguments to those in Theorem \ref{th3.6}(1), we have
 \bes
\lim_{n\to+\yy}(U_\sigma(x,t+nT), Z_\sigma(x,t+nT))=(\widehat{\mathsf{H}}_i^\sigma(x,t),
\widehat{\mathsf{V}}_i^\sigma(x,t))=(\mathsf{H}_i^\sigma(x,t), \mathsf{V}_i^\sigma(x,t))
 \lbl{4.23}\ees
uniformly in $\ol Q_T$.
This combines with \qq{4.24} gives
\bes
 \limsup_{n\to+\yy}\big(H_i(x,t+nT), V_i(x,t+nT)\big)\le (\mathsf{H}_i^\sigma(x,t), \mathsf{V}_i^\sigma(x,t))\;\;\;\mbox{uniformly in}\;\;\ol Q_T.
  \lbl{4.25}\ees
Noticing that the continuous positive solution $(\mathsf{H}_i^\sigma, \mathsf{V}_i^\sigma)$ of \qq{4.12} exists and is unique. From the expression \qq{4.9} of $\ell_{kl}^\sigma$, it is easy to see that $\ell_{kl}^\sigma$ is increasing in $\sigma$. Then $(\mathsf{H}_i^\sigma, \mathsf{V}_i^\sigma)\ge(\mathsf{H}_i, \mathsf{V}_i)$ and $(\mathsf{H}_i^\sigma, \mathsf{V}_i^\sigma)$ is increasing in $\sigma$ by the comparison principle and the uniqueness. Therefore,  $\lim_{\sigma\to0^+}(\mathsf{H}_i^\sigma, \mathsf{V}_i^\sigma)=(\mathsf{H}^*_i, \mathsf{V}^*_i)$ exists, $(\mathsf{H}^*_i, \mathsf{V}^*_i)\le (\mathsf{H}, \mathsf{V})$  and $(\mathsf{H}^*_i, \mathsf{V}^*_i)$ is a positive solution of the integral system of \qq{4.7}. It is easy to verify that $(\mathsf{H}^*_i, \mathsf{V}^*_i)$ belongs to $\mathbb{Z}^2$, and is also  a positive solution of \qq{4.7} and \qq{4.7x}.

Now we prove that $\mathsf{H}^*_i<\mathsf{H}$ and $\mathsf{V}^*_i<\mathsf{V}$. Let
$\mathsf{Z}=\mathsf{H}-\mathsf{H}^*_i$. Then $\mathsf{Z}\ge 0$ and satisfies
 \bess\begin{aligned}
 &\mathsf{Z}_t=d_1\!\int_\oo\!J_1(x,y)\mathsf{Z}(y,t)\dy
 +\dd\kk(a_1-b_1-d_1^*-c_1\mathsf{H}-\mu_1\frac{\mathsf{V}^*_i}{\mathsf{H}}\rr)
 \mathsf{Z}+a_1(x,t)\mathsf{H}^*_i\\[1mm]
 &\hspace{4mm}>d_1\!\int_\oo\!J_1(x,y)\mathsf{Z}(y,t)\dy
 \dd\kk(a_1-b_1-d_1^*-c_1\mathsf{H}-\mu_1\frac{\mathsf{V}^*_i}{\mathsf{H}}\rr)
 \mathsf{Z},\;\;(x,t)\in\ol Q_T,\\[1mm]
 &\mathsf{Z}(x,0)=\mathsf{Z}(x,T),\;\;\;\;x\in\boo
 \end{aligned}\eess
by $(4.4_1)$ and the first equation of \qq{4.7}. According to Lemma \ref{le3.3}, we conclude $\mathsf{Z}>0$ in $\ol Q_T$ and $\inf_{Q_T}\mathsf{Z}=:\kappa_1>0$, i.e.,
  \[\mathsf{H}^*_i\le\mathsf{H}-\kappa_1\;\;\;\mbox{in }\;\ol Q_T.\]
Similarly,
 \[\mathsf{V}^*_i\le\mathsf{V}-\kappa_2\;\;\;\mbox{in }\;\ol Q_T\]
for some $\kappa_2>0$.

It is easy to see that, in the range of $(0,0)\le(\mathsf{H}_i, \mathsf{V}_i)\le(\mathsf{H}-\kappa_1, \mathsf{V}-\kappa_2)$, the system \qq{4.7} is cooperative and the corresponding conditions {\bf(F2)} and {\bf(F3)} hold.  Notice that $(\mathsf{H}^*_i, \mathsf{V}^*_i) \in \mathbb{Z}^2$ is a positive solution of \qq{4.7}. By Theorem \ref{th3.6} and Remark \ref{r3.2}  (or by repeating the arguments to those in Theorem \ref{th3.4}), we have $(\mathsf{H}^*_i, \mathsf{V}^*_i)=(\mathsf{H}_i, \mathsf{V}_i)$. Hence, $\lim_{\sigma\to0^+}({\mathsf{H}_i^\sigma}, {\mathsf{V}_i^\sigma})=(\mathsf{H}_i, \mathsf{V}_i)$ uniformly in $\ol Q_T$. This combined with \qq{4.25} yields
\bes
 \limsup_{n\to+\yy}(H_i(x,t+nT), V_i(x,t+nT))\le(\mathsf{H}_i(x,t), \mathsf{V}_i(x,t)) \;\;\;{\rm uniformly\; in}\;\;\ol Q_T.
\lbl{4.26}\ees

{\it Step 2}. Noticing that $\mathsf{H}_i<\mathsf{H}$ and $\mathsf{V}_i<\mathsf{V}$ in $\ol Q_T$. There exists $0<\tau_0<\sigma_0$ such that
 \bes
 \mathsf{H}_i<\mathsf{H}-2\tau\phi^\ep_1,\; \;\; \mathsf{V}_i<\mathsf{V}-2\tau\phi^\ep_2\;\;\;\mbox{in}\;\;\ol Q_T
   \lbl{4.27}\ees
for all $0<\tau<\tau_0$. Thanks to \qq{4.26} and \qq{4.27}, there exists $N^*_\tau\gg1$ such that
  \bes
  \begin{cases}
  	H_i(x,t+n T)<\mathsf{H}_i(x,t)+\tau \phi^\ep_1(x,t)
  	<\mathsf{H}(x,t)-\tau\phi^\ep_1(x,t), \;\;(x,t)\in\ol Q_T,\;n\ge N^*_\tau,\\
  	V_i(x,t+n T)<\mathsf{V}_i(x,t)+\tau \phi^\ep_2(x,t)
  	<\mathsf{V}(x,t)-\tau\phi^\ep_2(x,t), \;\;(x,t)\in\ol Q_T,\;n\ge N^*_\tau.
  	  \end{cases}\label{equ:HV:tau}
  	\ees

{\it Step 3}. For such $\tau$ determined in Step 2. Thanks to \qq{4.19}, there exists $\hat N_\tau\gg 1$ such that
 \bess\left\{\begin{aligned}
&0<\mathsf{H}(x,t)-\tau\phi^\ep_1(x,t)\le H(x,t+nT)\le \mathsf{H}(x,t)+\tau\phi^\ep_1(x,t),\!&&(x,t)\in\ol Q_T,\; n\ge\hat N_\tau,\\
&0<\mathsf{V}(x,t)-\tau\phi^\ep_2(x,t)\le V(x,t+nT)\le \mathsf{V}(x,t)+\tau\phi^\ep_2(x,t),\!&&(x,t)\in\ol Q_T,\; n\ge\hat N_\tau.
  \end{aligned}\rr.\eess
Noticing that $H_u=H-H_i$ and $V_u=V-V_i$. Let $N_\tau=N_\tau^*+\hat N_\tau$. Take advantage of \qq{equ:HV:tau}, similar to the derivation of \qq{4.21} it can be deduced that $\big(H_i, V_i\big)$ satisfies
 \bess\kk\{\begin{aligned}
&H_{it}\ge d_1\!\dd\int_\oo\! J_1(x,y)H_i(y,t)\dy-d^*_1(x) H_i+ \mu_1(x,t)\dd\frac{(\mathsf{H}-\tau\phi^\ep_1-H_i)^+}{\mathsf{H}+\tau\phi^\ep_1}V_i\\
 &\hspace{16mm}-\big[b_1(x,t)+c_1(x,t)(\mathsf{H}
 +\tau\phi^\ep_1)\big]H_i\\
&\hspace{4.5mm}=d_1\!\dd\int_\oo\! J_1(x,y)H_i(y,t)\dy+\ell_{11}^{-\tau}(x,t) H_i+ \mu_1(x,t)\dd\frac{(\mathsf{H}-\tau\phi^\ep_1-H_i)^+}{\mathsf{H}+\tau\phi^\ep_1}V_i,\!&&x\in\boo,\, t\ge nT,\\
&V_{it}\ge d_2\!\dd\int_\oo\! J_2(x,y)V_i(y,t)\dy-d^*_2(x)V_i+ \mu_2(x,t)\dd\frac{(\mathsf{V}-\tau\phi^\ep_2-V_i)^+}{\mathsf{H}+\tau\phi^\ep_1}H_i\\
&\hspace{16mm}-\big[b_2(x,t)+c_2(x,t)(\mathsf{V}+\tau\phi^\ep_2)\big]V_i\\
&\hspace{4.5mm}=d_2\!\dd\int_\oo\! J_2(x,y)V_i(y,t)\dy+\ell_{22}^{-\tau}(x,t)V_i+ \mu_2(x,t)\dd\frac{(\mathsf{V}-\tau\phi^\ep_2-V_i)^+}{\mathsf{H}+\tau\phi^\ep_1}H_i, \!&&x\in\boo,\, t\ge nT
 \end{aligned}\rr.\eess
for all $n\ge N_\tau$. Let $(U_{-\tau},Z_{-\tau})$ be the solution of \qq{4.22} with $\sigma=-\tau$ and $(U_0(x,t),Z_0(x,t))=\big(H_i(x,t+N_\tau T), V_i(x,t+N_\tau T)\big)$. Then we have
 \bes
 (H_i(x,t+N_\tau T),V_i(x,t+N_\tau T))\ge(U_{-\tau}(x,t), Z_{-\tau}(x,t))\;\;\;\mbox{in}\;\; \boo\times[0,+\yy)
 \lbl{4.31}\ees
by the comparison principle. As we have known that $\lm({\mathscr L}_\sigma)>0$, the corresponding limit \qq{4.23} holds by repeating the arguments to those in Theorem \ref{th3.6}(1). This fact combines with \qq{4.31} yields
 \bes
 \liminf_{n\to+\yy}\big(H_i(x,t+nT), V_i(x,t+nT)\big)\ge (\widehat{\mathsf{H}}_i^{-\tau}(x,t),
 \widehat{\mathsf{V}}_i^{-\tau}(x,t))=(\mathsf{H}_i^{-\tau}(x,t), \mathsf{V}_i^{-\tau}(x,t))
 \lbl{4.32}\ees
uniformly in $\ol Q_T$. Similar to the arguments in Step 1, we have   $\lim_{\tau\to 0^+}(\mathsf{H}_i^{-\tau}, \mathsf{V}_i^{-\tau})=(\mathsf{H}_i, \mathsf{V}_i\big)$ uniformly in $\ol Q_T$. Hence, by using of \qq{4.32},
  \bess
 \liminf_{n\to+\yy}\big(H_i(x,t+nT), V_i(x,t+nT)\big)\ge \big(\mathsf{H}_i(x,t), \mathsf{V}_i(x,t)\big) \;\;\;{\rm uniformly\; in}\; \;\ol Q_T.
\eess
This, together with \qq{4.26}, yields that
 \[\lim_{n\to+\yy}\big(H_i(x,t+nT),V_i(x,t+nT)\big)=
 \big(\mathsf{H}_i(x,t),\mathsf{V}_i(x,t)\big) \;\;\;{\rm uniformly\; in}\; \;\ol Q_T.\]
Using \qq{4.19} we conclude that \qq{4.17} holds. The proof of conclusion (1) is complete.

(2)\, Assume that $\lm({\mathscr G}_1)>0$, $\lm({\mathscr G}_2)>0$ and $\lm({\mathscr L})\le 0$. In Step 2 of the proof of Lemma \ref{le4.1} we have shown that $\big(\mathsf{H}+\sigma\phi^\ep_1, \mathsf{V}+\sigma\phi^\ep_2\big)$ is a strict upper solution of \qq{4.11}.

{\it Case 1: $\lm({\mathscr L})<0$}. Then there exists a $0<\sigma_0\ll 1$ such that $\lm({\mathscr L}_\sigma)<0$ for all $0\le\sigma\le\sigma_0$ by the continuity. Noticing that \qq{4.20} and \qq{4.21} always hold. Let $(U_\sigma(x,t),Z_\sigma(x,t)\big)$ be the unique positive solution of \qq{4.22} with $(U_0(x,t), Z_0(x,t))=\big(\mathsf{H}(x,t)+\sigma\phi^\ep_1(x,t), \mathsf{V}(x,t)+\sigma\phi^\ep_2(x,t)\big)$. Then \qq{4.24} holds. Moreover,
 \[\lim_{t\to+\yy}(U_\sigma(x,t),Z_\sigma(x,t)\big)=(0,0) \;\;\text{ uniformly in } \; \overline{\Omega}\]
by Theorem \ref{th3.6}(2) and Remark \ref{r3.2}. This combines with \qq{4.24} derives
 \bes
 \lim_{t\to+\yy}(H_i(x,t), V_i(x,t))=(0, 0) \;\;\;{\rm uniformly\; in}\; \;\boo.
 \lbl{4.33}\ees
Therefore, \qq{4.18} holds by \qq{4.19}.

{\it Case 2: $\lm({\mathscr L})=0$}. Then there exists a $0<\sigma_0\ll 1$ such that $\lm({\mathscr L}_\sigma)>0$ for all $0<\sigma\le\sigma_0$ by the continuity. Then \qq{4.12} has a unique continuous positive solution $\big(\mathsf{H}_i^\sigma, \mathsf{V}_i^\sigma\big)$ and $\big(\mathsf{H}_i^\sigma, \mathsf{V}_i^\sigma\big)< \big(\mathsf{H}+\sigma\phi^\ep_1, \mathsf{V}+\sigma\phi^\ep_2\big)$. Similar to the proof of \qq{4.26} we can show that
\bes
 \limsup_{n\to+\yy}\big(H_i(x,t+nT), V_i(x,t+nT)\big)\le \big(\mathsf{H}_i^\sigma(x,t), \mathsf{V}_i^\sigma(x,t)\big) \;\;\;{\rm uniformly\; in}\; \;\ol Q_T.
\lbl{4.34}\ees

Since $\ell_{kl}^\sigma$ is increasing in $\sigma$, it is easy to see that $\big(\mathsf{H}_i^{\sigma'}, \mathsf{V}_i^{\sigma'}\big)$ is an strict upper solution of \qq{4.12} when $\sigma'>\sigma$. So, $\big(\mathsf{H}_i^\sigma, \mathsf{V}_i^\sigma\big)$ is strict increasing in $\sigma>0$, and then the limit
 \bes\lim_{\sigma\to 0^+}\big(\mathsf{H}_i^\sigma, \mathsf{V}_i^\sigma\big)=\big(\mathsf{H}_i^*, \mathsf{V}_i^*\big)\lbl{4.35}\ees
exists and is a nonnegative solution of \qq{4.7}. Similar to the discussions in Step 1 of conclusion (1), we can show that $\mathsf{H}^*_i\le\mathsf{H}-\kappa$ and  $\mathsf{V}^*_i\le\mathsf{V}-\kappa$ in $\ol Q_T$ for some $\kappa>0$. According to Theorem \ref{th3.2}, either $(\mathsf{H}_i^*,\,\mathsf{V}_i^*)\equiv 0$, or $\mathsf{H}_i^*, \mathsf{V}_i^*>0$ in $\ol Q_T$ and $\inf_{Q_T}\mathsf{H}_i^*>0$ and $\inf_{Q_T}\mathsf{V}_i^*>0$. Notice that \qq{4.7} is a cooperation system and the corresponding conditions {\bf(F2)} and {\bf(F3)} hold in the range of $(0,0)\le(\mathsf{H}_i, \mathsf{V}_i)\le(\mathsf{H}-\kappa, \mathsf{V}-\kappa)$.
By Theorem \ref{th3.6} and Remark \ref{r3.2} again, there holds $(\mathsf{H}_i^*,\,\mathsf{V}_i^*)\equiv 0$ in $\ol Q_T$. This combines with \qq{4.34} and \qq{4.35} implies that \qq{4.33} holds. The proof is complete.
 \end{proof}

\section{Discussions}

In this paper, we studied systems of nonlocal operators with cooperative and irreducible structure. By constructing the monotonic upper and lower control systems, we obtained the generalized principal eigenvalue and give two applications. Through these two examples, we saw that this generalized principal eigenvalue plays the same role as the usual principal eigenvalue.

For the given systems of nonlocal operators with cooperative and irreducible structure, certain additional conditions should be imposed to guarantee the existence of the principal eigenvalue. However, for the system \qq{4.7}, all we know is  that $\mathsf{H}(x,t)$ and $\mathsf{V}(x,t)$ are positive solutions of $(4.4_k)$ for $k=1,2$, respectively. Without a deeper understanding of their properties, imposing additional assumptions on $\mathsf{H}(x,t)$ and $\mathsf{V}(x,t)$ to ensure the existence of the principal eigenvalue seems rather contrived. The method presented in this paper can circumvent this issue.

Certainly, our method can be applied to address the nonlocal dispersal version of the models studied in \cite{LZ21, WZhang24}. Furthermore, the dynamics of the May-Nowak model \cite{Nowak-B96, Bo97} and the Capasso-Maddalena model \cite{CM81} with nonlocal dispersal in time-periodic environments can also be analyzed using the generalized principal eigenvalue through the approximation method presented in this paper, which has been discussed in \cite{WZhang25} for the autonomous case.

We finish this section with a brief discussion on the propagation dynamics of cooperative nonlocal dispersal systems. The authors of \cite{BSS2019CPAA} established the theory of traveling wave solutions and spreading speeds of cooperative nonlocal dispersal systems in time-periodic and space-periodic environments. In their work, they assume that the system has a monostable structure, meaning that the system admits a time-periodic and space-periodic positive solution that is globally asymptotically stable. In this paper, Theorem \ref{th3.6} demonstrates that this assumption is valid when the zero solution is unstable. Additionally, to achieve the linear determinacy of the spreading speeds, they also assume the existence of the principal eigenvalue. This assumption can be removed using the perturbation method we present in Theorem \ref{t2.1}. The details of this are omitted here.

 \def\theequation{\Alph{section}.\arabic{equation}}

\end{document}